\documentclass[11pt]{article}
\usepackage[a4paper, left=2cm,right=2cm,top=2cm,bottom=2cm]{geometry}
\usepackage{authblk} 
\usepackage[T1]{fontenc}
\usepackage{mlmodern}
\usepackage{lipsum} 
\usepackage{graphicx}   
\usepackage{blindtext} 
\usepackage{amsmath,amssymb,fixmath}   
\usepackage{amsthm}
\usepackage{amsfonts}
\usepackage{empheq}
\usepackage[numbers]{natbib} 
 \usepackage{hyperref}
 \usepackage{cleveref}
\hypersetup{
	colorlinks=true,
	linkcolor={blue},
	citecolor={blue},
	urlcolor={blue},
	breaklinks=true,
	plainpages=true
}
\usepackage{nicefrac,xfrac}
\usepackage{bm}
\usepackage[normalem]{ulem}
\usepackage{tikz} 
\usetikzlibrary{positioning,arrows.meta}  
\usetikzlibrary{shapes.geometric}
\usepackage[dvipsnames]{xcolor}  
\usepackage{booktabs} 
\usepackage{multirow}
\usepackage{makecell}
\usepackage{url}
\usepackage{algorithm}
\usepackage{algpseudocode}
\graphicspath{{Pictures/}}
\usetikzlibrary{positioning}  

\definecolor{myblockfill}{RGB}{231,219,219}  
\usepackage{subcaption} 
\usepackage{pgfplots}
\pgfplotsset{compat=1.18}
\usepackage{caption}

\usetikzlibrary{calc}
\usepackage{nicefrac}

\theoremstyle{definition}

\newtheorem{theorem}{\textbf{Theorem}}[section]

\newtheorem{lemma}[theorem]{\textbf{Lemma}}

\newtheorem{proposition}{\textbf{Proposition}}
\newtheorem{remark}[theorem]{\textbf{Remark}}
\newtheorem{assumption}{\textbf{Assumption}}[section]

\newcommand{\spatialD}{\Omega}

\newcommand{\spDom}{\Omega} 
\newcommand{\stDom}{U} 

\newcommand{\xcomp}[1]{x_{#1}}
\newcommand{\xii}{\xcomp{i}}
\newcommand{\ucomp}[1]{u_{#1}}

\newcommand{\xvec}{\mvec{x}}

\newcommand{\uprime}{\uvec'}
\newcommand{\pprime}{p'}
\newcommand{\avec}{\mvec{a}}
\newcommand{\bvec}{\mvec{b}}
\newcommand{\fvec}{\mvec{f}}

\newcommand{\gvec}{\mvec{g}}

\newcommand{\uvec}{\mvec{u}}
\newcommand{\vvec}{\mvec{v}}
\newcommand{\wvec}{\mvec{w}}

\newcommand{\uvech}{\mvec{u}_h}
\newcommand{\vvech}{\mvec{v}_h}

\newcommand{\zerovec}{\mathbf{0}}
\newcommand{\ph}{p_h}

\newcommand{\resMom}{\mvec{R}_m}
\newcommand{\resCon}{\mvec{R}_c}
\newcommand{\resMomStrong}{\mvec{r}_m^h}
\newcommand{\resConStrong}{\mvec{r}_c^h}
\newcommand{\resMomStrongTilde}{\widetilde{\mvec{r}}_m^h}

\newcommand{\partialder}[2]{\frac{\partial #1}{\partial #2}}

\newcommand{\timederShort}[1]{\partial_t #1}

\newcommand{\taum}{\tau_m}
\newcommand{\tauc}{\tau_c}

\newcommand{\visco}{\nu}
\newcommand{\pder}[1]{\partialder{p}{\xii}}
\newcommand{\normalvec}{{\mvec{n}}}

\newcommand{\uvecN}{{\uvec}^n}

\newcommand{\pN}{{p}^n}

\newcommand{\fvecN}{\mvec{f}^n}
\newcommand{\gvecN}{\mvec{g}^n}
\newcommand{\hvecN}{\mvec{h}^n}

\newcommand{\uvect}{\tilde{{\uvec}}}
\newcommand{\uvectN}{\tilde{{\uvec}}^n}
\newcommand{\uvecthN}{\tilde{\mvec{u}}_h^n}

\newcommand{\uvecth}{\tilde{{\uvec}}_h}

\newcommand{\uprimet}{\tilde{{\uvec}}'}
\newcommand{\uvecSN}{\hat{{\uvec}}^n}
\newcommand{\uvecSh}{\hat{{\uvec}}_h}

\newcommand{\pstar}{{p}^{*}}

\newcommand{\phat}{\hat{p}}
\newcommand{\phatN}{\hat{p}^{\,n}}
\newcommand{\phath}{\hat{p}_h}
\newcommand{\phathN}{\hat{p}_h^{\,n}}

\newcommand{\spaceV}{\mvec{V}}
\newcommand{\spaceVd}{\mvec{V}^D}
\newcommand{\spaceVh}{{\spaceV}_h}
\newcommand{\spaceVhd}{\spaceVh^D}
\newcommand{\spaceQ}{Q}
\newcommand{\spaceQh}{{\spaceQ}_h}

\newcommand{\nsd}{d}

\newcommand{\mesh}{K^h}

\newcommand{\grad}{{{\nabla}}}
\newcommand{\divergence}{\grad\cdot}
\newcommand{\ddiv}[1]{#1\cdot\grad}
\newcommand{\udiv}{\uvec\cdot\grad}

\newcommand{\laplacian}{\Delta}

\newcommand{\mvec}[1]{{\bm{#1}}}
\newcommand{\mmat}[1]{\uuline{\bm{#1}}}

\newcommand{\intSpace}{\int_{\spDom}}

\newcommand{\ds}{\mathrm{d}s}
\newcommand{\dx}{\mathrm{d}x}
\newcommand{\dxv}{\mathrm{d}\xvec}

\newcommand{\inner}[2]{\left( #1, #2 \right)}
\newcommand{\innerB}[2]{\left( #1, #2 \right)_{\partial \Omega}}
\newcommand{\innerh}[2]{\left( #1, #2 \right)_h}

\newcommand{\norm}[1]{\left\| #1 \right\|}

\newcommand{\vertiii}[1]{{\left\vert\kern-0.25ex\left\vert\kern-0.25ex\left\vert #1 \right\vert\kern-0.25ex\right\vert\kern-0.25ex\right\vert}}

\newcommand{\oneOver}[1]{\frac{1}{#1}}

\newcommand{\halfnice}{\nicefrac{1}{2}}

\newcommand{\dt}{\Delta t}

\newcommand{\cinv}{C_{I}}

\newcommand{\colref}[2]{\hyperref[#2]{#1~\ref*{#2}}}
\newcommand{\secref}[1]{\colref{Section}{#1}}
\newcommand{\figref}[1]{\colref{Figure}{#1}}
\newcommand{\tabref}[1]{\colref{Table}{#1}}

\providecommand{\phat}{\hat{p}}                  
\providecommand{\Nop}[1]{\mathcal{N}(#1)}        

\newcommand{\logLogSlopeTriangle}[5]
{
	
	\pgfplotsextra
	{
		\pgfkeysgetvalue{/pgfplots/xmin}{\xmin}
		\pgfkeysgetvalue{/pgfplots/xmax}{\xmax}
		\pgfkeysgetvalue{/pgfplots/ymin}{\ymin}
		\pgfkeysgetvalue{/pgfplots/ymax}{\ymax}
		
		\pgfmathsetmacro{\xArel}{#1}
		\pgfmathsetmacro{\yArel}{#3}
		\pgfmathsetmacro{\xBrel}{#1-#2}
		\pgfmathsetmacro{\yBrel}{\yArel}
		\pgfmathsetmacro{\xCrel}{\xArel}
		
		\pgfmathsetmacro{\lnxB}{\xmin*(1-(#1-#2))+\xmax*(#1-#2)} 
		\pgfmathsetmacro{\lnxA}{\xmin*(1-#1)+\xmax*#1} 
		\pgfmathsetmacro{\lnyA}{\ymin*(1-#3)+\ymax*#3} 
		\pgfmathsetmacro{\lnyC}{\lnyA+#4*(\lnxA-\lnxB)}
		\pgfmathsetmacro{\yCrel}{\lnyC-\ymin)/(\ymax-\ymin)} 
		
		\coordinate (A) at (rel axis cs:\xArel,\yArel);
		\coordinate (B) at (rel axis cs:\xBrel,\yBrel);
		\coordinate (C) at (rel axis cs:\xCrel,\yCrel);
		
		\draw[#5]   (A)-- node[pos=0.5,anchor=north] {1}
		(B)-- 
		(C)-- node[pos=0.5,anchor=west] {#4}
		cycle;
	}
}

\newcommand{\stab}{\text{s}}

\newcommand{\AlgStage}[1]{%
	\Statex
	\State \hspace*{-\algorithmicindent}%
	\textbf{\normalsize #1}%
}
\usepgfplotslibrary{fillbetween}
\usepackage{subcaption}
\usepackage{overpic}  
\title{\textbf{A Helmholtz-Leray projection method with variational multiscale stabilization for the Navier-Stokes equations}}

\author[1,*,$\dagger$]{Biswajit Khara}
\author[1,*]{Suresh Murugaiyan}
\author[3]{Makrand Khanwale}
\author[1,2,$\dagger$]{Baskar Ganapathysubramanian}

\affil[1]{Translational AI Center, Iowa State University, Ames, IA, 50011 USA}
\affil[2]{Department of Mechanical Engineering, Iowa State University, Ames, IA, 50011 USA}
\affil[3]{Technical University of Munich, School of Engineering and Design, Professorship for Computational Fluid Dynamics, Boltzmannstr. 15, Garching, 85748, Germany}
\affil[*]{These authors contributed equally to this work.}
\affil[$\dagger$]{Corresponding authors: \href{mailto:bkhara@iastate.edu}{bkhara@iastate.edu}; \href{mailto:baskarg@iastate.edu}{baskarg@iastate.edu}}

\date{}
\begin{document}

\maketitle

\begin{abstract}

The Galerkin finite element formulation of the incompressible Navier-Stokes equations presents two principal challenges: maintaining stable velocity-pressure coupling and controlling instability in advection-dominated regimes. Moreover, the monolithic formulation produces a coupled nonlinear saddle-point problem. In this work, we present a residual-based variational multiscale (VMS) stabilization of an incremental Helmholtz-Leray projection method that replaces this coupled saddle-point problem with a nonlinear velocity predictor, a pressure Poisson equation, and a velocity projection. The multiscale decomposition is applied only to the predicted velocity; neither the pressure nor the corrected, weakly divergence-free velocity is decomposed into coarse and fine scales. The modeled velocity fine scale contributes consistently to all three subproblems, introducing SUPG-like stabilization in the momentum predictor and a PSPG-like residual contribution in the pressure Poisson equation. We provide a formal error decomposition that separates the BDF2 time-discretization, projection-splitting, and spatial-VMS errors and, under stated stability and spatial-approximation assumptions, yields a combined velocity error estimate with second-order temporal accuracy. Numerical results for manufactured solutions, lid-driven cavity flow, flow past a cylinder, and the Taylor-Green vortex agree closely with established reference data. Comparisons with monolithic VMS indicate that omitting the pressure fine scale reduces drag overprediction and excess modeled dissipation, at the cost of increased divergence error. In the Taylor-Green tests, the projection formulation also reduces the average solution time per step by factors ranging from approximately $1.3\times$ to $2.7\times$ under identical solver settings.
\end{abstract}

\section{Introduction}\label{sec:intro}

It is well known that a Galerkin formulation of the Navier-Stokes equations (NSE) can lose stability primarily in two ways: (i) violation of the \emph{inf-sup} condition, and (ii) a dominant advective term in the momentum equations \cite{babuvska1973finite, brezzi1974existence, donea2003finite, brezzi2012mixed, john2016finite}. The first is associated with the saddle-point structure of the velocity-pressure formulation: the pressure acts as a Lagrange multiplier enforcing incompressibility, and stable discretizations must satisfy the Ladyzhenskaya-Babu\v{s}ka-Brezzi (LBB) condition or the inf-sup condition. This requirement excludes convenient equal-order approximation of velocity and pressure unless additional stabilization is introduced. The second difficulty arises from the advective term, which can make the system unstable at moderate to high Reynolds numbers if the mesh is not resolved enough to capture the small features of the flow. This leads to spurious oscillations and loss of robustness in the standard Galerkin method.

Stabilized finite element methods address these difficulties by augmenting the Galerkin formulation with additional terms. A large class of such methods has been developed, including streamline-upwind/Petrov-Galerkin (SUPG) \cite{brooks1982streamline}, pressure-stabilized/Petrov-Galerkin (PSPG) \cite{hughes1986new}, Galerkin/least-squares (GLS) \cite{hughes1989new}, the Douglas-Wang method \cite{douglas1989absolutely}, subgrid-scales \cite{codina2000stabilization}, among others. In this work, our focus will be on the variational multiscale method (VMS) \cite{hughes1995multiscale}. As the name suggests, VMS formulates the trial and test functions as direct sums of coarse and fine-scale quantities. The coarse-scale quantities are represented directly by the finite element space. The fine scales are unresolved by the mesh, but their \emph{effect} on the NSE is still incorporated in the discrete equations by \emph{modeling} the fine scales using the residuals of the coarse scales. Most importantly, the VMS methodology results in both advective and pressure stability in the NSE \cite{codina2002stabilized, bazilevs2007variational}. Finally, due to its multiscale interpretation of stabilization, VMS can also be considered as a turbulence model \cite{hughes2000large}.

Many VMS formulations are written and solved in a fully coupled, monolithic form, in which the velocity and pressure unknowns are solved simultaneously \cite{john2005finite, bazilevs2007variational}. Such formulations lead to nonlinear algebraic systems with a saddle-point structure that are more expensive to solve. This motivates the development of splitting strategies that retain the stabilizing advantages of VMS while reducing the complexity of the monolithic velocity-pressure formulation. To this end, we build on the Chorin-Temam-type splitting \cite{chorin1968numerical,temam1969sur} of the NSE, which transforms the saddle-point system into a series of elliptic problems. The Chorin-Temam splitting relies on the Helmholtz-Leray (HL) projection \cite{foias2001navier,guermond2006overview} to enforce the incompressibility constraint, which is why these schemes are also known as ``projection schemes.'' In a general projection scheme, the momentum equation is solved without the continuity constraint to obtain an intermediate velocity, which is then used to calculate the pressure. Finally, a divergence-free velocity is obtained using Helmholtz decomposition and the associated Leray projection.

The original scheme of Chorin uses a first-order time discretization and drops the pressure from the momentum predictor completely, which results in a $O(\dt)$ error in velocity and a $O(\dt^{\halfnice})$ error in pressure (where $\dt$ is the stepsize). But since then, projection methods have evolved into a broad family of fractional-step schemes that differ chiefly in how the pressure is treated in the momentum predictor. Prominent among these are the \textit{incremental pressure correction} methods \cite{goda1979multistep, vankan1986second, timmermans1996approximate}. These schemes use the pressure from the previous step in the momentum predictor, and then compute a correction to the pressure using the predicted velocity. These schemes generally improve the pressure accuracy as compared to that obtained in the original Chorin-Temam splitting. Systematic analyses and classifications of these variants are surveyed in \cite{guermond2006overview, quarteroni2000factorization}.

In this work, we combine these two prominent ideas, namely, projection scheme and multiscale stabilization, to obtain a more efficient method. We first apply the projection splitting to the time-discrete NSE, and then stabilize the split system of equations using VMS. The main contributions of this paper are therefore threefold. First, we derive a VMS-stabilized incremental projection method in which the fine scales of the predicted velocity, modeled through the residual-based multiscale framework, supply both advective and pressure stability, thereby enabling equal-order velocity-pressure interpolation without the need to satisfy the inf-sup condition. Second, we carry out a \emph{formal} error analysis of the resulting scheme and provide a systematic comparison against the monolithic VMS formulation, clarifying the accuracy and stability implications of the splitting. Third, we assess the method on a set of benchmark problems, demonstrating its accuracy, its robustness at moderate-to-high Reynolds numbers, and its computational advantages over the fully coupled formulation. The benchmark problems include the lid-driven cavity flow (2D), the flow past a cylinder (2D) and the Taylor-Green vortex (3D).


The remainder of this paper is organized as follows. In \secref{sec:formulations}, we first review the preliminaries (\secref{sec:implicit-galerkin-main} and \secref{sec:vms-stabilization-overview}), and then present the VMS-stabilized Helmholtz-Leray projection method for the incompressible NSE (\secref{sec:vms-projection}). In \secref{sec:analysis}, we provide a formal error analysis of the approach. Next, in \secref{sec:results}, we provide a series of numerical examples, and finally draw our conclusions in \secref{sec:conclusions}.

\section{Formulations}\label{sec:formulations}
In this section, we give a brief overview of the fully implicit VMS formulation of the NSE, and then present the proposed VMS-stabilized projection method. We begin with the notations and terminologies.

\subsection{Notation}\label{sec:notation}
\begin{table}[h]
	\centering
	\caption{Notation for the velocity fields used in the monolithic and projection formulations. All quantities are at time-step $n$.}
	\label{tab:velocity-notation}
	\begin{tabular}{r|l|l|l}
		\toprule[2pt]
		Symbol & Description & Note & Formulation \\
		\midrule
		$\uvec$
		& Exact solution of the NSE
		& --
		& -- \\
		
		\midrule
		
		$\uvecN$
		& Total velocity solution
		& \multirow{4}{*}{$
			\begin{aligned}
				\uvecN & = \uvec_c + \uvec_f \\
				& \approx \uvech + \uprime
			\end{aligned}
			$}
		& \multirow{4}{*}{VMS in Monolithic} \\
		
		$\uvec_c$, $\uvec_f$
		& Coarse and fine-scale components of $\uvecN$
		&
		& \\
		
		$\uvech$ 
		& Coarse-scale \emph{approximation}
		&
		& \\
		
		$\uprime$ 
		& Fine-scale \emph{approximation}
		&
		& \\
		
		\midrule
		
		$\uvectN$
		& Total predicted velocity
		& \multirow{5}{*}{$
			\begin{aligned}
				\uvectN & = \uvect_c + \uvect_f \\
				& \approx \uvect + \uprimet
			\end{aligned}
			$}
		& \multirow{5}{*}{VMS in HL projection} \\
		
		$\uvect_c$, $\uvect_f$
		& Coarse and fine-scale components of $\uvectN$
		&
		& \\
		
		$\uvecth$ 
		& Coarse-scale \emph{approximation}
		&
		& \\
		
		$\uprimet$ 
		& Fine-scale \emph{approximation}
		&
		& \\
		
		$\uvecSN$, $\uvecSh$
		& Corrected velocity after the projection step
		&
		& \\
		\bottomrule[2pt]
	\end{tabular}
\end{table}
In the next subsection and thereafter, we introduce several velocity fields associated with two stabilized formulations of the NSE. The notation is summarized in \tabref{tab:velocity-notation} for quick reference, and each quantity is explained when it is first encountered in the text.

Among other frequently used notations, we denote vector-valued quantities by boldface letters. For example, $\xvec$ given by
\begin{align*}
	\xvec = {\left[ \xcomp{1}, \xcomp{2}, \ldots, \xcomp{d} \right]}^\top =
	\begin{bmatrix}
		\xcomp{1} \\
		\xcomp{2} \\
		\vdots \\
		\xcomp{d}
	\end{bmatrix}
\end{align*}
denotes the spatial coordinates where each $\xcomp{i}$, $i = 1, 2,\ldots, d$ are scalars. Similarly,
$\uvec(t,\xvec) = {\left[ u_1, u_2, \ldots, u_n \right]}^\top $ denotes a vector-valued \emph{function} where each ${\{\ucomp{i}\}}_{i=1}^{n}$ is a function of both time and space, i.e., $\ucomp{i} = \ucomp{i}(t,\xvec)$.

Given two scalar-valued functions $a(\xvec)$ and $b(\xvec)$ on a domain $\spDom \subset \mathbb{R}^d$, we denote the inner-product between them as
\begin{align*}
	\inner{a}{b} := \intSpace a(\xvec)\ b(\xvec) \mbox{d}\xvec.
\end{align*}
For two \emph{vector-valued} functions $\avec(\xvec)$ and $\bvec(\xvec)$, we denote the inner-product in a similar manner as
\begin{align*}
	\inner{\avec}{\bvec} := \intSpace \avec(\xvec) \cdot \bvec(\xvec) \mbox{d}\xvec =\intSpace {\avec(\xvec)}^\top\bvec(\xvec) \mbox{d}\xvec = \intSpace \left[ \sum_{i = 1}^{n} a_i(\xvec)\ b_i(\xvec) \right]\mbox{d}\xvec = \sum_{i = 1}^{n} \inner{a_i}{b_i}.
\end{align*}

For two vector-valued functions $\avec$ and $\bvec$, we also define the outer product $\bm{C}$ as
\begin{align*}
	\bm{C} = \avec \otimes \bvec = \avec \bvec^{\top}.
\end{align*}
Here, $\bm{C}$ is a second order tensor whose $(i,j)$-th component is given as
\begin{align*}
	C_{ij} = a_i b_j.
\end{align*}
For matrices, we will use a double underbar, e.g., $\mmat{C}$.

\subsection{Implicit Galerkin formulation of the Navier-Stokes equations}
\label{sec:implicit-galerkin-main}

\begin{figure}[t]
	\centering
	\begin{tikzpicture}[line cap=round, line join=round]
		
		\coordinate (P1) at (-2.2,  0.0);   
		\coordinate (P2) at ( 1.8, -0.3);   
		
		\fill[gray!12]
		(P1)
		.. controls (-2.0, -0.8) and (-0.8, -2.0) .. ( 0.2, -1.9)
		.. controls ( 1.2, -1.8) and ( 1.0, -0.7) .. (P2)
		.. controls ( 2.6,  0.1) and ( 1.5,  1.8) .. ( 0.0,  2.2)
		.. controls (-1.5,  2.6) and (-2.4,  0.8) .. (P1);
		
		\draw[very thick]
		(P1)
		.. controls (-2.0, -0.8) and (-0.8, -2.0) .. ( 0.2, -1.9)
		.. controls ( 1.2, -1.8) and ( 1.0, -0.7) .. (P2);
		
		\draw[very thick, dash pattern=on 7pt off 3.5pt]
		(P2)
		.. controls ( 2.6,  0.1) and ( 1.5,  1.8) .. ( 0.0,  2.2)
		.. controls (-1.5,  2.6) and (-2.4,  0.8) .. (P1);
		
		\fill (P1) circle (2.5pt);
		\fill (P2) circle (2.5pt);
		
		\node[font=\large] at ( 0.0,  0.3) {$\Omega$};
		\node[font=\large] at ( 0.0, -2.25) {$\Gamma_D$};
		\node[font=\large] at (-1.5,  2.8) {$\Gamma_N$};
		
		\coordinate (NP) at (0.0, 2.2);
		\draw[-{Stealth[length=9pt, width=5.5pt]}, thick]
		(NP) -- ++(0.20, 0.75);
		\node[font=\normalsize] at ($(NP)+(0.46, 0.82)$) {$\hat{n}$};
		
	\end{tikzpicture}
	\caption{Sketch of a general domain with different conditions applied at different boundaries. Neumann conditions are prescribed on $\Gamma_N$, and Dirichlet conditions on the velocities are prescribed on $\Gamma_{D}$.}\label{fig:schematic}
\end{figure}
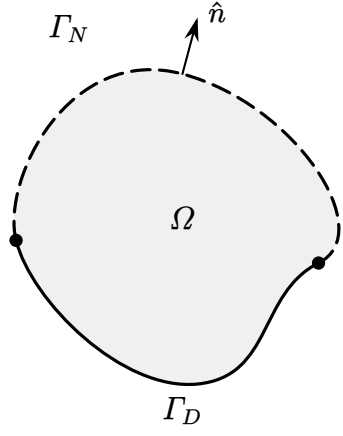

Suppose $\spDom$ is a spatial domain (see \figref{fig:schematic}). The incompressible Navier-Stokes equations in non-dimensional form on $\spDom$ are
\begin{subequations}\label{eq:nse}
	\begin{align}
		\label{eq:nse-momentum}
		\timederShort{\uvec} + (\udiv)\uvec - \visco\laplacian\uvec + \grad p - \fvec &= \mathbf{0}, \\
		\label{eq:nse-continuity}
		\divergence\uvec &= 0,
	\end{align}
\end{subequations}
along with the boundary conditions
\begin{subequations}\label{eq:nse-bc}
	\begin{align}
		\label{eq:nse-bc-d}
		\uvec &= \uvec_{D}\ \text{on}\ \Gamma_{D},\\
		\label{eq:nse-bc-n}
		\normalvec\cdot(-p\mvec{I}+\visco\grad\uvec) &= \mvec{h} \ \text{on}\ \Gamma_{N},
	\end{align}
\end{subequations}
and initial conditions
\begin{align}
	\label{eq:nse-ic}
	\uvec(t = 0, \cdot) = \uvec_0(\cdot).
\end{align}
Here, $\uvec$, $p$ and $\fvec$ are the velocity, pressure, and forcing, respectively. The parameter $\visco$ is the dimensionless viscosity ($\visco=Re^{-1}$). For unique solvability of equations~\eqref{eq:nse}--\eqref{eq:nse-ic}, we assume $\Gamma_{D}\neq\emptyset$.

\subsubsection{Discrete in time formulation}
\label{sec:semi-disc}

To numerically solve~\eqref{eq:nse}--\eqref{eq:nse-ic}, we must prescribe the temporal and spatial discretization schemes. We employ implicit time stepping (using backward difference formula (BDF) methods) and a continuous Galerkin finite element method in space. An $r$-th order BDF approximation of $\partial\uvec/\partial t$ at time step $n$ is
\begin{align}
	\partialder{\uvec}{t}\Big|_{n}
	&= \oneOver{\dt}\beta_0 \uvecN  + \oneOver{\dt}\sum_{i = 1}^{r} \beta_{-i}\uvec^{n-i} \notag \\
	&= \sigma \uvec^{n} + \oneOver{\dt}\sum_{i = 1}^{r} \beta_{-i}\uvec^{n-i},
\end{align}
where $\beta_0,\ {\{\beta_{-i}\}}_{i=1}^{r}$ are the standard coefficients of the BDF-$r$ scheme; and $\sigma=\nicefrac{\beta_0}{\dt}$. In a time-marching setting, $\uvec^n$ is unknown, while $\uvec^{n-1},\ldots,\uvec^{n-r}$ are known. The BDF-$r$ discretization of~\eqref{eq:nse} is given as
\begin{subequations}\label{eq:nse-bdf}
	\begin{align}
		\sigma \uvec^{n} + (\ddiv{\uvec^n})\uvec^n - \visco\laplacian\uvec^n + \grad p^n - \gvec^n &= \mathbf{0}, \label{eq:nse-bdf-mom} \\
		\divergence\uvec^n &= 0, \label{eq:nse-bdf-cont}
	\end{align}
\end{subequations}
where
\begin{align}
	\gvec^n := \fvec^n - \oneOver{\dt}\sum_{i=1}^{r}\beta_{-i}\uvec^{n-i}.
\end{align}

\subsubsection{Galerkin formulation}
\label{sec:implicit-galerkin-formulation}
We now derive a Galerkin formulation of~\eqref{eq:nse-bdf}. Define the trial and test spaces
\begin{subequations}\label{eq:fem-subspaces-modified}
	\begin{align}
		\spaceVd &:= \left\{ \mvec{v} \in \mvec{H}^1(\stDom): \vvec = \uvec_{D} \ \text{on}\ \Gamma_{D} \right\}, \\
		\spaceV &:= \left\{ \mvec{v} \in \mvec{H}^1(\stDom): \mvec{v} = \zerovec \text{ on } \Gamma_{D}\right\}, \\
		\spaceQ &:= {H}^1(\stDom).
	\end{align}
\end{subequations}
We seek $(\uvec^n,p^n)\in \spaceVd\times\spaceQ$ such that, for all $(\vvec,q)\in \spaceV\times\spaceQ$,
\begin{align}
	\intSpace \vvec\cdot \left[ \sigma\uvec^n + (\ddiv{\uvec^n})\uvec^n - \visco\laplacian\uvec^n + \grad p^n - \gvec^n \right] \dxv
	+ \intSpace q \ \divergence\uvec^n \ \dxv = 0.
\end{align}
Performing an integration-by-parts on the ``stress'' terms, i.e., the Laplacian $\laplacian \uvecN$ and the pressure gradient $\grad \pN$ as follows:
\begin{align}
	\intSpace \vvec \cdot \left[- \visco\laplacian\uvecN + \grad \pN\right] \dxv&= \intSpace \visco \grad \vvec : \grad \uvecN \dxv - \intSpace (\divergence\vvec)\ \pN \ \dxv - \int_{\partial \spatialD} \vvec \cdot \left(\normalvec \cdot \left[-\pN\mvec{I} + \visco\grad\uvecN\right]\right)\dxv \\
	&= \intSpace \visco \grad \vvec : \grad \uvecN \dxv - \intSpace (\divergence\vvec)\ \pN \ \dxv - \int_{\Gamma_{N}} \vvec \cdot \hvecN\ \dxv.
\end{align}
Collating all terms, and using the inner product notation described in \secref{sec:notation}, the Galerkin formulation of~\eqref{eq:nse-bdf} is to find $(\uvecN,\pN)\in\spaceVd\times\spaceQ$ such that
\begin{subequations}
	\label{eq:nse-bdf-galerkin}
	\begin{empheq}[left=\empheqlbrace]{alignat=2}
		&\resMom(\uvecN,\pN;\vvec) &&= \zerovec \quad \forall \vvec \in \spaceV, \\
		&\resCon(\uvecN;q)        &&= 0        \quad \forall q \in \spaceQ,
	\end{empheq}
\end{subequations}
where $\resMom$ and $\resCon$ are the weak momentum and continuity residuals defined as
\begin{subequations}
	\label{eq:coupled-weak-residuals}
	\begin{align}
		\resMom(\uvecN,\pN;\vvec)
		&:=
		\inner{\sigma\uvecN}{\vvec}
		+\inner{(\uvecN\cdot\grad)\uvecN}{\vvec}
		+\visco\inner{\grad\uvecN}{\grad\vvec}
		-\inner{\pN}{\divergence\vvec}
		\nonumber\\
		&\qquad
		-\inner{\gvecN}{\vvec}
		-\inner{\hvecN}{\vvec}_{\Gamma_N},
		\\
		\resCon(\uvecN;q)
		&:=
		\inner{\divergence\uvecN}{q}.
	\end{align}
\end{subequations}
$\resMom$ is nonlinear in $\uvecN$, whereas $\resCon$ is linear. 

\paragraph{Matrix structure of the coupled monolithic problem.}
The velocity $\uvecN$ and the pressure $\pN$ are coupled in the Galerkin formulation~\eqref{eq:nse-bdf-galerkin}, just like in the original NSE \eqref{eq:nse}.
To make the coupled structure explicit, we first write \eqref{eq:nse-bdf-galerkin} in its nonlinear algebraic form. Suppose
\begin{equation}
	\mvec{U} = [\uvec_1, \ldots, \uvec_{n_b}]^\top,
	\qquad \text{and}\qquad
	\mvec{P} = [p_1, \ldots, p_{n_b}]^\top
\end{equation}
are the nodal unknowns of $\uvecN$ and $\pN$ respectively, and $n_b$ denotes the number of basis functions in the finite element mesh. The nonlinear algebraic system arising from the spatial discretization of \eqref{eq:nse-bdf-galerkin} can then be written as
\begin{subequations}
	\label{eq:discrete-system-coupled-nonlinear}
	\begin{empheq}[left=\empheqlbrace]{alignat=3}
		&\mvec{N}(\mvec{U}) &{}+{}& \mmat{G}\mvec{P} &{}={}& \zerovec, \label{eq:discrete-system-coupled-nonlinear-mom} \\
		&\mmat{D}\mvec{U}  &     &                  &{}={}& \zerovec. \label{eq:discrete-system-coupled-nonlinear-con}
	\end{empheq}
\end{subequations}
Here, $\mvec{N}$ is a nonlinear (vector) function of $\mvec{U}$. The matrix $\mmat{G}$ is the discrete pressure-gradient operator, and $\mmat{D}$ is the discrete divergence operator.

A standard approach for solving~\eqref{eq:nse-bdf-galerkin} is Newton's method. The Newton system at the $k^{\textrm{th}}$-step ($k \geq 0$) for the increments $(\delta\mvec{U},\delta \mvec{P})$ can be written as
\begin{align}
	\label{eq:discrete-system-coupled-newton}
	\begin{bmatrix}
		\mmat{K}(\mvec{U}_{k}) & \mmat{G} \\
		\mmat{D} & \mmat{0}
	\end{bmatrix}
	\begin{bmatrix}
		\delta\mvec{U} \\
		\delta\mvec{P}
	\end{bmatrix}
	=
	-
	\begin{bmatrix}
		\mvec{R}_{M,k} \\
		\mvec{R}_{C,k}
	\end{bmatrix}.
\end{align}
Here, $\mmat{K} := \partialder{\mvec{N}}{\mvec{U}}$ is the tangent stiffness matrix acting on $\mvec{U}$. $\mmat{K}(\mvec{U}_{k})$ contains the transient mass contribution, the linearized advective contribution, and the viscous contribution. On the right hand side, $\mvec{R}_{M,k} := \mvec{N}(\mvec{U}_{k}) + \mmat{G}\mvec{P}_{k}$ and $ \mvec{R}_{C,k} := \mmat{D}\mvec{U}_{k} $ are the discrete momentum and continuity residuals evaluated at the $k^{\textrm{th}}$ iterates.
Once \eqref{eq:discrete-system-coupled-newton} is solved for the increments, we then have the next iterates
\begin{subequations}
	\label{eq:coupled-newton-update-weak}
	\begin{align}
		\mvec{U}_{k+1}^n &= \mvec{U}_{k}^n + \delta\mvec{U}, \\
		\mvec{P}_{k+1}^n &= \mvec{P}_{k}^n + \delta \mvec{P}.
	\end{align}
\end{subequations}
Both \eqref{eq:discrete-system-coupled-nonlinear} or \eqref{eq:discrete-system-coupled-newton} have the structure of a saddle-point problem with a zero diagonal block, which is typical of a Galerkin formulation of the NSE without any stabilization.


\subsubsection{Lack of stability of the continuous Galerkin method}
\label{sec:lack-of-stability}

Equation \eqref{eq:nse-bdf-galerkin} represents the weak form of the time-discrete incompressible Navier-Stokes equations. However, a discretization of \eqref{eq:nse-bdf-galerkin} using standard $C^0(\Omega)$ finite element basis functions is not automatically stable, and runs into two main challenges. The first is associated with the saddle-point character of the coupled velocity-pressure formulation. And the second is associated with the advective term in the momentum equation.

\paragraph{The saddle-point problem.}
The first issue arises from the incompressibility constraint and resulting saddle-point structure in \eqref{eq:discrete-system-coupled-nonlinear-con} and \eqref{eq:discrete-system-coupled-newton}. For such a system, the discrete velocity and pressure spaces cannot be chosen independently. If $\spaceV_h\subset\spaceV$ and $\spaceQ_h\subset\spaceQ$ denote the discrete velocity and pressure spaces, then the pair $(\spaceV_h,\spaceQ_h)$ must satisfy a discrete inf-sup condition of the form
\begin{align}
	\label{eq:discrete-infsup-condition}
	\inf_{q_h\in\spaceQh}
	\sup_{\vvec_h\in\spaceVh}
	\frac{\inner{\divergence\vvec_h}{q_h}}
	{\|\vvec_h\|_{\spaceVh}\|q_h\|_{\spaceQh}}
	\geq \beta_h > 0,	
\end{align}
for some $\beta_h \in \mathbb{R}$. This condition ensures that the discrete pressure is sufficiently controlled by the discrete divergence operator. If this compatibility condition is violated, the pressure approximation may develop nonphysical oscillations or checkerboard modes, and stability of the formulation cannot be guaranteed~\cite{babuvska1973finite,brezzi1974existence,john2016finite}.

A pair of finite element spaces that satisfies \eqref{eq:discrete-infsup-condition} is called an \emph{inf-sup-stable} or \emph{LBB-stable} pair~\cite{babuvska1973finite,brezzi1974existence,john2016finite}. Examples of such pairs include the Taylor--Hood pairs ($[P_m]^\nsd/P_{m-1}$ and $[Q_m]^\nsd/Q_{m-1}$, $m \geq 2$)~\cite{taylor1973numerical}, the mini element ($[P_1 \oplus B]^\nsd/P_1$)~\cite{arnold1984stable}, the Crouzeix--Raviart pair ($[P_1^{\mathrm{CR}}]^\nsd/P_0$)~\cite{crouzeix1973conforming}, and the Scott--Vogelius pairs ($[P_m]^\nsd/P_{m-1}^{\mathrm{disc}}$, for sufficiently large \(m\) under suitable assumptions on the mesh topology)~\cite{scott1985norm}.

Standard equal-order pairs, such as $[P_1]^\nsd/P_1$ or $[Q_1]^\nsd/Q_1$, are not inf-sup stable, and require additional stabilization such as the pressure stabilized Petrov Galerkin method  (PSPG)~\cite{hughes1986new, douglas1989absolutely, tezduyar1992incompressible}, the Galerkin/least-squares (GLS) method~\cite{hughes1989new}, subgrid-scale methods~\cite{codina2000stabilization} and the variational multiscale (VMS) method~\cite{hughes1995multiscale}, among others. In these stabilized schemes, the zero diagonal block of the original problem is regularized with a negative definite block matrix. In this work, we focus on the variational multiscale methods, which fall under this umbrella.

\paragraph{The advective operator.}
The second issue is associated with the advective term $(\udiv)\uvec$ in the momentum equation.
When advection dominates diffusion (i.e., the Reynolds number becomes large) or the mesh is not sufficiently fine relative to the flow features, the standard continuous Galerkin discretization of the advective term may produce spurious oscillations~\cite{brooks1982streamline,hughes1989new,codina1998comparison}.

There are several ways to resolve these difficulties, typically by introducing some form of upwind, residual-based, or fluctuation-control stabilization. Classical approaches include streamline-upwind/Petrov--Galerkin (SUPG) method~\cite{brooks1982streamline}, GLS~\cite{hughes1989new}, residual-free bubbles~\cite{brezzi1992residual,baiocchi1993stabilization}, and VMS~\cite{hughes1995multiscale}. Additional finite element strategies include discontinuous Galerkin methods with upwind numerical fluxes~\cite{cockburn2001runge}, local projection stabilization~\cite{braack2006local}, continuous interior penalty or edge-stabilization methods~\cite{douglas1976interior,burman2004edge}, and algebraic flux-correction or shock-capturing techniques for strongly underresolved flows~\cite{kuzmin2002flux,tezduyar1986discontinuity}.

\paragraph{}
As discussed above, we note that both GLS and VMS, along with various sub-grid scale theories, can provide stability against both issues. Moreover, in addition to being a method, VMS also acts as a multiscale model, and provides a general framework for formulating stabilized methods for various applications. Therefore, in the sequel, we will only focus on VMS for deriving stabilized schemes. We provide a brief description of the VMS methodology with respect to the Galerkin formulation  \eqref{eq:vms-decomposition-overview}.

\subsection{A review of the variational multiscale stabilization}
\label{sec:vms-stabilization-overview}
The variational multiscale (VMS) framework provides a unified way of deriving both advective and pressure stabilizations by decomposing the solution into multiple scales and by modeling the unresolved scales~\cite{hughes1995multiscale,hughes1989new,bazilevs2007variational}. The core idea of VMS is to write the continuous velocity and pressure fields as the sum of a \textit{coarse} and a \textit{fine} component:
\begin{subequations}
	\label{eq:vms-decomposition-overview}
	\begin{align}
		\uvecN &= \uvecN_c + \uvecN_f, \\
		\pN     &= \pN_c + \pN_f.
	\end{align}
\end{subequations}
The coarse scales represent the part of the solution that can be resolved by the discretization. And the fine scales represent features that remain unresolved. Therefore, we can approximate the coarse scales $(\uvecN_c, \pN_c)$ by the discrete functions $(\uvech,\ph)\in \spaceVh\times\spaceQh$ directly.


The fine scales $\uvecN_f$ and $\pN_f$ cannot be directly represented on the mesh. And yet, they cannot simply be discarded. Rather, their effect must be considered in the finite dimensional Galerkin formulation. The defining feature of any VMS method is precisely how this fine scale effect is approximated, and several closures have been proposed (for instance, those based on bubble functions, on the solution of local fine scale problems, or on algebraic models of the fine scale operator). Here we adopt the \textit{residual-based} VMS formulation, in which the fine scale quantities are \textit{modeled} by the residuals of the coarse scale quantities as
\begin{subequations}
	\label{eq:vms-finescale-model-overview}
	\begin{align}
		\uvecN_f &\approx \uprime := -\taum \resMomStrong, \label{eq:vms-finescale-model-overview-m} \\
		\pN_f   &\approx \pprime := -\tauc \resConStrong. \label{eq:vms-finescale-model-overview-c}
	\end{align}
\end{subequations}

where the element-wise strong-form residuals are
\begin{subequations}
	\label{eq:vms-coarse-residuals-overview}
	\begin{align}
		{\resMomStrong}
		&:=
		\sigma\uvech
		+
		(\uvech\cdot\grad)\uvech
		-
		\visco\laplacian\uvech
		+
		\grad \ph
		-
		\gvec^n, \label{eq:vms-coarse-residuals-overview-m}
		\\
		{\resConStrong}
		&:=
		\divergence\uvech, \label{eq:vms-coarse-residuals-overview-c}
	\end{align}
\end{subequations}
and $\taum$ and $\tauc$ are elementwise stabilization parameters. These parameters encode the local temporal, advective, and diffusive scales of the problem. These parameters can be calculated as~\cite{shakib1991new,bazilevs2007variational}
\begin{subequations}
	\label{eq:vms-tau-overview}
	\begin{align}
		\taum
		&=
		\left[
		\frac{4}{\dt^2}
		+\uvech\cdot\mathbf{G}\uvech
		+\cinv\visco^2(\mathbf{G}:\mathbf{G})
		\right]^{-1/2}, \label{eq:vms-tau-m}
		\\
		\tauc
		&=
		\frac{1}{\taum(\boldsymbol{g}\cdot\boldsymbol{g})}, \label{eq:vms-tau-c}
	\end{align}
\end{subequations}
where $\mathbf{G}$ and $\boldsymbol{g}$ are element metric quantities and $\cinv$ is a dimensionless constant. Their construction from the element map is given in Appendix~\ref{app:vms-parameters}.
Now, substituting the decomposition \eqref{eq:vms-decomposition-overview} in \eqref{eq:coupled-weak-residuals}, we have
\begin{subequations}
	\label{eq:coupled-vms}
	\begin{empheq}[left=\empheqlbrace]{alignat=2}
		&\resMom^{\text{VMS}}(\uvecN,\pN;\vvec) &&= \zerovec \quad \forall \vvec \in \spaceVh, \label{eq:coupled-vms-mom} \\
		&\resCon^{\text{VMS}}(\uvecN;q)        &&= 0        \quad \forall q \in \spaceQh, \label{eq:coupled-vms-con}
	\end{empheq}
\end{subequations}
where
\begin{subequations}
	\label{eq:coupled-weak-residuals-decomposed}
	\begin{align}
		\resMom^{\text{VMS}}(\uvecN,\pN;\vvec)
		&:=
		\innerh{\sigma\uvech}{\vvec}
		+\innerh{((\uvech+\uprime)\cdot\grad)(\uvech+\uprime)}{\vvec} \nonumber \\
		&\qquad+\visco\innerh{\grad\uvech}{\grad\vvec}
		-\innerh{(\ph+\pprime)}{\divergence\vvec}
		\nonumber\\
		&\qquad
		-\innerh{\gvecN}{\vvec}
		-\inner{\hvecN}{\vvec}_{\Gamma_N,h}, \label{eq:coupled-weak-residuals-mom}
		\\
		\resCon^{\text{VMS}}(\uvecN;q)
		&:=
		\innerh{\divergence(\uvech+\uprime)}{q}, \label{eq:coupled-weak-residuals-con}
	\end{align}
\end{subequations}
where all the inner products are now understood to be sums of elementwise inner products since the fine scale quantities $\uprime$ and $\pprime$ are only defined in element interiors. Equation \eqref{eq:coupled-weak-residuals-decomposed}, together with the fine scale model given by \eqref{eq:vms-finescale-model-overview} and the definitions in \eqref{eq:vms-coarse-residuals-overview} complete the description of the method.

Expanding the advective term, we have
\begin{align*}
	\innerh{(\uvech+\uprime)\cdot\grad(\uvech+\uprime)}{\vvec} &= \innerh{(\uvech+\uprime)\cdot\grad\uvech}{\vvec} + \innerh{(\uvech+\uprime)\cdot\grad\uprime}{\vvec}.
\end{align*}
To avoid a differentiation of $\uprime$ (which appears in the second inner product), we write
\begin{align*}
	\innerh{(\uvech+\uprime)\cdot\grad\uprime}{\vvec} &= \inner{\divergence\left[\uprime(\uvech+\uprime)^T\right]}{\vvech}
	- \inner{(\divergence(\uvech+\uprime))\,\uprime}{\vvech} \\
	&\approx \inner{\divergence\left[\uprime(\uvech+\uprime)^T\right]}{\vvech} \\
	&= - \inner{\uprime(\uvech+\uprime)^T}{\grad\vvech} + \innerB{\normalvec\cdot\left[\uprime(\uvech+\uprime)^T\right]}{\vvech} \\
	&= - \inner{\uprime}{\uvech\cdot\grad\vvech} - \inner{\uprime}{\uprime\cdot\grad\vvech} + \innerB{\normalvec\cdot\left[\uprime(\uvech+\uprime)^T\right]}{\vvech},
\end{align*}
where in the second step, we have assumed that the total velocity $\uvec = \uvech + \uprime$ is  divergence free; and in the third step, we have performed an integration by parts. The inner product $ - \inner{\uprime}{\uvech\cdot\grad\vvech} $ is the so-called SUPG term since it introduces the streamline derivative of the test function against the full residual,
\begin{align*}
	-\innerh{\uprime}{\uvech\cdot\grad\vvec} = -\innerh{-\taum\resMomStrong}{\uvech\cdot\grad\vvec} = \innerh{\taum\resMomStrong}{\uvech\cdot\grad\vvec}.
\end{align*}
Similarly, focusing on the continuity term, we have
\begin{align}
	\innerh{\divergence(\uvech+\uprime)}{q} &= \inner{\divergence\uvech}{q} + \innerh{\divergence\uprime}{q} \nonumber\\
	&= \inner{\divergence\uvech}{q} - \innerh{\uprime}{\grad q} + \inner{\normalvec\cdot\uprime}{q}_{\Gamma_{N},h} \quad\text{(integ.-by-parts)} \nonumber \\
	&= \inner{\divergence\uvech}{q} + \innerh{\taum\resMomStrong}{\grad q} + \inner{\normalvec\cdot\uprime}{q}_{\Gamma_{N},h}
	\label{eq:vms-discussion-pspg-term}
\end{align}
where, in the second step, we have used the assumption that $\uvecN = \uvech$ on $\Gamma_D$ and therefore $\uprime=0$. The second term in the last line, i.e., $\innerh{\taum\resMomStrong}{\grad q}$ is exactly the stabilization term applied in the PSPG method. And finally, the pressure terms are given by
\begin{align*}
	-\innerh{(\ph+\pprime)}{\divergence\vvec} &= -\inner{\ph}{\divergence\vvec} - \innerh{\pprime}{\divergence\vvec} \\
	&= -\inner{\ph}{\divergence\vvec} + \innerh{\tauc \resConStrong}{\divergence\vvec},
\end{align*}
where, we can identify the second inner product as the ``grad-div'' stabilization term. A broader discussion on both fully-implicit and semi-implicit formulations of VMS can be found in \cite{khara2025semi}.

\begin{remark}[Underlying assumptions]
	\label{rem:vms-assumptions-time-diffusion}
	The implicit assumptions in \eqref{eq:coupled-weak-residuals-mom} are:
	\begin{enumerate}
		\item In the discrete time-derivative term, we have only considered the coarse-scale $\uvech$. This follows from the underlying assumption
		\begin{align}
			\inner{\timederShort{\uprime}}{\vvec} = 0 \quad \text{or } \inner{\timederShort{\uvech + \uprime}}{\vvec} = \inner{\timederShort{\uvech}}{\vvec}.
		\end{align}
		In the discretized form, this term enters \eqref{eq:coupled-weak-residuals-mom} through both $\inner{\sigma \uvech}{\vvec}$ and $\inner{\gvecN}{\vvec}$.
		\item We also make a similar simplification in the viscous term:
		\begin{align}
			\inner{\grad \uprime}{\grad \vvec} = 0 \quad\text{or } \inner{\grad (\uvech+\uprime)}{\grad \vvec} = \inner{\grad \uvech}{\grad \vvec}.
		\end{align}
		\item We neglect the fine scales $\uprime$ on the boundaries, i.e.,
		\begin{align}
			\uprime = 0\ \text{on}\ \partial\Omega.
		\end{align}
	\end{enumerate}
	These simplifications follow the assumptions made in \cite{bazilevs2007variational}.
\end{remark}

\paragraph{Matrix equations.}
At the algebraic level, the stabilization changes the structure of the Newton system \eqref{eq:discrete-system-coupled-newton} as
\begin{align}
	\label{eq:linear-system-vms}
	\begin{bmatrix}
		\mmat{K}^{\stab} 
		&
		\mmat{G}^{\stab}
		\\
		\mmat{D}^{\stab}
		&
		-\mmat{S}^{\stab}
	\end{bmatrix}
	\begin{bmatrix}
		\delta\mvec{U} \\
		\delta\mvec{P}
	\end{bmatrix}
	=
	-
	\begin{bmatrix}
		\mvec{R}_{M}^{\stab} \\
		\mvec{R}_{C}^{\stab}
	\end{bmatrix}.
\end{align}
Here, the superscript ``${\stab}$'' stands for ``stabilized.'' Note that the zero pressure diagonal block is replaced by a discrete negative Laplacian operator $-\mmat{S}^{\stab}$. Similarly the other block matrices $\mmat{K}$, $\mmat{G}$ and $\mmat{D}$, along with the residuals are replaced by their stabilized counterparts. Note that both $ \mmat{K}^{\stab} $ and $ \mmat{S}^{\stab} $ are positive definite (accounting for appropriate boundary conditions), and therefore the system in \eqref{eq:linear-system-vms} is still a saddle-point system.

%

\paragraph{}
The preceding discussion shows that residual-based VMS stabilization provides a systematic way to regularize the fully coupled Galerkin formulation: the SUPG-like terms control advective instabilities, while the PSPG- and grad-div-like terms relax the strict inf-sup requirement and modify the algebraic saddle-point structure. However, the resulting monolithic stabilized system remains a nonlinear saddle-point problem with velocity-pressure coupling. This motivates an alternative strategy in which the stabilizing effect of VMS is retained, but the velocity-pressure coupling is relaxed by a splitting approach. We discuss this below.

\subsection{A VMS-stabilized projection scheme}
\label{sec:vms-projection}
In this section, we introduce a VMS-stabilized projection method for the incompressible Navier-Stokes equations. The core idea is to use a second order splitting method based on a Helmholtz-Leray-type projection method (inspired by the classical Chorin-Temam-type schemes), but stabilize the momentum equations using a VMS-style coarse-fine decomposition of the velocities.

%

\subsubsection{The Chorin-Temam family of projection methods}
The classical projection method of Chorin and Temam~\cite{chorin1968numerical, temam1969sur} has inspired a large family of fractional-step schemes, including the incremental~\cite{goda1979multistep} and rotational incremental pressure-correction variants~\cite{timmermans1996approximate}. In this work we begin with the incremental pressure-correction scheme due to van Kan~\cite{vankan1986second}, which retains the previous pressure $\nabla p^{n}$ in the viscous step. For the time-discrete equations \eqref{eq:nse-bdf}, this scheme can be written as
\begin{subequations}
	\label{eq:chorin-temam-generic}
	\begin{align}
	\sigma \uvectN + (\ddiv{\uvectN})\uvectN - \visco\laplacian\uvectN + \grad \pstar - \gvec^n &= \mathbf{0} \ \text{in}\ \spatialD, \\
	\uvectN &= \uvec_{D} \ \text{on}\ \Gamma_D,
	\end{align}
\end{subequations}
where $\uvectN$ is an intermediate velocity solution, and $\pstar$ is an extrapolation of pressure from the previous steps, i.e., $\pstar$ is a known quantity at time-step $n$. The solenoidality constraint is ignored  in the solution of \eqref{eq:chorin-temam-generic}, thus $\uvectN$ is not generally solenoidal. A divergence-free velocity $\uvecSN$ and the pressure $\phatN$ are then obtained from
\begin{subequations}
	\label{eq:pcvu-hp}
	\begin{align}
		\sigma (\uvecSN - \uvectN) + \grad \phatN - \grad \pstar &= 0 \ \text{in}\ \spatialD, \label{eq:pcvu-hp-1}\\
		\divergence\uvecSN &= 0 \ \text{in}\ \spatialD, \label{eq:pcvu-hp-2}\\
		\uvecSN \cdot \normalvec &= \uvectN \cdot \normalvec\ \text{on}\ \Gamma_D. \label{eq:pcvu-hp-3}
	\end{align}
\end{subequations}
Eq.~\eqref{eq:pcvu-hp} is still a coupled set of equations between velocity and pressure. However, they can be split further into two new subproblems. Applying $-\grad\cdot$ to Eq~\eqref{eq:pcvu-hp-1},
\begin{subequations}
	\label{eq:ppe}
	\begin{align}
		-\laplacian \phatN &= - \laplacian \pstar
		-\sigma\grad\cdot\uvectN \text{ in }\spatialD\\
		\grad (\phatN-\pstar)\cdot\normalvec &= 0 \ \text{on}\ \Gamma_D.
	\end{align}
\end{subequations}
Eq. \eqref{eq:ppe} can be solved solely for $\phatN$, and is known as the pressure Poisson equation. Finally, the solenoidal velocity $\uvecSN$ is given by
\begin{align}
	\label{eq:vu}
	\uvecSN = \uvectN - \nicefrac{1}{\sigma} \grad (\phatN-\pstar) \ \text{in}\ \spatialD.
\end{align}
Eq.~\eqref{eq:vu} is often called the velocity update equation. Note that $\uvecSN$ is divergence-free but may not respect the Dirichlet boundary conditions since they are not imposed on $\uvecSN$. 

\paragraph{Choice of $\bm{\pstar}$.}
Popular choices for $\pstar$ are
\begin{align}
	\pstar &=
	\begin{cases}
		0\ \text{(Chorin-Temam)}, \\
		p^{n-1}\ (1^{\textrm{st}}\ \text{order extrapolation-rule}).
	\end{cases}
\end{align}
In this work, we use the first order extrapolation rule, which yields the ``incremental pressure correction'' method~\cite{goda1979multistep,vankan1986second}. A broader overview of projection-based methods can be found in~\cite{guermond2006overview}.

\paragraph{Galerkin formulation.}
\label{sec:galerkin-of-proj}
The Galerkin formulation of \eqref{eq:chorin-temam-generic} mirrors the presentation in \secref{sec:implicit-galerkin-formulation}. We seek $\uvectN \in \spaceVd$ such that, for all $\vvec \in \spaceV$,
\begin{multline} \label{eq:nse-bdf-mom-galerkin}
	\inner{\sigma \uvectN}{\vvec} + \inner{(\uvectN\cdot\grad)\uvectN}{\vvec} + \visco\inner{\grad\uvectN}{\grad \vvec} + \inner{\grad\pstar}{\vvec}  + \inner{\normalvec\cdot\left[- \visco\grad \uvectN\right]}{\vvec}_{\Gamma_{N}} = \inner{\gvecN}{\vvec}
\end{multline}
holds for any $\vvec\in \spaceV $. For the Galerkin formulation of \eqref{eq:ppe}, let us define 
\begin{align}
	\phi = \phatN - \pstar.
\end{align}
Then, for a test function $q \in \spaceQ$, we can write
\begin{align}
	-\intSpace \laplacian \phi \, q\, \dx = \intSpace \grad \phi \cdot \grad q \,\dx + \int_{\partial\spDom} (\normalvec\cdot\grad \phi) q\, \ds.
\end{align}
Now, $\partial \spDom = \Gamma_D^p \cup \Gamma_N^p$, and we have $\normalvec\cdot\grad \phi = 0$ on $\Gamma_N^p$, and $q = 0$ on $\Gamma_D^p$, therefore, the boundary term vanishes. So, we have
\begin{align}
	 \label{eq:nse-bdf-ppe-galerkin}
	\inner{\grad \phatN}{\grad q} = \inner{\grad \pstar}{\grad q} - \sigma \inner{\divergence\uvectN}{q}
\end{align}
for all $q \in \spaceQ$. Finally, the velocity update is given by
\begin{align}
	 \label{eq:nse-bdf-vue-galerkin}
	\inner{\uvecSN}{\wvec} = \inner{\uvectN}{\wvec} - \nicefrac{1}{\sigma} \inner{\grad (\phatN - \pstar)}{\wvec} \quad \forall \wvec\in H^1(\spDom).
\end{align}
Therefore, in summary, we seek $\uvectN\in \spaceVd$, $\phatN \in \spaceQ$ and $\uvecSN \in \bm{H}^1(\Omega)$ such that,
\begin{subequations}\label{eq:chorin-temam-galerkin}
	\begin{empheq}[left=\empheqlbrace]{alignat=2}
		&{\text{Momentum}}:\quad
		&&\inner{\sigma \uvectN}{\vvec}
		+ \inner{\uvectN\cdot\grad\uvectN}{\vvec}
		+ \visco\inner{\grad\uvectN}{\grad \vvec}
		+ \inner{\grad\pstar}{\vvec}
		\nonumber\\
		&&&\quad
		+ \inner{\normalvec\cdot\left[- \visco\grad \uvectN\right]}{\vvec}_{\Gamma_{N}}
		= \inner{\gvecN}{\vvec}
		\ \forall \vvec\in \spaceV,
		\label{eq:chorin-temam-gal-mom} \\[0.5em]
		&{\text{Pressure\ Poisson}:}\quad
		&&\inner{\grad \phatN}{\grad q}
		= \inner{\grad \pstar}{\grad q}
		- \sigma \inner{\divergence\uvectN}{q}
		\ \forall q \in \spaceQ,
		\label{eq:chorin-temam-gal-ppe} \\[0.5em]
		&{\text{Velocity\ update}:}\quad
		&&\inner{\uvecSN}{\wvec}
		= \inner{\uvectN}{\wvec}
		- \nicefrac{1}{\sigma}\inner{\grad(\phatN-\pstar)}{\wvec}
		\ \forall \wvec\in \boldsymbol{H}^1(\spDom). \label{eq:chorin-temam-gal-vue}
	\end{empheq}
\end{subequations}

%
%

\subsubsection{The projection method with VMS stabilization}
\label{sec:projection-with-vms}
Eq. \eqref{eq:chorin-temam-gal-mom}, \eqref{eq:chorin-temam-gal-ppe} and \eqref{eq:chorin-temam-gal-vue} are solved sequentially in each time step. Thus, this system of equations no longer have the saddle-point nature of the fully coupled formulation given in \eqref{eq:nse-bdf-galerkin}. However, \eqref{eq:chorin-temam-galerkin} still needs to be provided with advective and pressure stability, both of which the VMS treatment confers on the resulting scheme, as discussed below. To this end, we adapt the VMS methodology in the projection scheme framework. 

%
%
%
%


In the classical VMS framework of \secref{sec:vms-stabilization-overview}, both the velocity and the pressure are decomposed into coarse and fine scales, because the monolithic formulation must be stabilized against \emph{both} sources of instability identified in \secref{sec:implicit-galerkin-main}: the inf--sup deficiency of the velocity-pressure saddle point, and the dominance of advection in the momentum equation. The projection splitting changes this picture fundamentally, because it dismantles the saddle point itself. It is therefore worth examining the three subproblems~\eqref{eq:chorin-temam-gal-mom}--\eqref{eq:chorin-temam-gal-vue} in turn to decide which of them still requires a fine-scale model.

Consider first the pressure. In the monolithic formulation the pressure carries a
fine scale $\pprime$---and with it the parameter $\tauc$---precisely to relax the
inf--sup condition for equal-order spaces. Here the pressure is no longer the
Lagrange multiplier of an indefinite saddle point: the projection has recast it as
the unknown of the pressure Poisson equation~\eqref{eq:chorin-temam-gal-ppe}, whose
bilinear form $\inner{\grad\phatN}{\grad q}$ is coercive on
$\spaceQ\cap L^2_0(\spDom)$ by the Poincar\'e inequality. By the Lax--Milgram theorem
this problem is well-posed for equal-order velocity-pressure spaces with \emph{no}
inf--sup condition. The structural reason for decomposing
the pressure has thus disappeared, and $\phatN$ may be approximated directly by
$\phath$ without any scale separation.

The same holds for the corrected velocity. The update~\eqref{eq:chorin-temam-gal-vue}
is an $L^2$-projection onto the weakly divergence-free subspace; its bilinear form is
the $L^2$ inner product, coercive with constant one and unconditionally stable, and it
involves neither an advective operator nor a constraint. There is consequently nothing
to stabilize, and the divergence-free velocity $\uvecSN$ is not decomposed either.

What the splitting does \emph{not} remove is the advective instability, and that
instability is now confined to a single subproblem. The momentum
predictor~\eqref{eq:chorin-temam-gal-mom} is an advection--diffusion--reaction problem
in which the pressure has been demoted to known data; at moderate-to-high Reynolds
number its standard Galerkin discretization suffers exactly the spurious oscillations
described in \secref{sec:implicit-galerkin-main}. This is the only subproblem whose
well-posedness genuinely requires stabilization, and it is therefore the only one in
which a fine scale must be introduced. We accordingly decompose \emph{only} the
intermediate velocity $\uvectN$ and model its fine part by the momentum residual,
which furnishes the streamline (SUPG-type) control familiar from residual-based
stabilized methods~\citep{brooks1982streamline,hughes1995multiscale}.

Finally, although only the predictor \emph{requires} a fine scale, the unresolved
part $\uprimet$ that it generates is a genuine component of the intermediate velocity
and must be retained wherever that velocity subsequently reappears: it sources the
pressure through the divergence $\divergence(\uvecth+\uprimet)$ on the right-hand side
of~\eqref{eq:chorin-temam-gal-ppe}, and it is part of the field projected
in~\eqref{eq:chorin-temam-gal-vue}. In this sense all three subproblems ``see'' the fine scale. Its appearance in the pressure Poisson
equation is moreover what recovers a consistent PSPG-like term
(cf.~\eqref{eq:ppe-pspg-equiv-1}). We emphasize, however, that carrying $\uprimet$
into~\eqref{eq:chorin-temam-gal-ppe} and~\eqref{eq:chorin-temam-gal-vue} is a matter
of \emph{consistency} with the stabilized momentum balance, not of stability: both
subproblems are independently well-posed whether or not $\uprimet$ is retained. Only
the predictor needs the fine scale to be stable; the remaining two subproblems merely
inherit it. 

We therefore apply the VMS decomposition to $\uvectN$ as
\begin{align}
	\label{eq:vms-coarse-fine-decomposition}
	\uvectN(\xvec) &= \uvecth(\xvec) + \uvect_f(\xvec),
\end{align}
where $\uvecth$ and $\uvect_f$ denote the usual coarse and fine scales. Just like in \eqref{eq:vms-finescale-model-overview}, the fine scales of $\uvectN$ are modeled \emph{element-wise} in terms of the coarse-scale residuals of the momentum equations, i.e.,
\begin{align}
	\label{eq:nse-fine-scale-approximations}
	\uvect_f &\approx \uprimet := -\taum \resMomStrongTilde
\end{align}
where $\resMomStrongTilde$ is the elementwise strong-form residual of the momentum equation given by
\begin{align}
	\label{eq:coarse-scale-residuals-proj}
	\resMomStrongTilde := \sigma\uvecth + (\uvecth\cdot\grad)\uvecth -\visco\laplacian\uvecth + \grad\pstar_h - \gvec^n.
\end{align}
The stabilization parameter $\taum$ is evaluated using the same formula in \eqref{eq:vms-tau-m}, with $\uvech$ replaced by $\uvecth$.
Finally, substituting ~\eqref{eq:vms-coarse-fine-decomposition} in~\eqref{eq:chorin-temam-galerkin} and replacing $\uvect_f$ by $\uprimet$, we have
\begin{subequations}\label{eq:chorin-temam-vms}
	\begin{empheq}[left=\empheqlbrace, box=\fbox]{alignat=2}
		&{\text{Momentum (VMS)}}:\quad
		&&\inner{\sigma \uvecth}{\vvech}
		+ \inner{(\uvecth+\uprimet)\cdot\grad(\uvecth+\uprimet)}{\vvech} \nonumber\\
		&&&+ \visco\inner{\grad\uvecth}{\grad \vvech} 
		+ \inner{\grad\pstar_h}{\vvech} \nonumber \\ 
		&&&+ \inner{\normalvec\cdot\left[- \visco\grad \uvecth\right]}{\vvech}_{\Gamma_{N}}
		= \inner{\gvecN}{\vvech}
		\quad \forall \vvech\in \spaceVh,
		\label{eq:chorin-temam-vms-mom} \\[0.5em]
		&{\text{Pressure\ Poisson (VMS)}:}\quad
		&&\inner{\grad \phath}{\grad q}
		= \inner{\grad \pstar_h}{\grad q}
		- \sigma \inner{\divergence(\uvecth+\uprimet)}{q}
		\ \forall q \in \spaceQ,
		\label{eq:chorin-temam-vms-ppe} \\[0.5em]
		&{\text{Velocity\ update (VMS)}:}\quad
		&&\inner{\uvecSh}{\wvec}
		= \inner{(\uvecth+\uprimet)}{\wvec}
		- \tfrac{1}{\sigma}\inner{\grad(\phath-\pstar_h)}{\wvec}
		\ \forall \wvec\in \boldsymbol{H}^1(\spDom). \label{eq:chorin-temam-vms-vue}
	\end{empheq}
\end{subequations}
Eq. \eqref{eq:chorin-temam-vms}, along with \eqref{eq:nse-fine-scale-approximations} and \eqref{eq:coarse-scale-residuals-proj} specifies the stabilized projection method. Note that the simplifications discussed in Remark~\ref{rem:vms-assumptions-time-diffusion} are also applied to $\uvecth$ in \eqref{eq:chorin-temam-vms-mom}.

Since $\uprime$ is defined element-wise, all inner products involving $\uprime$ in \eqref{eq:chorin-temam-vms} are to be interpreted as elementwise. Once \eqref{eq:nse-fine-scale-approximations}--\eqref{eq:coarse-scale-residuals-proj} are substituted in \eqref{eq:chorin-temam-vms-mom}, it can be solved for $\uvecth$. Next, \eqref{eq:chorin-temam-vms-ppe} can be solved for $\phath$ since $\pstar$, $\uvecth$ and $\uprimet$ are known at this point. Finally, \eqref{eq:chorin-temam-vms-vue} can be solved for $\uvecSh$ since all the quantities in the right hand side are known. The full algorithm is provided in Algorithm~\ref{alg:vms-chorin-projection}. The fully expanded forms of these equations in terms of $\uvecth$ (after substituting the fine-scale model from \eqref{eq:nse-fine-scale-approximations} and \eqref{eq:coarse-scale-residuals-proj}) are also given in Algorithm~\ref{alg:vms-chorin-projection}.

In the VMS-stabilized projection method discussed above, the monolithic problem \eqref{eq:linear-system-vms} is replaced by a sequence of elliptic subproblems given in \eqref{eq:chorin-temam-vms}. The intermediate velocity problem is a nonlinear advection-diffusion-type equation, whereas the pressure Poisson and the velocity update equations are linear elliptic problems.

\begin{remark}
	In \eqref{eq:chorin-temam-vms-ppe}, the spatial derivatives of $\uprimet$ are not computable, therefore one must perform an integration-by-parts to keep it derivative-free, i.e.,
	\begin{align}
		\label{eq:fine-scale-treatment-cont}
		-\innerh{\divergence(\uvecth+\uprimet)}{q} = -\inner{\divergence\uvecth}{q} + \innerh{\uprimet}{\grad q} - \inner{\normalvec\cdot\uprimet}{q}_{\Gamma_{N},h}.
	\end{align}
	The second inner product on the right hand side is
	\begin{align}
		\innerh{\uprimet}{\grad q} = -\innerh{\taum \resMomStrongTilde}{\grad q},
	\end{align}
	which is immediately recognized to be very similar to the classical PSPG stabilization term (cf. \eqref{eq:vms-discussion-pspg-term}).
\end{remark}
\begin{remark}
	Note that, due to the fact that pressure is not decomposed into coarse and fine scales, the second stabilization parameter $\tauc$ does not appear in \eqref{eq:chorin-temam-vms}.
\end{remark}
\begin{algorithm}[t]
	\caption{VMS-stabilized projection scheme (one time step $t^n \to t^{n+1}$)}
    \label{alg:vms-chorin-projection}
	\begin{algorithmic}

	\Require Previous-time solution fields and boundary data; forcing
	$\gvec^n$; extrapolated pressure $\pstar_h$; coefficient $\sigma$
	from the BDF scheme.
	\Ensure Updated velocity $\uvecSh$ and pressure $\phath$.

	\AlgStage{Step 1: Momentum predictor (nonlinear)}
	\State Find $\uvecth \in \spaceVhd$ such that, for all
	$\vvech \in \spaceVh$,
	\begin{multline*}
		\scriptsize
		\sigma\,(\uvecth, \vvech)
		+ \bigl((\uvecth\cdot\grad)\uvecth, \vvech\bigr)
		+ \visco\,(\nabla\uvecth, \nabla\vvech)
		+ (\nabla\pstar_h, \vvech) + \bigl(\normalvec\cdot[- \visco\,\nabla\uvecth],\vvech\bigr)_{\Gamma_N} \\
		- \sum_{K \in \mesh} \bigl(\taum \left[ \sigma\uvecth + (\uvecth\cdot\grad)\uvecth -\visco\laplacian\uvecth + \grad\pstar_h -\gvecN \right]\cdot\grad\uvecth, \vvech\bigr)_K \\
		+ \sum_{K \in \mesh} \bigl(\taum \left[ \sigma\uvecth + (\uvecth\cdot\grad)\uvecth -\visco\laplacian\uvecth + \grad\pstar_h -\gvecN \right], (\uvecth\cdot\grad)\vvech\bigr)_K \\
		- \sum_{K \in \mesh} \bigl(\taum \left[ \sigma\uvecth + (\uvecth\cdot\grad)\uvecth -\visco\laplacian\uvecth + \grad\pstar_h -\gvecN \right], \taum \left[ \sigma\uvecth + (\uvecth\cdot\grad)\uvecth -\visco\laplacian\uvecth + \grad\pstar_h -\gvecN \right]\cdot\grad\vvech\bigr)_K \\
		- (\gvec^n, \vvech) = 0.
	\end{multline*}
	Here, $\uprimet = -\taum\resMomStrongTilde$, and the boundary integral in the penultimate line arises from the integration by parts of the fine-scale advection; it vanishes when the fine scales are neglected on $\partial\Omega$ (Remark~\ref{rem:vms-assumptions-time-diffusion}). This nonlinear system is solved with Newton's method.

	\AlgStage{Step 2: Pressure Poisson equation (linear)}
	\State With $\uvecth$ given from Step 1, find
	$\phath \in \spaceQh$ such that, for all $q \in \spaceQh$,
	\begin{align*}
		(\nabla\phath, \nabla q) = (\nabla\pstar_h, \nabla q)
		- \sigma\,\bigl(\nabla\cdot\uvecth, q\bigr) - \sigma \sum_{K \in \mesh} \bigl(\taum \left[\sigma\uvecth + (\uvecth\cdot\grad)\uvecth -\visco\laplacian\uvecth + \grad\pstar_h -\gvecN \right], \nabla q\bigr)_K.
	\end{align*}

	\AlgStage{Step 3: Velocity projection (linear)}
	\State Given $\uvecth$ (from Step 1) and $\phath$ (from Step 2), find the corrected velocity
	$\uvecSh \in \boldsymbol{H}^1(\Omega)$ such that, for all
	$\wvec \in \boldsymbol{H}^1(\Omega)$,
	\begin{align*}
		(\uvecSh, \wvec) = (\uvecth, \wvec) - \sum_{K \in \mesh} \bigl(\taum \left[\sigma\uvecth + (\uvecth\cdot\grad)\uvecth -\visco\laplacian\uvecth + \grad\pstar_h -\gvecN \right], \wvec \bigr)_K
		- \tfrac{1}{\sigma}\bigl(\nabla(\phath - \pstar_h), \wvec\bigr).
	\end{align*}

	\AlgStage{Step 4: Pressure update}
	\State Set $\pstar_h \gets \phath$ for the next time step.

	\end{algorithmic}
\end{algorithm}

\begin{remark}[Linearization of the momentum predictor]
\label{rem:linearization}
The momentum predictor in Step~1 of Algorithm~\ref{alg:vms-chorin-projection}
is a fully nonlinear system in $\uvecth$. We solve it with Newton's
method (PETSc SNES), reassembling the residual and Jacobian at each
Newton iterate. The fine-scale closure
$\uprimet = -\taum(\uvecth)\,\resMomStrongTilde(\uvecth,\pstar_h)$
depends on the current iterate of $\uvecth$ and is recomputed
consistently at every Newton step; no lagging or extrapolation is
applied to $\uprimet$ or to the advecting velocity inside the step.
The resulting fully implicit treatment converges in one or two Newton
iterations per time step for the cases reported in
Section~\ref{sec:results}, with the inner linear systems solved by
BiCGStab with additive Schwarz preconditioning.
\end{remark}

\subsubsection{Discussion on the proposed method}
\begin{figure}[t]
	\centering
	\begin{tikzpicture}[>=stealth]
		\node (A) at (0,3.0)    {$\left(\uvec(t^{n}),\, p(t^{n})\right)$};
		\node (B) at (0,0)      {$\left(\uvecN,\, \pN\right)$};
		\node (C) at (3.0,-2.4) {$\left(\uvectN,\, \uvecSN,\, \phatN\right)$};
		\node (D) at (3.0,-5.2) {$\left(\uvecth,\, \uvecSh,\, \phath\right)$};
		\node (E) at (0,-5.2)   {$\left(\uvech,\, \ph\right)$};
		\draw[->,very thick] (A) -- node[midway,left,align=center]{\small time\\ \small discretization} (B);
		\draw[->,very thick] (B) -- node[midway,above,sloped,align=center]{\small projection\\ \small splitting} (C);
		\draw[->,very thick] (C) -- node[midway,right,align=center]{\small VMS\\ \small (split)} (D);
		\draw[->,dashed]     (B) -- node[midway,left,align=center]{\small VMS\\ \small (monolithic)} (E);
	\end{tikzpicture}
	\caption{A schematic diagram of the path to the VMS-stabilized projection method \eqref{eq:chorin-temam-vms}. For reference, the path to the monolithic VMS problem \eqref{eq:coupled-vms} is shown with a dashed line.}
	\label{fig:schematic-method-path}
\end{figure}
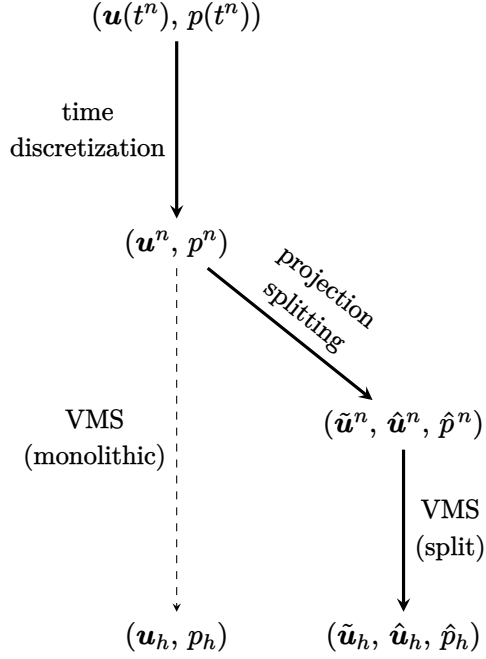
The VMS-stabilized projection scheme given by \eqref{eq:chorin-temam-vms} is both similar to and different from the monolithic VMS formulation given by \eqref{eq:coupled-vms} in multiple ways. Below we provide a comparative discussion.

\paragraph{Logical sequence of discretization.} The two formulations and the steps relating them are summarized in \figref{fig:schematic-method-path}. Beginning from the continuous solution $\left(\uvec(t^{n}),p(t^{n})\right)$, a time discretization gives the coupled time-discrete problem $\left(\uvecN,\pN\right)$. From there, The dashed path applies the VMS discretization directly to the coupled problem, yielding the monolithic solution $\left(\uvech,\ph\right)$ of \eqref{eq:coupled-vms}. Whereas, the solid path applies the projection splitting to obtain the fractional-step quantities $\left(\uvectN,\uvecSN,\phatN\right)$, and then the VMS spatial discretization of the split problem to obtain the proposed solution $\left(\uvecth,\uvecSh,\phath\right)$ of \eqref{eq:chorin-temam-vms}. The two routes share the common starting point $\left(\uvecN,\pN\right)$.

\paragraph{Momentum equation / velocity prediction.}
In the VMS-stabilized projection scheme \eqref{eq:chorin-temam-vms}, the momentum equation (or the velocity predictor) is nonlinear in $\uvech$, just like the monolithic VMS formulation \eqref{eq:coupled-vms}. In fact, the left hand side of \eqref{eq:chorin-temam-vms-mom} looks very close to the monolithic momentum residual \eqref{eq:coupled-weak-residuals-mom}, except the pressure, which is a known quantity in \eqref{eq:chorin-temam-vms-mom}, but a coupled unknown in \eqref{eq:coupled-weak-residuals-mom}.

\paragraph{Pressure Poisson equation.} The monolithic formulation \eqref{eq:coupled-vms} does not comprise a separate pressure Poisson equation. However, we can still draw parallels between the continuity equation in \eqref{eq:coupled-vms-con} and the PPE \eqref{eq:chorin-temam-vms-ppe}. For the discussion below, let us assume that $\Gamma_{N} = \phi$.

From \eqref{eq:fine-scale-treatment-cont}, we can write
\begin{align}
	\innerh{\divergence\uprimet}{q} &= -\innerh{\uprimet}{\grad q} + \inner{\normalvec\cdot\uprimet}{q}_{\Gamma_{N},h} = -\innerh{\uprimet}{\grad q} = \innerh{\taum \resMomStrongTilde}{\grad q}.
\end{align}
From the projection method PPE \eqref{eq:chorin-temam-vms-ppe}, we have
\begin{align}
	0 &= \inner{\grad \phath}{\grad q}
	- \inner{\grad \pstar_h}{\grad q}
	+ \sigma \inner{\divergence(\uvecth+\uprimet)}{q} \nonumber \\
	&= \inner{\grad \phath}{\grad q}
	- \inner{\grad \pstar_h}{\grad q} + \sigma \inner{\divergence\uvecth}{q} \nonumber \\ &\qquad\qquad + \innerh{\sigma\taum \left[\sigma\uvecth + (\uvecth\cdot\grad)\uvecth -\visco\laplacian\uvecth + \grad\pstar_h - \gvec^n\right]}{\grad q}. \label{eq:ppe-pspg-equiv-1}
\end{align}
Similarly for the monolithic case, multiplying \eqref{eq:coupled-vms-con} by $\sigma$, we have
\begin{align}
	0 &=\sigma \innerh{\divergence(\uvech+\uprime)}{q} \nonumber \\
	&= \sigma \inner{\divergence\uvech}{q} + \sigma\innerh{\taum \left[\sigma\uvech
		+ (\uvech\cdot\grad)\uvech - \visco\laplacian\uvech + \grad \phath - \gvec^n\right]}{\grad q} \nonumber \\
	&= \inner{\sigma\taum \grad\phath}{\grad q} - \inner{\sigma\taum \grad{p^{**}}}{\grad q} + \sigma\inner{\divergence\uvech}{q}  \nonumber \\ &\qquad\qquad + \innerh{\sigma\taum \left[\sigma\uvech
		+ (\uvech\cdot\grad)\uvech - \visco\laplacian\uvech  + \grad {p^{**}} - \gvec^n\right]}{\grad q}, \label{eq:ppe-pspg-equiv-2}
\end{align}
where we have introduced a new quantity $p^{**}$ in the third line to make the comparison clearer, even though $p^{**}$ is not the same as $\pstar$. However, $p^{**}$ can be chosen in a similar way, e.g., $p^{**}$ could be set to $p^{n-1}_h$, i.e., $p^{**} \gets p^{n-1}_h$. Written this way, we can see that the equations given by \eqref{eq:ppe-pspg-equiv-1} and \eqref{eq:ppe-pspg-equiv-2} are very similar to each other. Both formulations result in a Poisson operator on the pressure with different coefficients. A major difference however, is that \eqref{eq:ppe-pspg-equiv-1} is a linear equation in $\phathN$ since $\uvecthN$ is known (see Algorithm \ref{alg:vms-chorin-projection}). Whereas, \eqref{eq:ppe-pspg-equiv-2} is linear in $\phath$ but nonlinear in $\uvecthN$. Note that the expressions in \eqref{eq:ppe-pspg-equiv-1} and \eqref{eq:ppe-pspg-equiv-2} are dimensionally consistent since $\taum$ has units of time (see \eqref{eq:vms-tau-m}).

\paragraph{Velocity update equation.} Define the solenoidal space
\begin{align}
	\mvec{V}_{\mathrm{div}}(\spatialD)
	:= \left\{\wvec\in\mvec{H}^1(\spatialD):
	\divergence\wvec=0\ \text{in }\spatialD,\quad
	\wvec\cdot\normalvec=0\ \text{on }\partial\spatialD\right\},
\end{align}
For the Galerkin method, it is well known that the corrected velocity $\uvecSN$ is weakly divergence-free (by design), and is the $L^2(\spatialD)$-projection of $\uvectN$ into $\mvec{V}_{\mathrm{div}}(\spatialD)$. These properties also hold for the VMS-stabilized method presented above. In \eqref{eq:chorin-temam-vms-vue}, if we choose the test function $\wvec$ from $\mvec{V}_{\mathrm{div}}(\spatialD)$ instead of $\mvec{H}^1(\spatialD)$, we have
\begin{align}
	0 &= \inner{\uvecSN - (\uvecth+\uprimet)}{\wvec} - \tfrac{1}{\sigma}\inner{\grad (\phath - \pstar_h)}{\wvec} \\
	&= \inner{\uvecSN - (\uvecth+\uprimet)}{\wvec} + \tfrac{1}{\sigma}\inner{\phath - \pstar_h}{\divergence\wvec} \\
	&= \inner{\uvecSN - (\uvecthN+\uprime)}{\wvec}
\end{align}
since $\divergence\wvec = 0$ in $\spatialD$.

\paragraph{Difference between the end of step velocity solutions.}
The monolithic formulation solves for $\uvech$ which is a coarse scale quantity. It has a fine scale counterpart that must be accounted for when the total velocity is needed in a calculation. The intermediate velocity in the VMS/projection formulation \eqref{eq:chorin-temam-vms} is of the same characteristic, i.e., it has a fine-scale counterpart as seen in \eqref{eq:chorin-temam-vms}. However, the corrected velocity $\uvecSN$ does not have this interpretation since it is an $L^2(\spatialD)$-projection of the total intermediate velocity.

\section{Error analysis}
\label{sec:analysis}

In this section, we provide a \emph{formal} analysis of the VMS-stabilized projection scheme derived above. The scheme involves two distinct approximations of the continuous, coupled Navier-Stokes equations~\eqref{eq:nse}: a \emph{temporal splitting}, in which the Helmholtz-Leray projection replaces the coupled momentum--continuity system by the sequence \eqref{eq:chorin-temam-generic}--\eqref{eq:vu}, and the \emph{spatial discretization} of the split equations by the variational multiscale method, yielding~\eqref{eq:chorin-temam-vms}. We quantify the errors of these two approximations in a single time step---the splitting error and the VMS spatial error, together with the underlying BDF time-discretization error---in \secref{sec:single-step-error}, and combine them into a full \textit{a priori} estimate in \secref{sec:full-estimate}. Throughout, we assume that the exact solution is sufficiently regular and that the scheme is run in a stable regime, so that per-step errors are not amplified during time marching.

\subsection{Error in a single step}
\label{sec:single-step-error}
We now analyze the errors in a single time-step. At a fixed time $t^{n}$, we have the following solution tuples (see also \tabref{tab:velocity-notation} and \figref{fig:schematic-method-path}):
\begin{itemize}
	\item $\uvec(t^{n})$, the exact solution of the coupled
	NSE~\eqref{eq:nse},
	\item $(\uvecN,\pN)$, the solution of the time-discrete \emph{coupled}
	problem~\eqref{eq:nse-bdf} (BDF-$r$ in time, continuous in space),
	\item $(\uvech, \ph)$, the solution of the \emph{monolithic} VMS problem \eqref{eq:coupled-vms},
	\item $(\uvectN,\uvecSN,\phat^{\,n})$, the solution of the time-discrete
	\emph{projection} problem~\eqref{eq:chorin-temam-generic}--\eqref{eq:vu},
	with predictor $\uvectN$, corrected velocity $\uvecSN$, and
	pressure $\phatN$,
	\item $(\uvecth, \uvecSh, \phath)$, the solution of the VMS stabilized projection problem~\eqref{eq:chorin-temam-vms}.
\end{itemize}
Here, $\uvec(t^n)$ is continuous in both time and space; quantities with a superscript $n$ alone (e.g., $\uvecN$, $\uvectN$, $\uvecSN$) are discrete in time but continuous in space; and $\uvech$, $\uvecth$, $\uvecSh$ are discrete in both.

The total errors in the VMS-stabilized projection scheme \eqref{eq:chorin-temam-vms} are given by $\norm{\uvecSh-\uvec(t^n)}$ and $\norm{\phath - p(t^n)}$. Using the triangle inequality, we can decompose the velocity error as
\begin{equation}\label{eq:master-decomp}
	\norm{\uvec(t^{n})-\uvecSh}
	\;\le\;
	\underbrace{\norm{\uvec(t^{n})-\uvecN}}_{\text{time discretization}}
	\;+\;
	\underbrace{\norm{\uvecN-\uvecSN}}_{\text{splitting error}}
	\;+\;
	\underbrace{\norm{\uvecSN-\uvecSh}}_{\text{spatial VMS error}}.
\end{equation}
We quantify the three contributions in turn.

\subsubsection{Time discretization error}
The first term in~\eqref{eq:master-decomp} is the standard BDF-$r$
time-discretization error of the coupled problem, which is of order $O(\dt^{\,r})$
for sufficiently smooth solutions. For the BDF2 scheme used in this work, the
standard consistency estimate~\citep{butcher2016numerical} gives
\begin{align}
	\norm{\uvec(t^{n})-\uvecN} \leq C\, \dt^2.
\end{align}

\subsubsection{The splitting error}
\label{sec:splitting}
The second term in~\eqref{eq:master-decomp}, $\norm{\uvecN-\uvecSN}$, is the
error introduced by the projection splitting. This splitting acts on the
time-discrete coupled problem~\eqref{eq:nse-bdf} \emph{before} any spatial
discretization: with $\sigma=\beta_0/\dt$ and
$\Nop{\uvec}:=(\uvec\cdot\grad)\uvec$, the incremental
pressure-correction step computes a predictor
	\begin{equation}\label{eq:pred-td}
		\frac{1}{\dt}\Bigl(\beta_0\,\uvectN+\sum_{i=1}^{r}\beta_{-i}\,\hat{\uvec}^{\,n-i}\Bigr)
		+\Nop{\uvectN}-\visco\laplacian\uvectN+\grad\pstar=\fvecN,
		\qquad \pstar=\phat^{\,n-1},
	\end{equation}
	where $\hat{\uvec}^{\,n-i}$ denotes the corrected projection velocity from a previous time step,
	which for the BDF2 scheme used in this work reads
\begin{equation}\label{eq:pred-td-bdf2}
	\frac{3\uvectN-4\hat{\uvec}^{\,n-1}+\hat{\uvec}^{\,n-2}}{2\dt}
	+\Nop{\uvectN}-\visco\laplacian\uvectN+\grad\pstar=\fvecN,
\end{equation}
followed by the projection
\begin{equation}\label{eq:proj-td}
	\uvecSN=\uvectN-\tfrac1\sigma\grad\phi^{\,n},
	\qquad \phi^{\,n}:=\phat^{\,n}-\phat^{\,n-1},
	\qquad \divergence\uvecSN=0 .
\end{equation}
Because the VMS model enters only at the spatial-discretization stage
(\Cref{sec:spatial-error}), the split velocity $\uvecSN$ defined
by~\eqref{eq:pred-td}--\eqref{eq:proj-td} is precisely the solution of the
classical incremental (non-rotational) pressure-correction
scheme~\citep{goda1979multistep,vankan1986second}, whose splitting error
is established, in the continuous-in-space setting, by
\citet{shen1996error} and reviewed by \citet{guermond2006overview}; the
equivalent algebraic approximate-factorization viewpoint is due to
\citet{perot1993analysis}.

\begin{lemma}[Splitting error of the incremental scheme]\label{lem:rate}
	Let $(\uvectN,\uvecSN,\phat^{\,n})$ solve the incremental
	pressure-correction problem~\eqref{eq:pred-td}--\eqref{eq:proj-td}, and
	$(\uvecN,\pN)$ the coupled time-discrete problem~\eqref{eq:nse-bdf}, with the
	same data and sufficient regularity. Then
	\begin{equation}\label{eq:rate}
		\Bigl(\dt\sum_{n}\norm{\uvecN-\uvecSN}^{2}\Bigr)^{\halfnice}=O(\dt^{2}),
		\qquad
		\Bigl(\dt\sum_{n}\norm{\pN-\phat^{\,n}}^{2}\Bigr)^{\halfnice}=O(\dt).
	\end{equation}
\end{lemma}

The first-order pressure rate is intrinsic to the non-rotational
incremental scheme: the projection~\eqref{eq:proj-td} imposes the
artificial homogeneous-Neumann condition
$\grad\phi^{\,n}\cdot\hat{\mvec{n}}=0$ on the pressure increment, whose
boundary layer limits the pressure accuracy; the rotational variant
recovers $O(\dt^{3/2})$~\citep{guermond2006overview}.

Combining \Cref{lem:rate} with the $O(\dt^{2})$ BDF2 time-discretization error
of the coupled problem gives
\[
	\left(\dt\sum_n\norm{\uvec(t^{n})-\uvecSN}^{2}\right)^{\halfnice}
	=O(\dt^{2})
\]
for the velocity, consistent with the temporal convergence reported in
\Cref{sec:results}.

\begin{remark}[Role of the pressure increment]\label{rem:increment}
	The splitting defect carries one factor $\sigma^{-1}=O(\dt)$ from the
	projection~\eqref{eq:proj-td} and a second factor from the pressure increment
	$\phi^{\,n}$. The incremental choice $\pstar=\phat^{\,n-1}$ renders
	$\phi^{\,n}=O(\dt)$, so the two combine to $O(\dt^{2})$, whereas the
	non-incremental choice leaves $\phi^{\,n}=O(1)$ and only $O(\dt)$ remains.
	This is the continuous counterpart of the algebraic statement that the
	projection replaces the momentum-operator inverse by a scaled mass-matrix
	inverse~\citep{perot1993analysis}.
\end{remark}

\subsubsection{The spatial discretization error}
\label{sec:spatial-error}

The third term in~\eqref{eq:master-decomp}, $\norm{\uvecSN-\uvecSh}$, is
the error incurred by the VMS \emph{spatial} discretization of the split
problem; both $\uvecSN$ and $\uvecSh$ correspond to the same time level
and the same fractional-step equations, so this term contains no temporal
contribution. Because the projection replaces the velocity-pressure saddle
point by the sequence~\eqref{eq:chorin-temam-vms}---a nonlinear
advection--diffusion--reaction predictor, a pressure Poisson equation, and an
$L^{2}$ projection---the spatial error depends on the discretization of each
subproblem and on the propagation of errors through the sequence.

\begin{assumption}[Spatial approximation of the fractional subproblems]\label{ass:spatial}
	Let $q$ denote the spatial convergence order attained by the stabilized
	discretization in the $L^{2}$ norm; its value depends on the approximation
	spaces, solution regularity, stabilization model, and boundary conditions.
	We assume that the VMS discretizations of the three fractional subproblems
	are stable and consistent and that the total predicted velocity and pressure
	correction satisfy
	\begin{equation}\label{eq:spatial-rate}
		\norm{\uvectN-(\uvecth+\uprimet)}
		+\frac{1}{\sigma}\norm{\grad(\phat^{\,n}-\phath)}
		\leq C h^{q}.
	\end{equation}
	This assumption includes the effects of the modeled velocity fine scale,
	the approximation of the extrapolated pressure, and the boundary
	contributions.
\end{assumption}

Subtracting the continuous and discrete velocity-update equations and applying
the triangle inequality together with Assumption~\ref{ass:spatial} gives
\begin{equation}\label{eq:spatial-velocity-rate}
	\norm{\uvecSN-\uvecSh}\leq C h^{q}.
\end{equation}

\subsection{The full error estimate}
\label{sec:full-estimate}

Combining the estimates of \Cref{sec:splitting,sec:spatial-error} through
the decomposition~\eqref{eq:master-decomp} yields
the total error of the scheme.

\begin{proposition}[Formal error estimate of the VMS-stabilized projection scheme]\label{thm:full}
	Under the regularity, stability, and spatial-approximation assumptions
	above, the fully discrete incremental VMS-stabilized projection
	solution~\eqref{eq:chorin-temam-vms} obtained with BDF2 satisfies
	\begin{equation}\label{eq:full-estimate}
		\Bigl(\dt\sum_{n}\norm{\uvec(t^{n})-\uvecSh}^{2}\Bigr)^{\halfnice}
		\;=\; O\!\left(\dt^{2}+h^{q}\right).
	\end{equation}
\end{proposition}

\begin{proof}
	The decomposition~\eqref{eq:master-decomp} bounds the velocity error by its
	three contributions. The time-discretization term $\norm{\uvec(t^{n})-\uvecN}$
	is the $O(\dt^{2})$ BDF2 error of the coupled problem; the splitting term
	$\norm{\uvecN-\uvecSN}$ is $O(\dt^{2})$ by \Cref{lem:rate}; and the spatial VMS
	term $\norm{\uvecSN-\uvecSh}$ is $O(h^{q})$ by
	\eqref{eq:spatial-velocity-rate}. Under the
	standing stability assumption---the discrete energy of the scheme is
	controlled by its data, so that per-step consistency errors accumulate
	without amplification in order (a discrete Gr\"onwall
	argument~\citep{heywood1990finite})---summation of the per-step
	contributions gives~\eqref{eq:full-estimate}.
\end{proof}

\begin{remark}[Pressure and the role of each ingredient]\label{rem:full}
	For the pressure, the non-rotational incremental splitting gives the
	general temporal estimate $O(\dt)$ in the discrete $L^{2}$-in-time norm;
	a complete spatial pressure estimate is not pursued here. The rotational
	variant would improve the temporal pressure estimate to
	$O(\dt^{3/2})$~\citep{guermond2006overview}. The estimate~\eqref{eq:full-estimate} also makes explicit
	that the second-order temporal rate of the velocity requires both BDF2
	and the incremental pressure correction of \Cref{rem:increment}; with
	the non-incremental variant, the splitting error is only $O(\dt)$.
\end{remark}

\section{Results}
\label{sec:results}
We now examine the proposed formulation through a set of increasingly challenging numerical examples. These include -- (i) convergence tests with a manufactured solution, (ii) lid-driven cavity problem in 2D, (iii) flow past a cylinder in 2D and (iv) Taylor-Green flows in 3D.

\subsection{Method of Manufactured solutions}

\begin{figure}[t]
	\centering
	\begin{minipage}{.45\textwidth}
		\centering
		\begin{tikzpicture}
			\begin{loglogaxis}[
				width=0.99\linewidth, 
				xlabel=$\nicefrac{dt}{T}$,
				legend style={at={(0.01,0.99)},anchor=north west,legend columns=1}, 
				x tick label style={rotate=0,anchor=north}, 
				xtick={1/256, 1/128, 1/64, 1/32, 1/16},
				xticklabels={$2^{-8}$,$2^{-7}$,$2^{-6}$,$2^{-5}$,$2^{-4}$},
				]
				\addplot+[very thick] table[
				x expr={(1.0 / \thisrow{nsteps})},
				y expr={sqrt(\thisrow{err_x}^2 + \thisrow{err_y}^2)},
				col sep=comma
				]{data/mmsHarmonic/re_1e+03.txt};
				
				\addplot+[thick, dashed] table[
				x expr={(1.0 / \thisrow{nsteps})},
				y expr={sqrt(\thisrow{err_x}^2 + \thisrow{err_y}^2)},
				col sep=comma
				]{data/mmsHarmonic/re_1e+04.txt};
				
				\addplot+[thick, dashed] table[
				x expr={(1.0 / \thisrow{nsteps})},
				y expr={sqrt(\thisrow{err_x}^2 + \thisrow{err_y}^2)},
				col sep=comma
				]{data/mmsHarmonic/re_1e+05.txt};
				
				\addplot+[thick, dotted] table[
				x expr={(1.0 / \thisrow{nsteps})},
				y expr={sqrt(\thisrow{err_x}^2 + \thisrow{err_y}^2)},
				col sep=comma
				]{data/mmsHarmonic/re_1e+06.txt};
				
				\logLogSlopeTriangle{0.8}{0.2}{0.2}{2}{blue}
				\legend{{\tiny $Re = 10^3$}, {\tiny $Re = 10^4$}, {\tiny $Re = 10^5$}, {\tiny $Re = 10^6$}}
			\end{loglogaxis}
		\end{tikzpicture}
		\subcaption{$e_u(T)$.}
		\label{fig:nonlinear-s-1.0-mmsShen-velocity}
	\end{minipage}
	\hfill
	\begin{minipage}{.45\textwidth}
		\centering
		\begin{tikzpicture}
			\begin{loglogaxis}[
				width=0.99\linewidth, 
				xlabel=$\nicefrac{dt}{T}$,
				legend style={at={(0.01,0.99)},anchor=north west,legend columns=1}, 
				x tick label style={rotate=0,anchor=north}, 
				xtick={1/256, 1/128, 1/64, 1/32, 1/16},
				xticklabels={$2^{-8}$,$2^{-7}$,$2^{-6}$,$2^{-5}$,$2^{-4}$},
				]
				\addplot+[very thick] table[
				x expr={(1.0 / \thisrow{nsteps})},
				y expr={\thisrow{err_p}},
				col sep=comma
				]{data/mmsHarmonic/re_1e+03.txt};
				
				\addplot+[thick, dashed] table[
				x expr={(1.0 / \thisrow{nsteps})},
				y expr={\thisrow{err_p}},
				col sep=comma
				]{data/mmsHarmonic/re_1e+04.txt};
				
				\addplot+[thick, dashed] table[
				x expr={(1.0 / \thisrow{nsteps})},
				y expr={\thisrow{err_p}},
				col sep=comma
				]{data/mmsHarmonic/re_1e+05.txt};
				
				\addplot+[thick, dotted] table[
				x expr={(1.0 / \thisrow{nsteps})},
				y expr={\thisrow{err_p}},
				col sep=comma
				]{data/mmsHarmonic/re_1e+06.txt};
				
				\logLogSlopeTriangle{0.8}{0.2}{0.2}{2}{blue}
				\legend{{\tiny $Re = 10^3$}, {\tiny $Re = 10^4$}, {\tiny $Re = 10^5$}, {\tiny $Re = 10^6$}}
			\end{loglogaxis}
		\end{tikzpicture}
		\subcaption{$e_p(T)$.}
		\label{fig:nonlinear-s-1.0-mmsShen-pressure}
	\end{minipage}
	\caption{Convergence of velocity and pressure errors, $T=2.25$.}
	\label{fig:mmsShen-projection}
\end{figure}
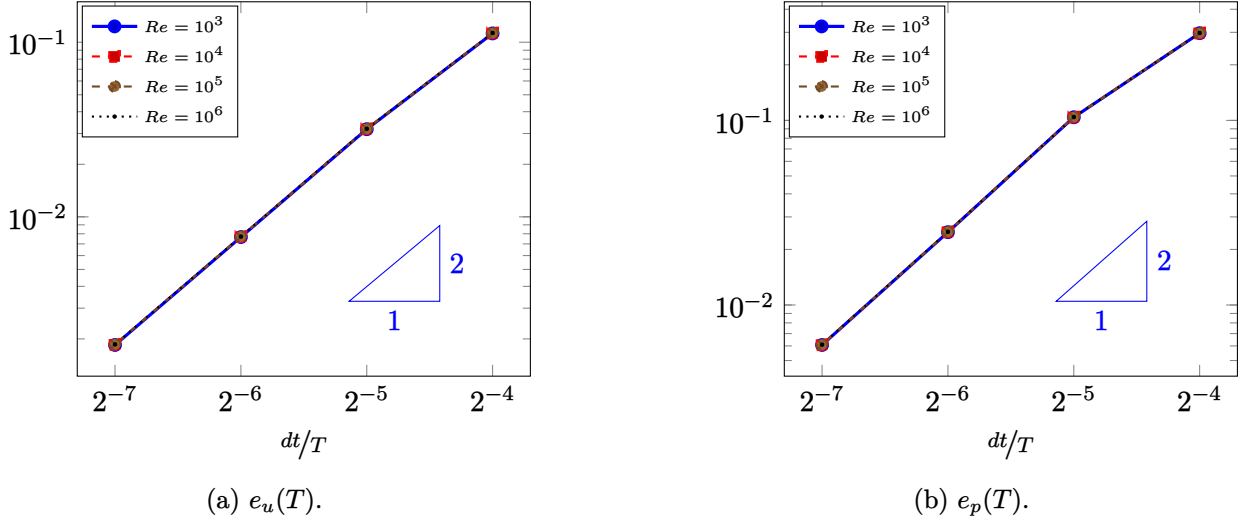

We begin by testing the convergence of the proposed method as the time-step size $\dt$ is reduced. We take the exact solution $ (\uvec, p) $ to be

\begin{subequations}\label{eq:mms-1}
    \begin{align}
        \bm{u} &= 
        \begin{bmatrix}
        \sin(\pi x) \cos(\pi y) \sin(2 \pi t) \\
        -\cos(\pi x) \sin(\pi y) \sin(2 \pi t) \\
        \end{bmatrix},\\
    	p &= \sin(\pi x) \sin(\pi y) \cos(2 \pi t),
    \end{align}
\end{subequations}

in a unit square domain, i.e., $\Omega = {[0,1]}^2$ and time interval $I_T=[0,T]$. We set $T=2.25$, since the velocities reach a peak at $t=n + \tfrac{1}{4}$ for $n\in\{0,1,2,\ldots\}$. The velocity $\uvec$ is chosen such that it satisfies $ \grad\cdot\uvec = 0 $ exactly pointwise. We obtain the analytical expression for the forcing $ \fvec $ by substituting the exact solutions $(\uvec, p)$ in~\eqref{eq:nse}. We also analytically obtain the initial condition $\uvec_0$ by setting $t=0$ in~\eqref{eq:mms-1}, and the boundary condition $\uvec_d$ on all sides by setting $\xvec = (x,y)$ appropriately. To uniquely determine the pressure in the discrete system, we apply a Dirichlet condition given by
\begin{align}
	\ph = 0 \  \text{at}\ (0,0).
\end{align}

We solve the incompressible Navier-Stokes equations \eqref{eq:nse} using the
proposed VMS-stabilized projection scheme \eqref{eq:chorin-temam-vms}, with the
manufactured forcing, initial condition, and boundary data derived above. To
assess the robustness of the temporal convergence across flow regimes, the test
is repeated for a range of Reynolds numbers, $Re \in \{10^3, 10^4, 10^5, 10^6\}$.
We measure the error in $(\uvecSh, \ph)$ with respect to the exact solution $(\uvec, p)$ using
\begin{align}
	e_u(t) &= \left[ \int_{\spDom} \big| \uvecSh(t) - \uvec(t) \big|^2\ \mbox{d}\xvec\right]^{\halfnice}, \\
	e_p(t) &= \left[ \int_{\spDom} (\phath - p)^2\ \mbox{d}\xvec\right]^{\halfnice}.
\end{align}
\figref{fig:mmsShen-projection} shows the convergence of these errors with respect to the step size $\dt$. Across all the tested Reynolds numbers, both errors exhibit approximately second-order temporal convergence. For the pressure, this observed rate is higher than the general first-order estimate for the non-rotational scheme, which need not be sharp for this smooth manufactured problem and its boundary conditions.

\subsection{Lid driven cavity}

\begin{figure}[!htb]
\centering
\begin{tikzpicture}[scale=4]

\draw[thick] (0,0) -- (1,0) -- (1,1) -- (0,1) -- cycle;

\node[align=center] at (0.5,-0.20) {$y=0:$ $\boldsymbol{u} = \boldsymbol{0}$,\\$\nabla(p-p^*)\cdot\hat{\boldsymbol{n}}=0$};   
\node[align=center] at (0.5,1.20) {$y=1:$ $\boldsymbol{u} = (1,0)^{T}$,\\$\nabla(p-p^*)\cdot\hat{\boldsymbol{n}}=0$};  
\node[align=center] at (-0.4,0.5) {$x=0:$ $\boldsymbol{u} = \boldsymbol{0}$,\\$\nabla(p-p^*)\cdot\hat{\boldsymbol{n}}=0$};   
\node[align=center] at (1.4,0.5) {$x=1:$ $\boldsymbol{u} = \boldsymbol{0}$,\\$\nabla(p-p^*)\cdot\hat{\boldsymbol{n}}=0$};    

\draw[dashed, very thick, blue!90] (0.5,0) -- (0.5,1);
\draw[dashed, very thick, blue!90] (0,0.5) -- (1,0.5);
\node[blue!90] at (0.55,0.15) {$A$};
\node[blue!90] at (0.15,0.55) {$B$};

\node at (-1.1,0.5) {$\Omega=[0,1]^2$};

\draw[dotted,very thick] (0,0) circle(0.07);

\fill (0,0) circle(0.015);

\node[align=center] at (-0.4,-0.25) (pzero) {$\textbf{x}=\textbf{0}$,\\$p=0$};
\draw[->,thick,>=stealth,line width=1pt] (pzero) -- (0,0);

\def\axOx{-1}
\def\axOy{0}
\def\axL{0.2}
\draw[very thick,->,>=stealth] (\axOx,\axOy) -- (\axOx+\axL,\axOy);
\draw[very thick,->,>=stealth] (\axOx,\axOy) -- (\axOx,\axOy+\axL);
\node at (\axOx+\axL,\axOy) [right] {$x$};
\node at (\axOx,\axOy+\axL) [left] {$y$};

\end{tikzpicture}
\caption{Boundary conditions for the lid-driven cavity problem. The blue dashed lines are the midlines where velocities will be evaluated and plotted (see \figref{fig:ldc-re-all-panel-1})}
\label{fig:schematic-ldc2d}
\end{figure}
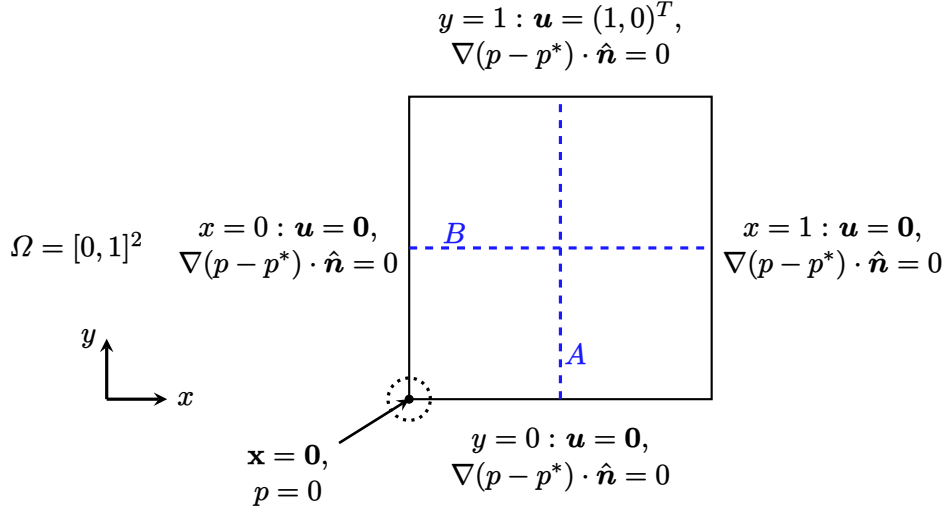

\input{supporting_tex/ldc_line_cuts.tex}

\figref{fig:schematic-ldc2d} shows a schematic of the lid-driven cavity problem in 2D. The computational domain is $\spDom = [0,1]^2$. The boundary conditions are static (do not change with time) and are given by
\begin{subequations}
	\begin{align}
		\uvec &= \zerovec,\ y=0,\\
		\uvec &= [1,0]^T,\ y=1,\\
		\uvec &= \zerovec\ x\in\{0,\ 1\}.
	\end{align}
\end{subequations}

The boundary condition comprise a ``lid'' at the top moving with horizontal velocity $=1$, and no-slip walls on the rest of the three sides. The forcing $\fvec = \zerovec$, and the initial condition is given by a zero function
\begin{align}
	\uvec_0 = \zerovec.
\end{align}
We solve this problem for a series of Reynolds numbers $ Re = $ 100, 1000, 5000 and 10000. The benchmark results for this problem generally comprise of steady-state solutions. Here, we choose $ T = 5\times 10^3 $ for all cases, such that a steady-state is achieved. We discretize the domain with a $128\times 128$ axis-aligned mesh of equal-order $Q_1/Q_1$ elements, graded to refine toward the boundaries.

\figref{fig:ldc-re-all-panel-1} shows the line cuts of the velocities at steady state. \figref{fig:ldc-re-ux} shows the $x$-component of velocity along the vertical midline ($x=0.5$, line A in \figref{fig:schematic-ldc2d}), and \figref{fig:ldc-re-uy} shows the $y$-component along the horizontal midline ($y=0.5$, line B in \figref{fig:schematic-ldc2d}). For comparison, we also include results from the monolithic VMS solver ($Q_1/Q_1$) and from an inf-sup-stable $Q_2/Q_1$ Taylor--Hood pair with SUPG stabilization, both run on the same mesh. Reference values from Ghia et al.\ are overlaid as well. The match between the projection-based and monolithic VMS solvers at all Reynolds numbers confirms that the projection-based splitting does not degrade steady-state accuracy.
\begin{figure}[!htb]
\centering
\begin{subfigure}{0.32\textwidth}
    \centering
    \begin{overpic}[width=\textwidth,trim={0.0cm 3.5cm 0.0cm 0.0cm},clip]
        {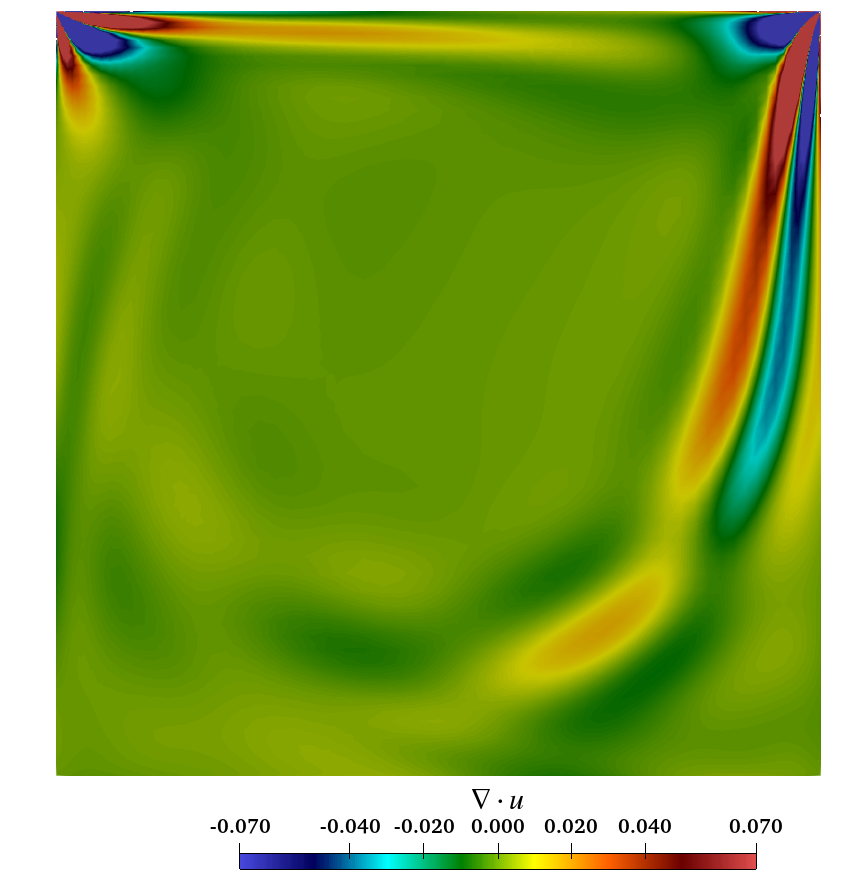}
        \put(78,27){\textcolor{black}{\dashbox{2}(20,65){}}}
    \end{overpic}
\caption{$\nabla \cdot \uvech$}
\end{subfigure}
\begin{subfigure}{0.32\textwidth}
    \centering
    \begin{overpic}[width=\textwidth,trim={0.0cm 3.5cm 0.0cm 0.0cm},clip]
        {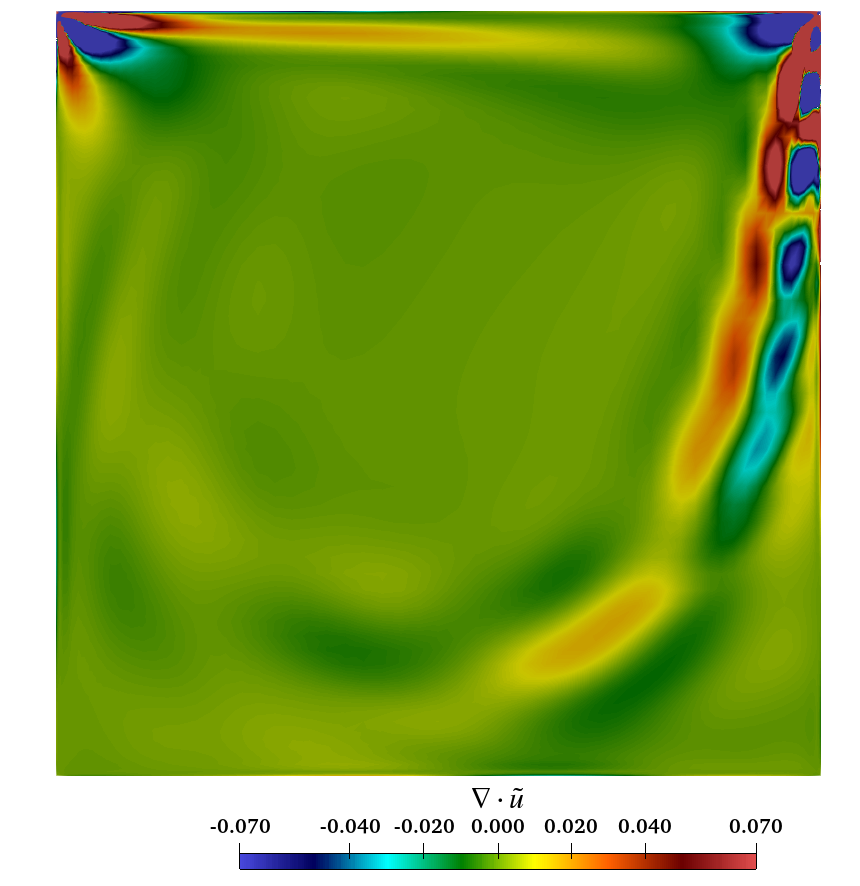}
        \put(78,27){\textcolor{black}{\dashbox{2}(20,65){}}}
    \end{overpic}
\caption{$\nabla \cdot \uvecth$ }
\end{subfigure}
\begin{subfigure}{0.32\textwidth}
    \centering
    \begin{overpic}[width=\textwidth,trim={0.0cm 3.5cm 0.0cm 0.0cm},clip]
        {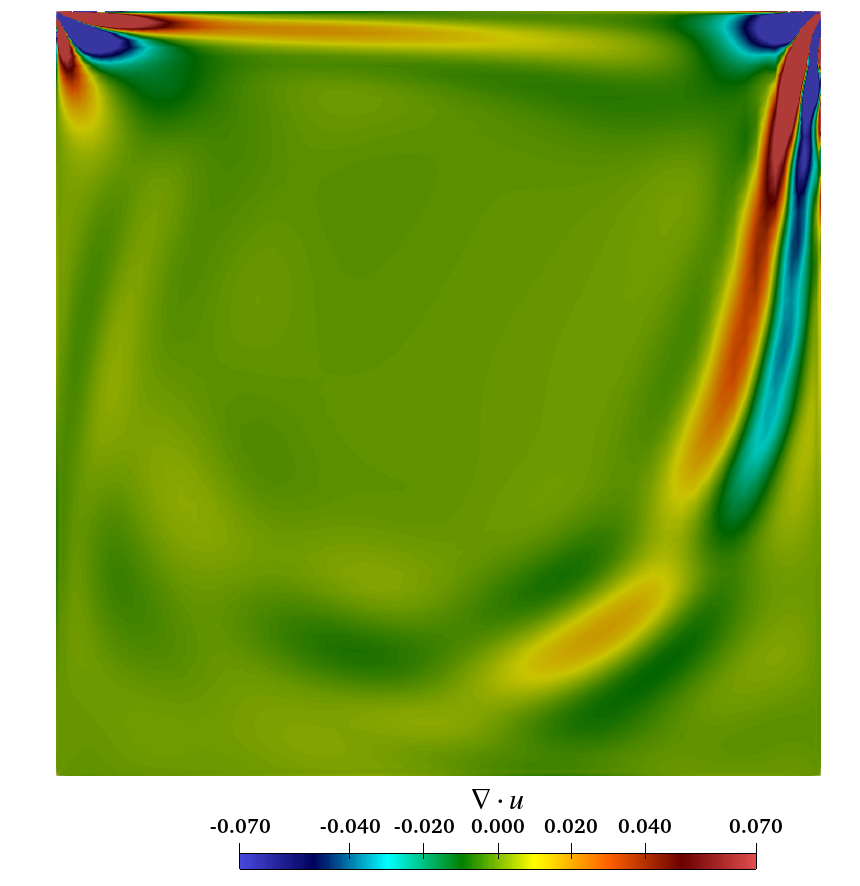}
        \put(78,27){\textcolor{black}{\dashbox{2}(20,65){}}}
    \end{overpic}
\caption{$\nabla \cdot \uvecSh$ }
\end{subfigure}
\\
\vspace{0.3cm}
\begin{subfigure}{\textwidth}
    \centering
    \begin{tikzpicture}
        \begin{axis}[
            hide axis,
            scale only axis,
            height=0pt,
            width=0pt,
            colormap={rainbowdesat}{
                rgb=(0.278, 0.278, 0.859)
                rgb=(0.000, 0.000, 0.360)
                rgb=(0.000, 1.000, 1.000)
                rgb=(0.000, 0.502, 0.000)
                rgb=(1.000, 1.000, 0.000)
                rgb=(1.000, 0.380, 0.380)
                rgb=(0.420, 0.000, 0.000)
            },
            point meta min=-0.07,
            point meta max=0.07,
            colorbar horizontal,
            colorbar style={
                width=8cm,
                height=0.35cm,
                xtick={-0.07,-0.04,-0.02,0.0,0.02,0.04,0.07},
                scaled x ticks=false,
                xticklabel style={font=\small,
                                  /pgf/number format/.cd,
                                  fixed,
                                  fixed zerofill,
                                  precision=2,
                                  /tikz/.cd},
                title style={font=\small, yshift=-1ex},
            },
        ]
        \addplot [draw=none] coordinates {(0,0) (1,1)};
        \end{axis}
    \end{tikzpicture}
\end{subfigure}
\caption{An example of divergence fields at steady state for the lid-driven cavity at $Re=1000$: (a) VMS monolithic, (b) VMS/projection: after prediction (intermediate velocity), (c) VMS/projection: after projection (updated velocity). Dashed boxes mark the region of interest near the top-right corner singularity.}
\label{fig:divergenceofvelocity}
\end{figure}

\begin{figure}[!htb]
\centering
\begin{subfigure}{0.4\textwidth}
    \centering
\includegraphics[width=\textwidth,trim={0.0cm 3.5cm 0.0cm 0.0cm},clip]{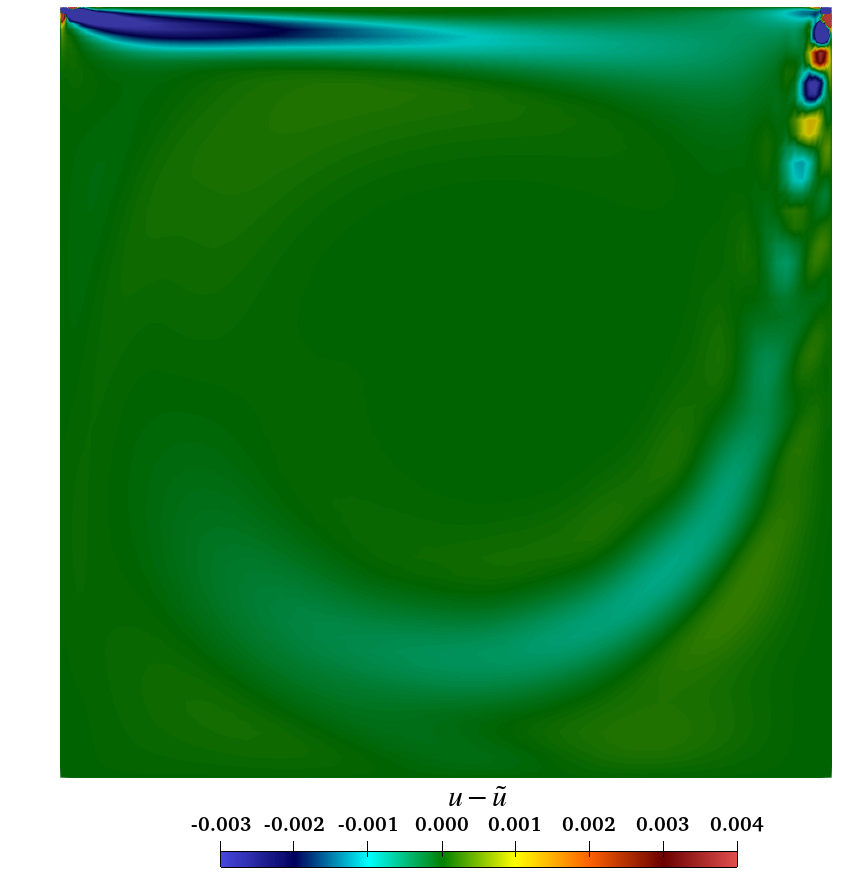}
\caption{$(\uvecth - \uvecSh)_x$ }
\end{subfigure}
\begin{subfigure}{0.4\textwidth}
    \centering
\includegraphics[width=\textwidth,trim={0.0cm 3.5cm 0.0cm 0.0cm},clip]{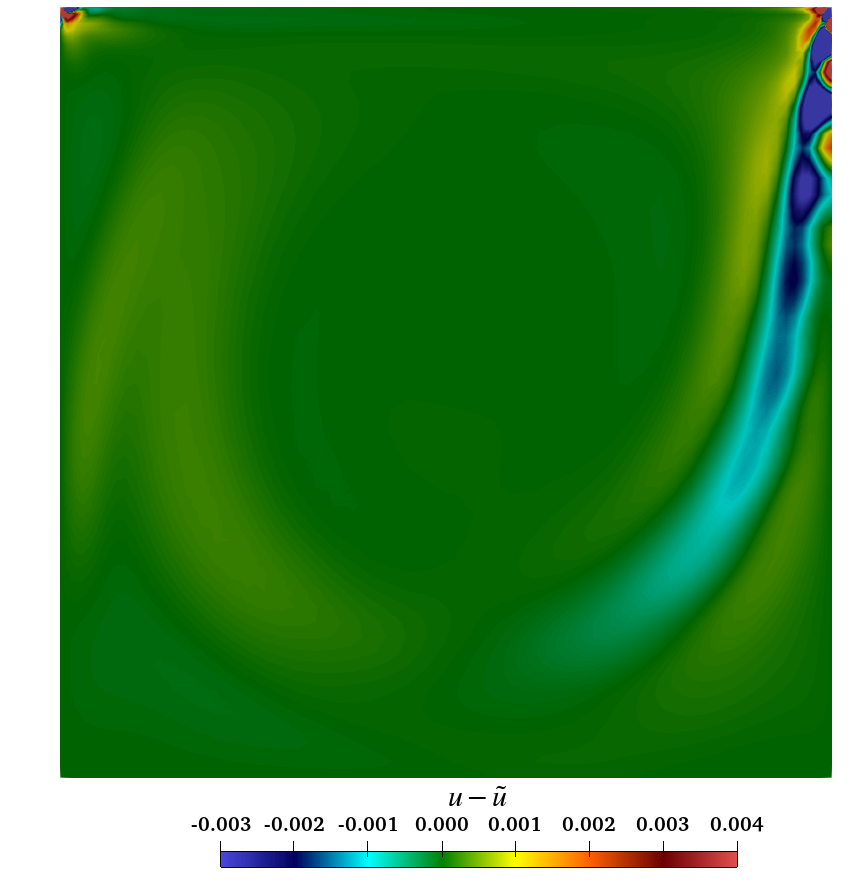}
\caption{$(\uvecth - \uvecSh)_y$ }
\end{subfigure}
\\
\vspace{0.3cm}
\begin{subfigure}{\textwidth}
    \centering
    \begin{tikzpicture}
        \begin{axis}[
            hide axis,
            scale only axis,
            height=0pt,
            width=0pt,
            colormap={rainbowdesat}{
                rgb=(0.278, 0.278, 0.859)
                rgb=(0.000, 0.000, 0.360)
                rgb=(0.000, 1.000, 1.000)
                rgb=(0.000, 0.502, 0.000)
                rgb=(1.000, 1.000, 0.000)
                rgb=(1.000, 0.380, 0.380)
                rgb=(0.420, 0.000, 0.000)
            },
            point meta min=-0.003,
            point meta max=0.004,
            colorbar horizontal,
            colorbar style={
                width=8cm,
                height=0.35cm,
                xtick={-0.003,-0.002,-0.001,0.0,0.001,0.002,0.003,0.004},
                scaled x ticks=false,
                xticklabel style={font=\small,
                                  /pgf/number format/.cd,
                                  fixed,
                                  fixed zerofill,
                                  precision=3,
                                  /tikz/.cd},
                title style={font=\small, yshift=-1ex},
            },
        ]
        \addplot [draw=none] coordinates {(0,0) (1,1)};
        \end{axis}
    \end{tikzpicture}
\end{subfigure}
\caption{Velocity correction applied by the projection step for the lid-driven cavity at $Re=1000$ at steady state. (a) $x$-component, $(\uvectN - \uvecSN)_x$. (b) $y$-component, $(\uvectN - \uvecSN)_y$. Both components are largest near the singular top corners of the cavity.}
\label{fig:velocitycorrection}
\end{figure}

\figref{fig:divergenceofvelocity} compares the divergence of the three discrete velocities --- (i) the monolithic velocity solution $\uvech$, (ii) the projection-based intermediate velocity $\uvecth$ and (iii) the projection-based updated velocity $\uvecSh$ --- at $Re=10^3$. The divergence of $\uvech$ is small throughout the cavity, with higher magnitudes prominent near the singular top corners where the moving lid meets the no-slip walls, and smaller magnitudes along the walls and inside the recirculation regions. The divergence of the 
intermediate velocity $\uvecth$ from the projection scheme exhibits a similar pattern but there are some visible differences, since solenoidality is not enforced at the
predictor stage. After the $L^2$-projection step, the corrected velocity
$\uvecSh$ has a divergence that is comparable in
pattern to that of the monolithic velocity, confirming that the projection step
recovers weak incompressibility at the same level as the monolithic
formulation. \figref{fig:velocitycorrection} shows the components
of the velocity correction $(\uvecth - \uvecSh)$ for this case. The corrections are seen to be the most pronounced near the corner singularities but also exhibit smaller, distributed contributions along the cavity walls and within the primary vortex.

\subsection{Flow past cylinder}

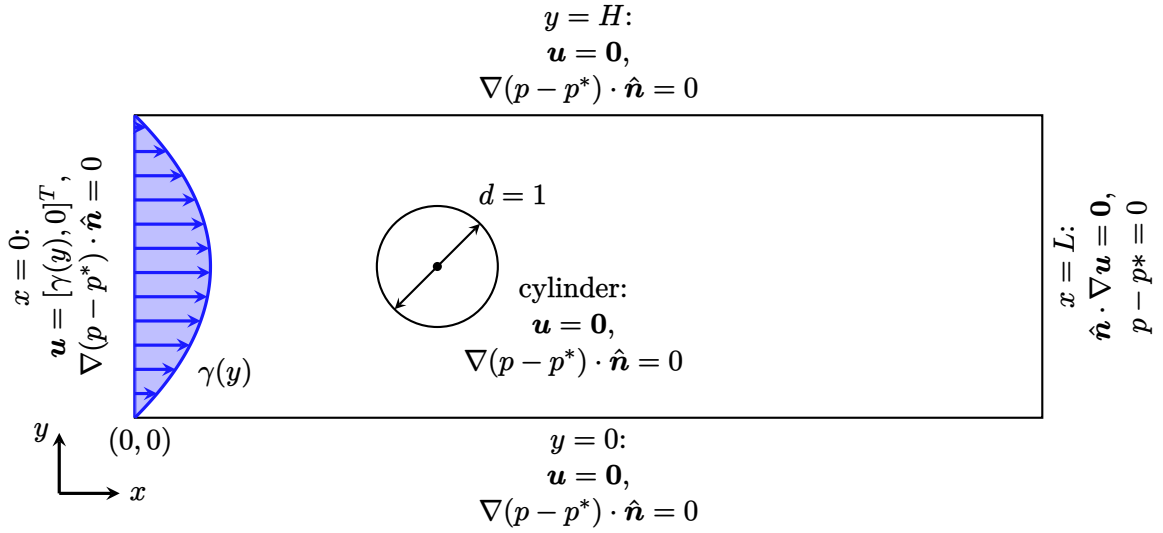
\begin{figure}[!htb]
\centering
\begin{tikzpicture}[scale=4]

\draw[thick] (0,0) rectangle (3,1);

\node at (1.5,-0.20)[align=center] {$y=0$:\\ $\boldsymbol{u}=\boldsymbol{0}$,\\$\nabla(p-p^*)\cdot\hat{\boldsymbol{n}}=0$};   
\node at (1.5,1.20)[align=center] {$y=H$:\\ $\boldsymbol{u}=\boldsymbol{0}$,\\$\nabla(p-p^*)\cdot\hat{\boldsymbol{n}}=0$};    
\node at (-0.25,0.5) [rotate=90, align=center] {$x=0$:\\$\boldsymbol{u}=[\gamma(y),0]^T$,\\$\nabla(p-p^*)\cdot\hat{\boldsymbol{n}}=0$};   
\node at (3.2,0.5) [rotate=90, align=center] {$x=L$:\\ $\hat{\boldsymbol{n}}\cdot\nabla\boldsymbol{u} = \boldsymbol{0}$,\\$p-p*=0$};    

\draw[thick] (1.0,0.5) circle(0.2);

\fill (1.0,0.5) circle(0.015);

\node[anchor=north west,inner sep=1pt] at (-0.1,-0.02) {$(0,0)$};

\node at (1.45,0.3)[align=center] {cylinder:\\$\boldsymbol{u}=\boldsymbol{0}$,\\$\nabla(p-p^*)\cdot\hat{\boldsymbol{n}}=0$};

\draw[thick,->,>=stealth] (1.0,0.5) -- (1.0+0.141,0.5+0.141);
\draw[thick,->,>=stealth] (1.0,0.5) -- (1.0-0.141,0.5-0.141);

\node at (1.0+0.25,0.5+0.24) {$d=1$};

\def\xvert{0.0}

\draw[very thick,blue!90] (\xvert,0) -- (\xvert,1);

\fill[blue!50,opacity=0.5]
    (\xvert,0)
    \foreach \y in {0,0.01,...,1} {
        -- ({\xvert + 0.25*4*\y*(1-\y)},\y)
    }
    -- (\xvert,1) -- (\xvert,0);

\draw[very thick,blue!90]
    plot[domain=0:1,samples=100] 
    ({\xvert + 0.25*4*\x*(1-\x)},\x);

\foreach \y in {0.08,0.16,...,1} {
    \pgfmathsetmacro\xparabola{\xvert + 0.25*4*\y*(1-\y)}
    \draw[very thick,->,>=stealth,blue!90] (\xvert,\y) -- (\xparabola,\y);
}
\node at (0.3,0.15) [very thick]{$\gamma(y)$};

\def\axO{-0.25}
\def\axL{0.2}
\draw[very thick,->,>=stealth] (\axO,\axO) -- (\axO+\axL,\axO);
\draw[very thick,->,>=stealth] (\axO,\axO) -- (\axO,\axO+\axL);
\node at (\axO+\axL,\axO) [right] {$x$};
\node at (\axO,\axO+\axL) [left] {$y$};

\end{tikzpicture}
\caption{Schematic diagram of the flow past a circular cylinder in 2D. The inlet velocity profile $\gamma(y)$ is shown in blue curves and text.}
\label{fig:fpc2d-schematic}
\end{figure}

\begin{figure}[!htb]
\centering
\begin{subfigure}{0.49\textwidth}
    \centering
    \includegraphics[width=\textwidth,
    trim={0.0cm 11.5cm 9.0cm 11.0cm},clip]
    {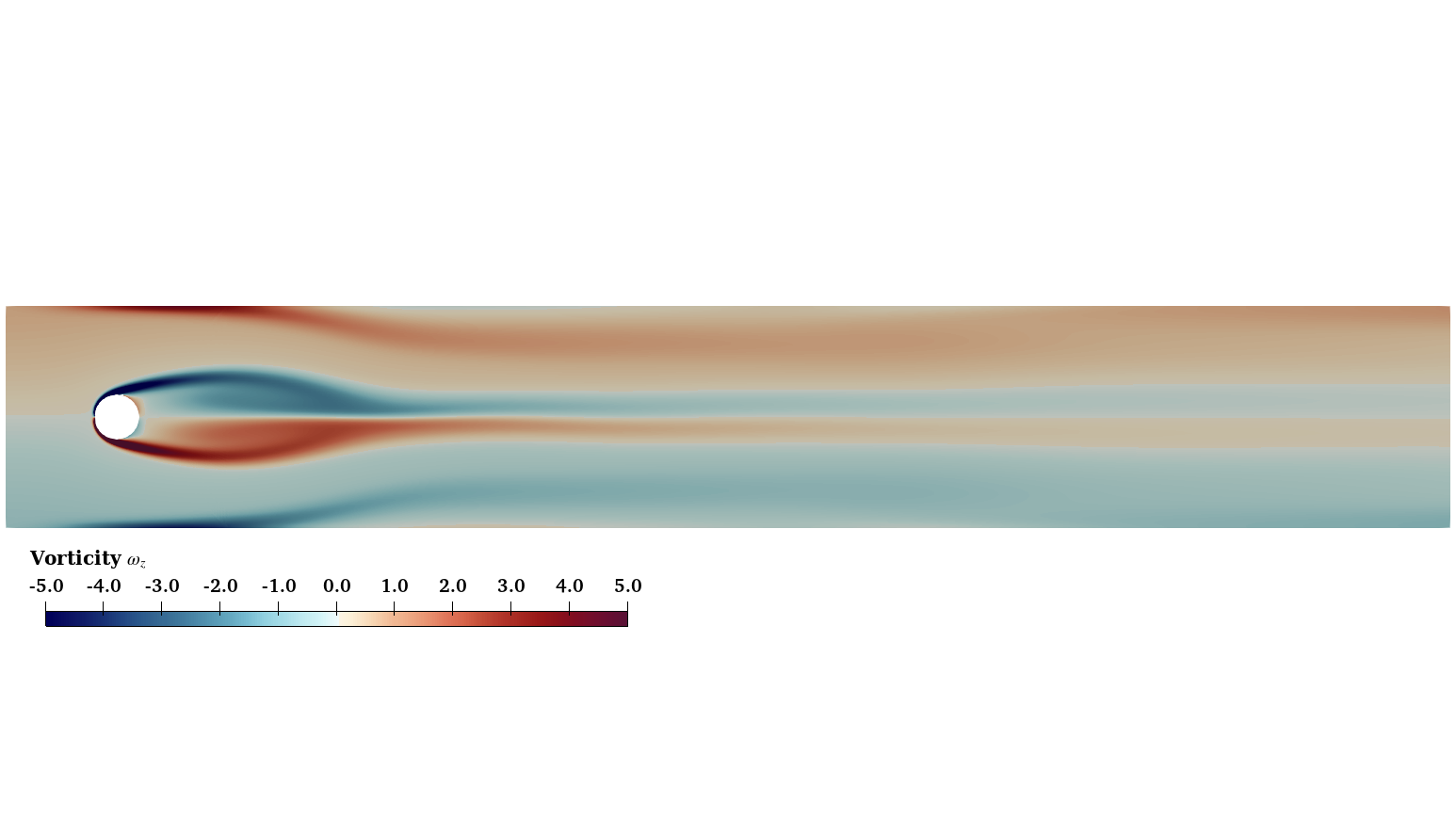}
    \caption{$t=50$}
\end{subfigure}
\hfill
\begin{subfigure}{0.49\textwidth}
    \centering
    \includegraphics[width=\textwidth,
    trim={0.0cm 11.5cm 9.0cm 11.0cm},clip]
    {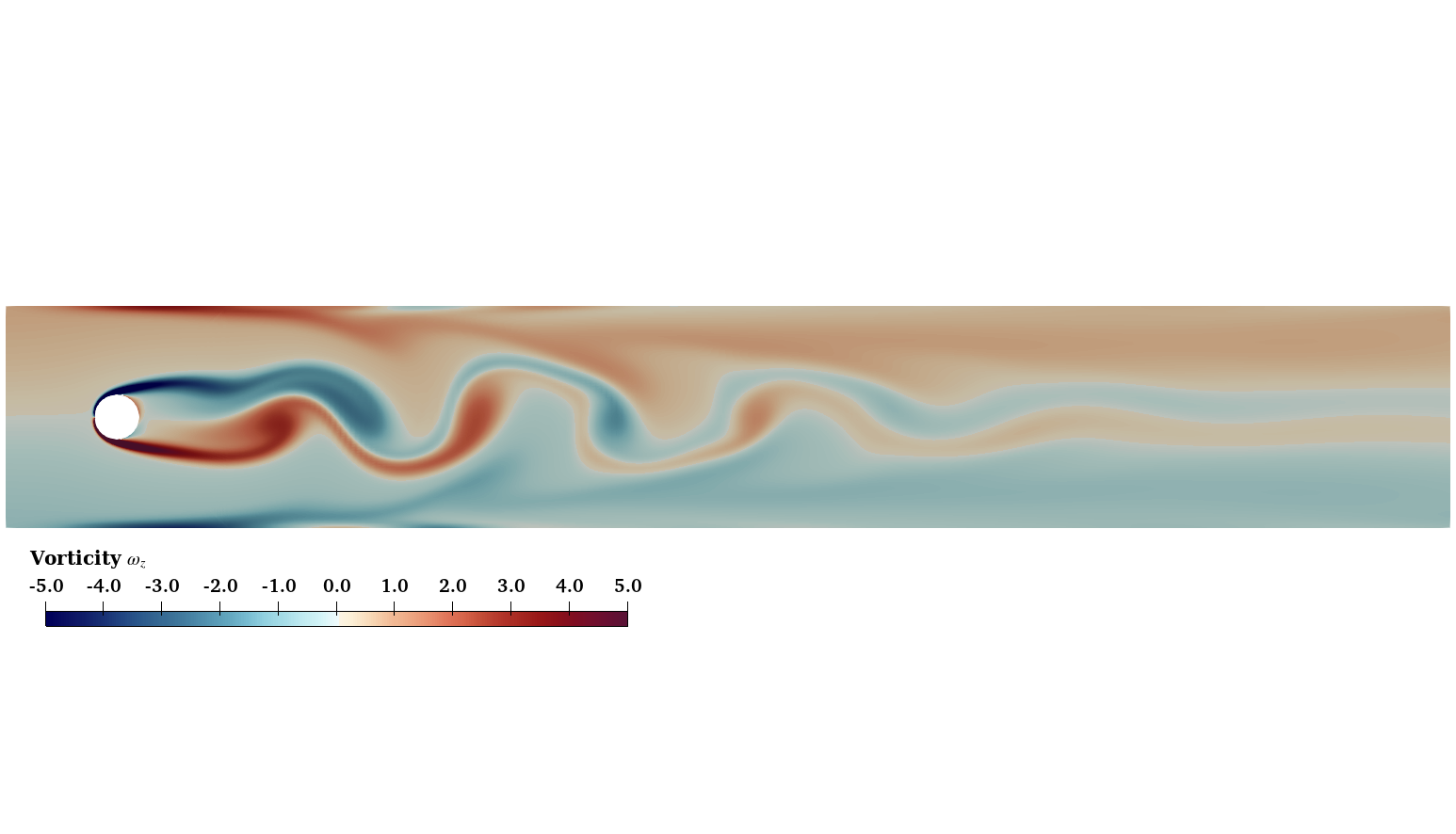}
    \caption{$t=70$}
\end{subfigure}

\vspace{0.3cm}

\begin{subfigure}{0.49\textwidth}
    \centering
    \includegraphics[width=\textwidth,
    trim={0.0cm 11.5cm 9.0cm 11.0cm},clip]
    {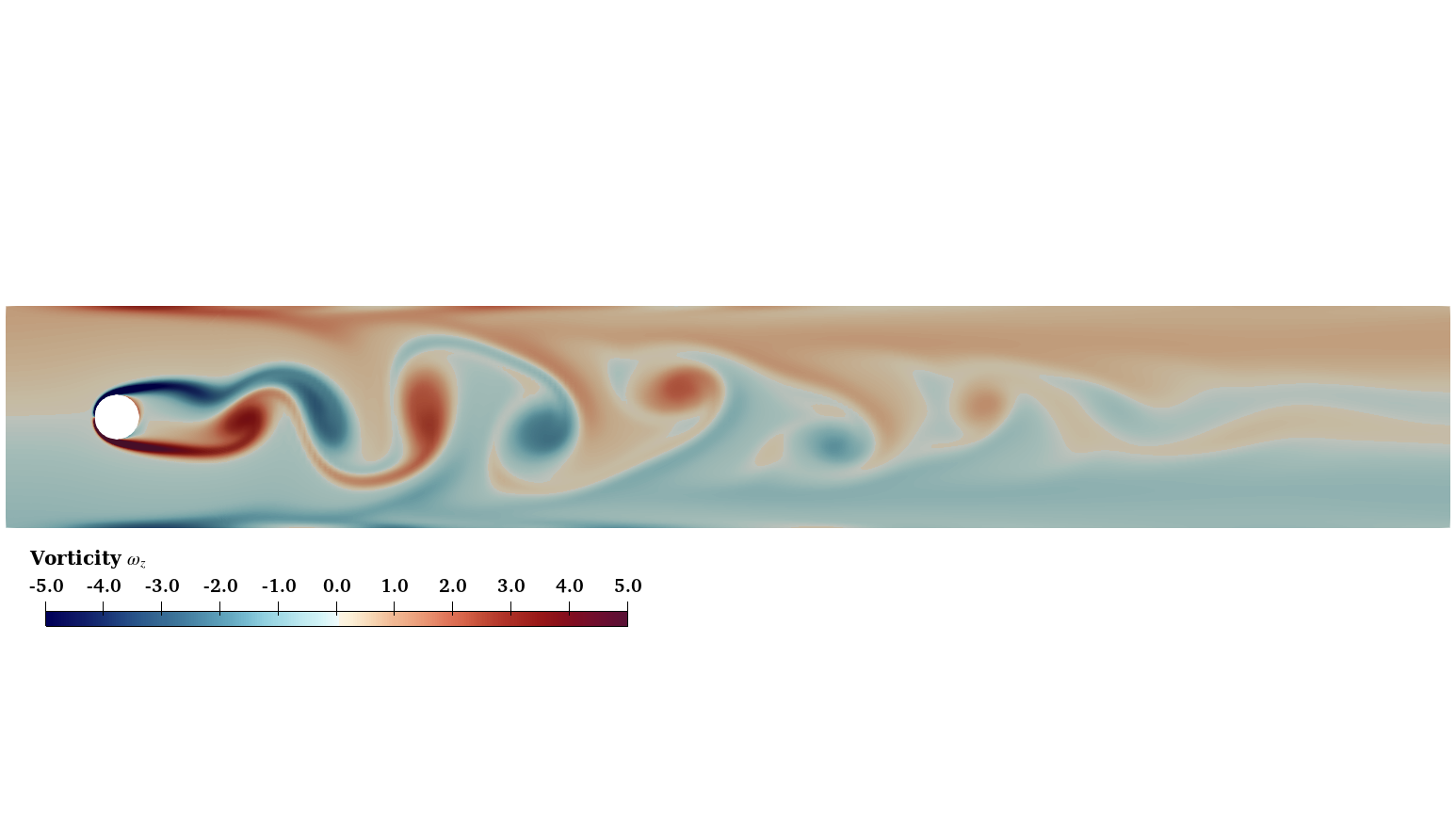}
    \caption{$t=75$}
\end{subfigure}
\hfill
\begin{subfigure}{0.49\textwidth}
    \centering
    \includegraphics[width=\textwidth,
    trim={0.0cm 11.5cm 9.0cm 11.0cm},clip]
    {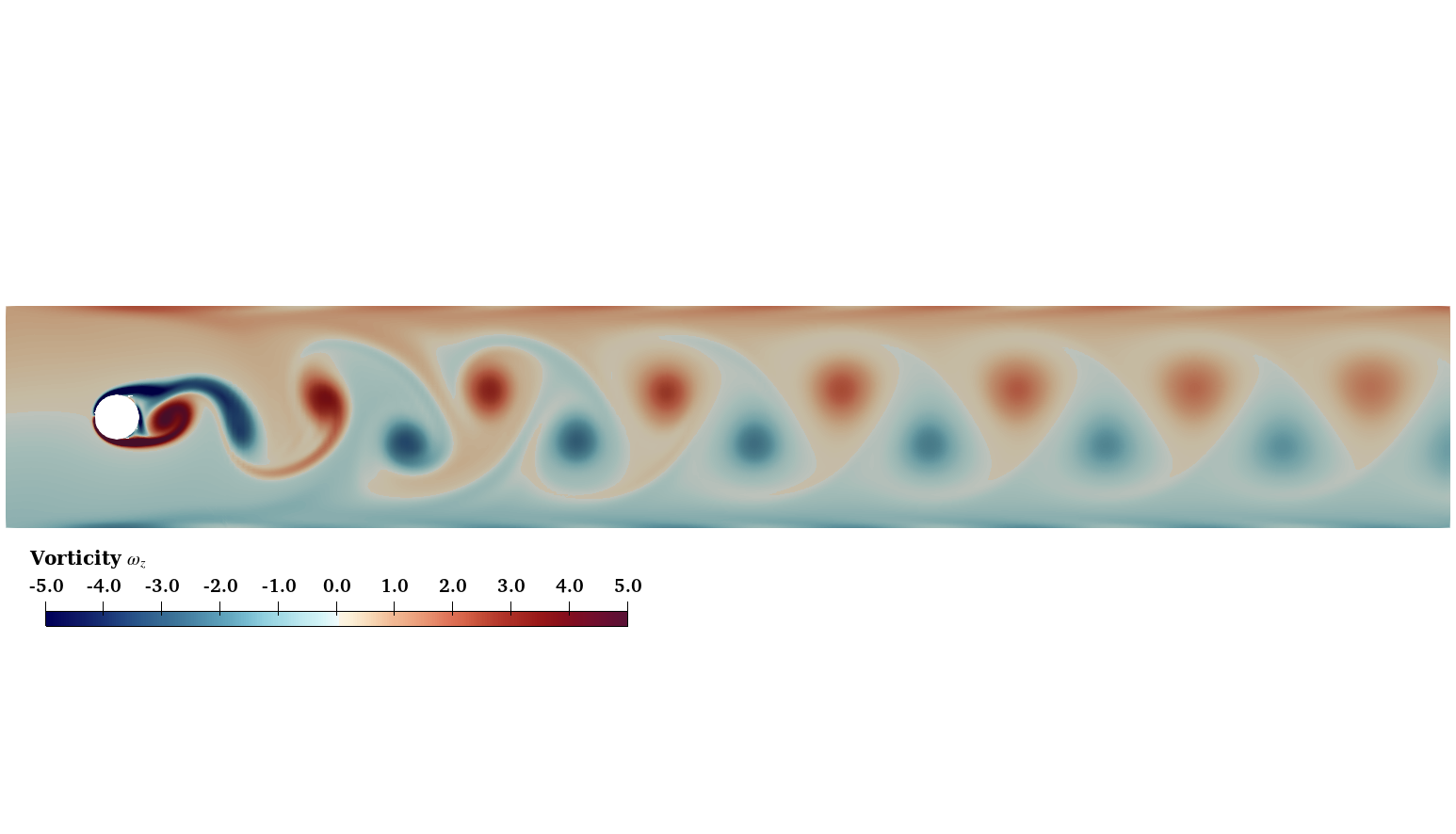}
    \caption{$t=150$}
\end{subfigure}

\vspace{0.3cm}


\begin{subfigure}{\textwidth}
    \centering
    \begin{tikzpicture}
        \begin{axis}[
            hide axis,
            scale only axis,
            height=0pt,
            width=0pt,
            colormap={RdBu}{
                rgb=(0.020,0.188,0.380)
                rgb=(0.129,0.400,0.674)
                rgb=(0.404,0.663,0.812)
                rgb=(0.698,0.847,0.910)
                rgb=(0.969,0.969,0.969)
                rgb=(0.992,0.859,0.780)
                rgb=(0.957,0.647,0.510)
                rgb=(0.839,0.376,0.302)
                rgb=(0.647,0.000,0.149)
            },
            point meta min=-5.0,
            point meta max=5.0,
            colorbar horizontal,
            colorbar style={
                width=8cm,
                height=0.35cm,
                xtick={-5,-4,-3,-2,-1,0,1,2,3,4,5},
                xticklabel style={font=\small},
                title={Vorticity $\omega_z$},
                title style={font=\small, yshift=-1ex},
            },
        ]
        \addplot [draw=none] coordinates {(0,0) (1,1)};
        \end{axis}
    \end{tikzpicture}
\end{subfigure}
\caption{Instantaneous vorticity contours for two-dimensional flow past a circular cylinder at Reynolds number $Re=300$. The evolution of the wake shows the onset and development of periodic vortex shedding and the formation of a von Kármán vortex street.}
\label{fig:four_vertical}
\end{figure}











Next, we test our method on the problem of two-dimensional flow past a circular 
cylinder. The setup is taken from \citet{schafer1996benchmark,kanaris2011three,zovatto2001flow}. The schematic of the problem is shown in
\figref{fig:fpc2d-schematic}. The spatial domain is $ \spDom = ((I_x\times I_y) \backslash C) $, where $ I_x = [0,L],\ I_y = [0,H] $ and $ C $ denotes a circle of diameter $ d=1 $ with center at $ (x_c,y_c) $. The length and width of the channel are given by $L=20$, $H=5$. The blockage ratio is $\nicefrac{d}{H} = 0.2$. The cylinder center is located at $x_c=3$ and $y_c=\tfrac{H}{2}+\tfrac{d}{10}=2.6$, placing it slightly above the channel centerline and thus asymmetrically between the bottom and top walls.
The boundary and initial conditions are given as:

\begin{subequations}
	\label{eq:fpc-boundary-conditions}
	\begin{align}
		\uvec &= [\gamma(y), 0]^T\ \text{on}\ x=0,\\
		\uvec &= [0,0]^T\ \text{on}\ y = 0,\ H,\\
		\uvec &= [0,0]^T\ \text{on cylinder}\ C,\\
		p-p* &= 0\ \text{on}\ x=L,\\
		\normalvec\cdot\grad\uvec &= \zerovec\ \text{on}\ x=L.
	\end{align}
\end{subequations}

where the inlet flow profile $ \gamma(y) $ is given by
\begin{align}\label{eq:fpc-inlet-vel}
	\gamma(y) = 4 u_c y (H-y) / H^2,
\end{align}
which is a parabolic profile with $ \gamma(0) = \gamma(H)=0 $ and $ \gamma_m := \gamma(\nicefrac{H}{2}) = u_c $. The average inlet velocity is given by $ \bar{\gamma} = \nicefrac{2}{3} u_c $. The Reynolds number for this problem is defined as $ Re = \frac{U_{ref} d}{\nu}$. $ U_{ref} = \gamma_m = u_c $. The force on the cylinder in the direction $ \wvec $ is expressed as
\begin{align}
	F_w = \int_{\Gamma_C} \left[ \left( -p \mmat{I} + \visco [\grad \uvec + \grad \uvec^T] \right) \cdot \normalvec \right]\cdot \wvec \ dS.
\end{align}
We set $ \wvec = (1,0) $ to extract the drag force, and $ \wvec = (0,1) $ to extract the lift force. The drag coefficient $C_d$ and the lift coefficient $C_l$ are then calculated as
\begin{align}
	C_d &= \frac{2 F_{\mathrm{drag}}}{AU^2} = 2 F_{\mathrm{drag}}, \\
	C_l &= \frac{2 F_{\mathrm{lift}}}{AU^2} = 2 F_{\mathrm{lift}}
\end{align}
where $A=1$ and $U=1$ are the reference area and speed respectively.

After a sufficient time is passed, the solution of this example is known to reach either a steady state or a periodic state depending on whether the Reynolds number is below or above a critical value. When the solution achieves a periodic state, so do the drag and lift forces calculated above. Suppose $T_p$ is the time period of the lift coefficient $C_l$ obtained from the numerical simulation. Then the corresponding Strouhal number is calculated as
\begin{align}
	St = \frac{L}{UT_p} = \frac{1}{T_p}
\end{align}
where $L=1$ and $U=1$ are the reference length and the speed respectively.

We solve this problem numerically on an unstructured triangular mesh of equal-order $P_1/P_1$ elements (with 57,170 triangles, denoted `M1') for various Reynolds numbers between $Re=100$ to $Re=300$. \figref{fig:four_vertical} shows the wake evolution at ${Re}=300$.
At $t=50$ the wake is symmetric. By $t=70$ the symmetry breaks and
alternating vortices begin forming near the cylinder. At $t=75$ these
vortices detach and advect downstream, and by $t=150$ the flow has
reached a periodic von Kármán vortex street.

We also compute the drag coefficient $C_d$ and the Strouhal number $St$ for each $Re =$ 100, 150, 200, 250, 300. These values are plotted against $Re$ in \figref{fig:fpc2d-validation}. Also shown are results from an inf-sup stable solver ($P_2/P_1$ with SUPG stabilization), and reference values from the literature~\cite{kanaris2011three, zovatto2001flow, singha2010flow}. The results from the proposed VMS-stabilized projection method match well with these values.

\paragraph{Monolithic vs. projection.}
It can be noted that in \figref{fig:fpc2d-validation}, the $C_d$ values from the standard monolithic VMS solver are higher (by about 25\%) compared to the projection solver. These values can be improved by considering a more refined mesh (denoted `M2', with 282,812 triangles), as shown in the plots.

Finally, the major difference between these two methods (apart from the splitting) is that the proposed VMS/projection scheme only decomposes $\uvect$, whereas the general monolithic solver decomposes both $\uvec$ and $p$. When decomposed, the small scales of pressure are approximated by $\pprime \approx \tauc \divergence \uvech$, which introduces the ``least-squares'' (also, grad-div)-like term $\inner{\tauc \divergence \uvech}{\divergence \vvech}$ in the momentum equations (see \eqref{eq:vms-finescale-model-overview-c} and \eqref{eq:vms-coarse-residuals-overview-c}). Now, if $p$ is left undecomposed (i.e., $\pprime=0$) in the monolithic solver, one obtains $C_d$ values that are closer to that from the VMS/projection solver (shown in the plot). This is also shown in \figref{fig:fpc2d-validation}.

In conclusion, the pressure decomposition in the monolithic solver leads to a overprediction of $C_d$ on the same mesh, which can be improved by considering a mesh with a higher resolution. Discarding the pressure fine scale also leads to more accurate $C_d$ and $St$ values, however, this may result in poorer divergence errors, as shown in the next example.



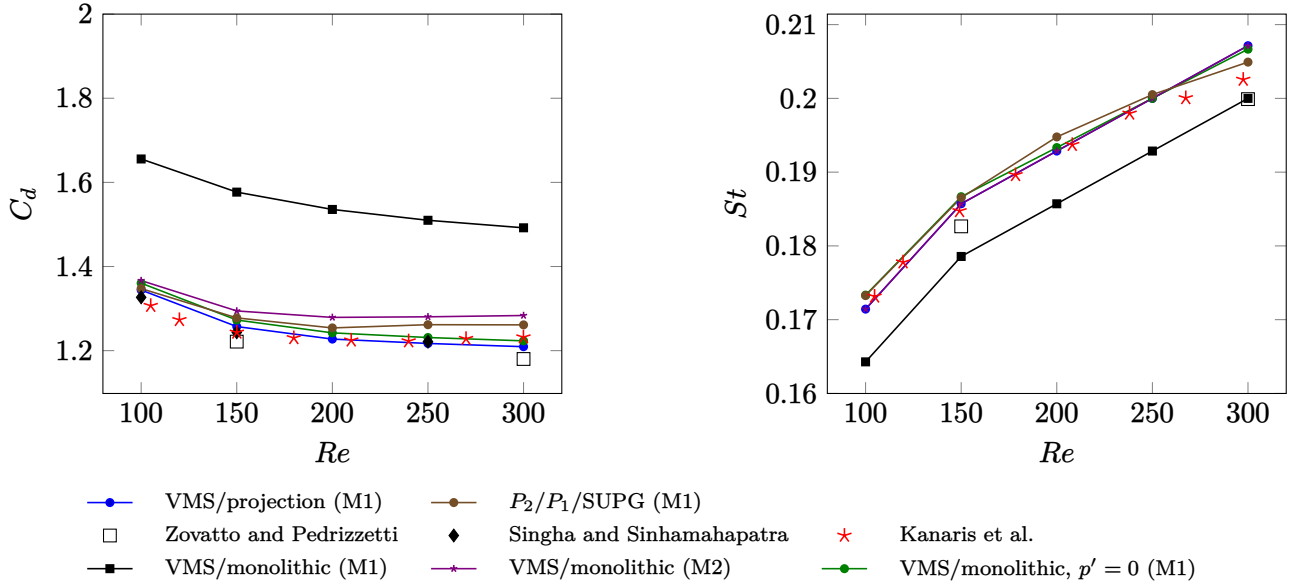
\begin{figure}[t]
	\centering
	\begin{tikzpicture}
		\begin{axis}[
			width=0.45\textwidth,
			scaled y ticks=true,
			xlabel={$Re$},
			ylabel={$C_d$},
			ymax=2.0,
			restrict x to domain=100:300,
			unbounded coords=discard,
			]

			\addplot+[semithick, mark options={scale=0.7, solid}] table [x expr=\thisrowno{0},y expr=\thisrowno{1},col sep=comma] {data/fpcKanaris/forces_projection_nonlinear_trimesh.txt};

			\addplot+[semithick, color=black, mark options={scale=0.7, solid}] table [x expr=\thisrowno{0},y expr=\thisrowno{1},col sep=comma] {data/fpcKanaris/forces_monolithic_nonlinear_trimesh.txt};

			\addplot+[semithick, color=green!50!black, mark options={scale=0.7, solid}] table [x expr=\thisrowno{0},y expr=\thisrowno{1},col sep=comma] {data/fpcKanaris/forces_monolithic_nonlinear_trimesh_s0_tauc_0.txt};

			\addplot+[semithick, color=violet, mark options={scale=0.7, solid}] table [x expr=\thisrowno{0},y expr=\thisrowno{1},col sep=comma] {data/fpcKanaris/forces_monolithic_nonlinear_trimesh_s0_refined.txt};

			\addplot+[semithick, color=brown!60!black, mark=*, mark options={scale=0.7, solid}]
			table [x=Re, y=meanCd, col sep=comma] {data/fpcKanaris/all.csv};

			\addplot+[only marks, mark=diamond*, mark options={scale=1.2, solid}, color=black]
			table [x expr=\thisrowno{0}, y expr=\thisrowno{1}, col sep=comma]
			{data/fpcKanaris/reference/singha-drag-data.txt};

			\addplot+[only marks, mark=square, mark options={scale=1.2, solid}, color=black]
			table [x expr=\thisrowno{1}, y expr=\thisrowno{2}, col sep=comma]
			{data/fpcKanaris/reference/zovatto-drag-data.txt};

			\addplot+[only marks, mark=star, mark options={semithick, scale=1.4, solid}, color=red]
			table [x expr=\thisrowno{0}, y expr=\thisrowno{1}, col sep=comma]
			{data/fpcKanaris/reference/kanaris-drag-data.txt};
		\end{axis}
	\end{tikzpicture}
	\hfill
	\begin{tikzpicture}
		\begin{axis}[
			width=0.45\textwidth,
			scaled y ticks=true,
			xlabel={$Re$},
			ylabel={$St$},
			restrict x to domain=100:300,
			unbounded coords=discard,
			]

			\addplot+[semithick, mark options={scale=0.7, solid}] table [x expr=\thisrowno{0},y expr=\thisrowno{3},col sep=comma] {data/fpcKanaris/forces_projection_nonlinear_trimesh.txt};

			\addplot+[semithick, color=black, mark options={scale=0.7, solid}] table [x expr=\thisrowno{0},y expr=\thisrowno{3},col sep=comma] {data/fpcKanaris/forces_monolithic_nonlinear_trimesh_s0.txt};

			\addplot+[semithick, color=green!50!black, mark options={scale=0.7, solid}] table [x expr=\thisrowno{0},y expr=\thisrowno{3},col sep=comma] {data/fpcKanaris/forces_monolithic_nonlinear_trimesh_s0_tauc_0.txt};

			\addplot+[semithick, color=violet, mark options={scale=0.7, solid}] table [x expr=\thisrowno{0},y expr=\thisrowno{3},col sep=comma] {data/fpcKanaris/forces_monolithic_nonlinear_trimesh_s0_refined.txt};

			\addplot+[semithick, color=brown!60!black, mark=*, mark options={scale=0.7, solid}]
			table [x=Re, y=St, col sep=comma] {data/fpcKanaris/all.csv};

			\addplot+[only marks, mark=square, mark options={scale=1.2, solid}, color=black]
			table [x expr=\thisrowno{1}, y expr=\thisrowno{3}, col sep=comma]
			{data/fpcKanaris/reference/zovatto-drag-data.txt};

			\addplot+[only marks, mark=star, mark options={semithick, scale=1.4, solid}, color=red]
			table [x expr=\thisrowno{0}, y expr=\thisrowno{1}, col sep=comma]
			{data/fpcKanaris/reference/kanaris-st-data.txt};
		\end{axis}
	\end{tikzpicture}

	\par\vspace{0.1cm}
	\begin{tikzpicture}
		\begin{axis}[
			hide axis,
			xmin=0, xmax=1, ymin=0, ymax=1,
			width=\linewidth,
			height=1.8cm,
			legend style={
				at={(0.5,1)},
				anchor=north,
				draw=none,
				fill=none,
				font=\scriptsize,
				legend columns=3,
				column sep=10pt,
				row sep=1pt,
			},
			legend cell align=left,
			]
			\addlegendimage{semithick, blue, mark=*, mark options={scale=0.7, solid, fill=blue!80!black}}
			\addlegendentry{VMS/projection (M1)}
			\addlegendimage{semithick, brown!60!black, mark=*, mark options={scale=0.7, solid}}
			\addlegendentry{$P_2/P_1$/{SUPG} (M1)}
			\addlegendimage{empty legend}
			\addlegendentry{}
			\addlegendimage{only marks, mark=square, mark options={scale=1.2, solid}, black}
			\addlegendentry{Zovatto	and Pedrizzetti}
			\addlegendimage{only marks, mark=diamond*, mark options={scale=1.2, solid}, black}
			\addlegendentry{Singha and Sinhamahapatra}
			\addlegendimage{only marks, mark=star, mark options={semithick, scale=1.4, solid}, red}
			\addlegendentry{Kanaris et al.}
			\addlegendimage{semithick, black, mark=square*, mark options={scale=0.7, solid}}
			\addlegendentry{VMS/monolithic (M1)}
			\addlegendimage{semithick, violet, mark=star, mark options={scale=0.7, solid}}
			\addlegendentry{VMS/monolithic (M2)}
			\addlegendimage{semithick, green!50!black, mark=otimes*, mark options={scale=0.7, solid}}
			\addlegendentry{VMS/monolithic, $\pprime = 0$ (M1)}
		\end{axis}
	\end{tikzpicture}

	\caption{{Results for flow past a circular cylinder in 2D:} drag coefficient $C_d$ (left) and Strouhal number $St$ (right). Mesh M1 has 57,170 triangles, and M2 has 282,812 triangles. The VMS/projection and VMS/monolithic results are obtained on $P_1/P_1$ elements. The reference values are taken from~\citet{zovatto2001flow,singha2010flow,kanaris2011three}.}
	\label{fig:fpc2d-validation}
\end{figure}

\subsection{Taylor-Green vortex}

\begin{figure}[!htb]
\centering
\begin{minipage}{0.32\textwidth}
\centering
\includegraphics[width=\textwidth,trim={0.0cm 0.0cm 0.0cm 0.0cm},clip]{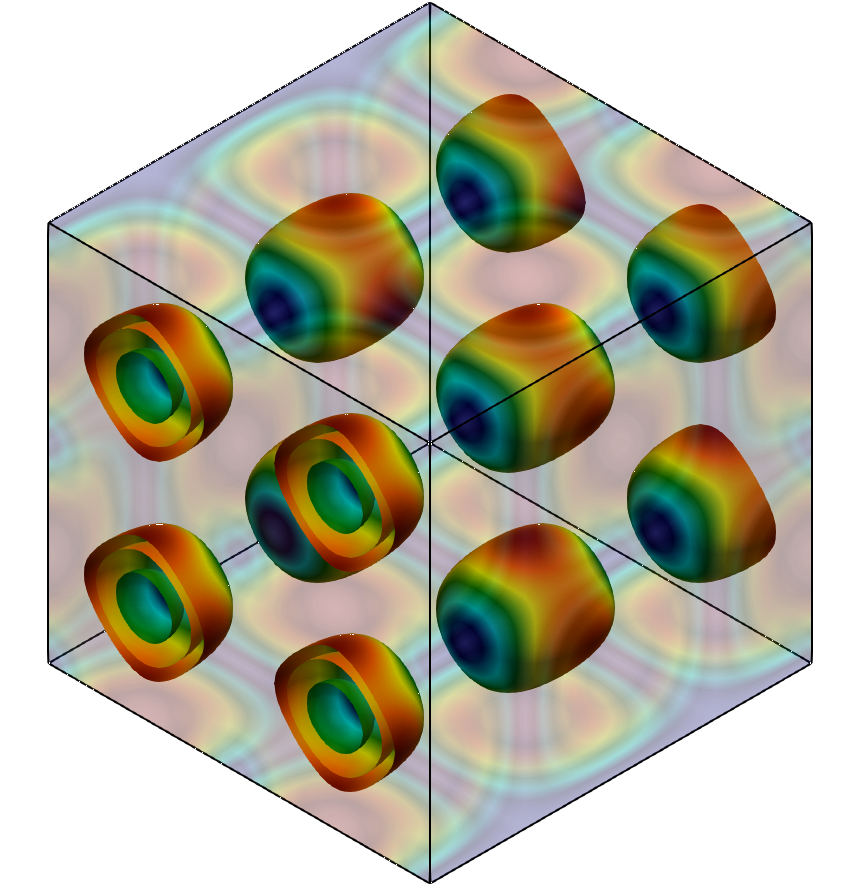}
\end{minipage}
\begin{minipage}{0.32\textwidth}
\centering
\includegraphics[width=\textwidth,trim={0.0cm 0.0cm 0.0cm 0.0cm},clip]{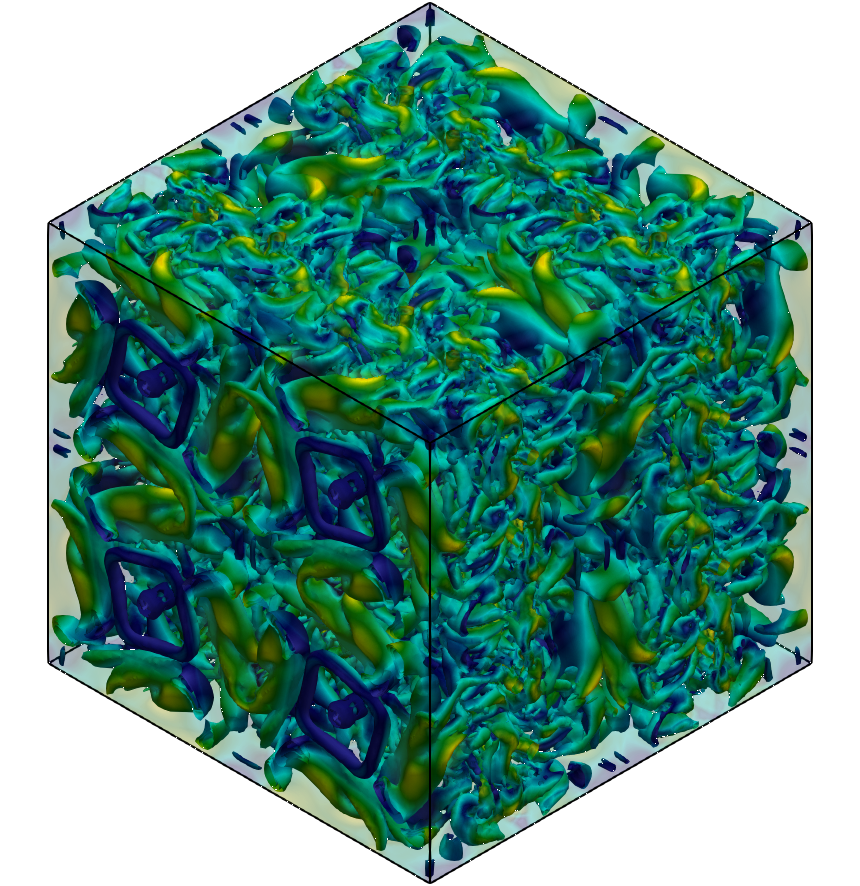}
\end{minipage}
\begin{minipage}{0.32\textwidth}
\centering
\includegraphics[width=\textwidth,trim={0.0cm 0.0cm 0.0cm 0.0cm},clip]{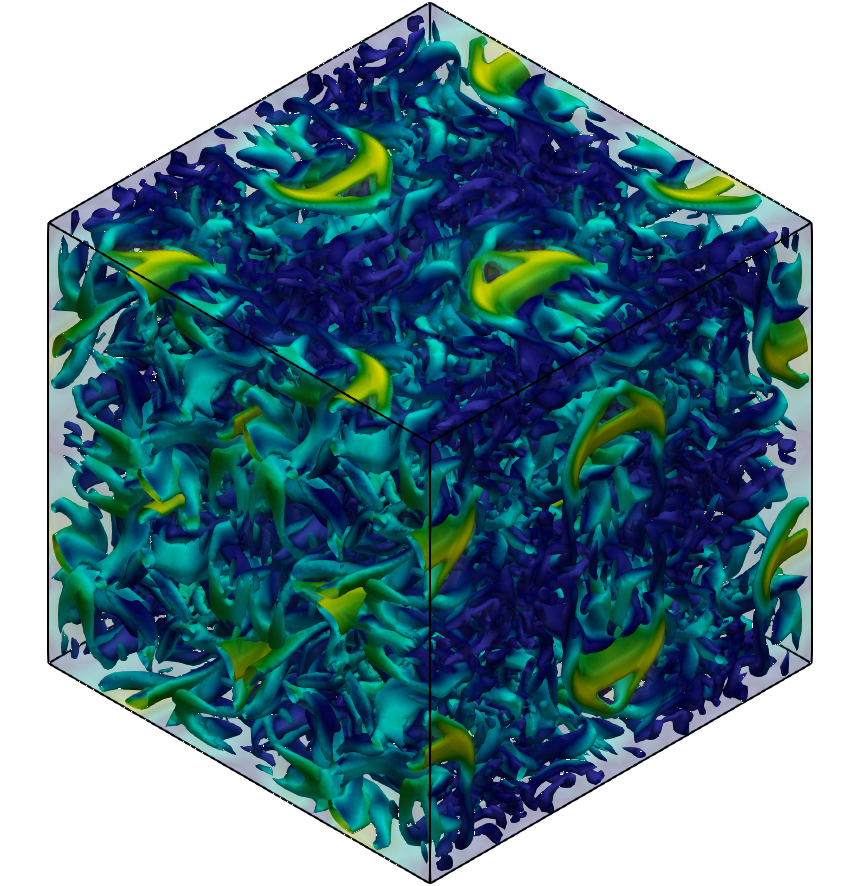}
\end{minipage}
\vspace{0.3cm}


\begin{subfigure}{\textwidth}
    \centering
    \begin{tikzpicture}
        \begin{axis}[
            hide axis,
            scale only axis,
            height=0pt,
            width=0pt,
            colormap={rainbowdesat}{
                rgb=(0.278, 0.278, 0.859)
                rgb=(0.000, 0.000, 0.360)
                rgb=(0.000, 1.000, 1.000)
                rgb=(0.000, 0.502, 0.000)
                rgb=(1.000, 1.000, 0.000)
                rgb=(1.000, 0.380, 0.380)
                rgb=(0.420, 0.000, 0.000)
            },
            point meta min=0.0,
            point meta max=1.0,
            colorbar horizontal,
            colorbar style={
                width=8cm,
                height=0.35cm,
                xtick={0.0,0.1,0.2,0.3,0.4,0.5,0.6,0.7,0.8,0.9,1.0},
                xticklabel style={font=\small},
                title={$\|\bm{u}\|$},
                title style={font=\small, yshift=-1ex},
            },
        ]
        \addplot [draw=none] coordinates {(0,0) (1,1)};
        \end{axis}
    \end{tikzpicture}
\end{subfigure}

\caption{Taylor-Green vortex at $\mathrm{Re}=1600$ from the
projection-based RBVMS scheme on the $128^3$ mesh, visualized using
Q-criterion iso-surfaces at $Q=0.3,\,0.5,\,0.8$,
coloured by velocity magnitude.
(Left) Initial condition at $t=0$, showing the organized cellular
structure. (Centre) Developed turbulent state at $t=15$, after the
breakdown of the initial vortex sheets. (Right) Decaying turbulent state
at $t=20$, showing the small-scale structures that persist late in the
simulation.}
\label{fig:taylorgreenqcriterion}
\end{figure}

The Taylor-Green vortex is a canonical problem used to study the decay of vortical structures in incompressible viscous flows. The computational domain is a cubic box defined as $\Omega = {[-\pi, \pi]}^3$ with periodic boundary conditions in all directions. The initial velocity and pressure fields are prescribed as
\[
u_x = \sin(x)\cos(y)\cos(z), \quad
u_y = -\cos(x)\sin(y)\cos(z), \quad
u_z = 0,
\]
\[
p = \frac{1}{16} \left[ \cos(2x) + \cos(2y) \right] \left[ \cos(2z) + 2 \right].
\]
These fields satisfy the incompressibility condition $\nabla \cdot \bm{u} = 0$. The flow evolves under the incompressible Navier-Stokes equations, leading to the gradual dissipation of kinetic energy due to viscous effects while maintaining the symmetry of the initial vortex structures. The Reynolds number is set to 1600. The volume-averaged kinetic energy is
\[
E_k(t) = \frac{1}{|\Omega|} \int_{\Omega}
\frac{\bm{u} \cdot \bm{u}}{2} \, d\Omega,
\]
and the volume-averaged enstrophy is
\[
\mathcal{E}(t) = \frac{1}{|\Omega|} \int_{\Omega}
\frac{\boldsymbol{\omega} \cdot \boldsymbol{\omega}}{2} \, d\Omega,
\qquad \boldsymbol{\omega} = \nabla \times \bm{u}.
\]
From these, we define two diagnostics of the kinetic-energy decay
rate. The \emph{energy-based} dissipation
\[
\varepsilon_k(t) \;\coloneq\; -\frac{dE_k}{dt},
\]
measures the instantaneous loss of resolved kinetic energy. And the
\emph{enstrophy-based} dissipation
\[
\varepsilon_\Omega(t) \;\coloneq \; \frac{2}{\mathrm{Re}} \, \mathcal{E}(t),
\]
is the viscous dissipation rate computed directly from the resolved
vorticity field. For an incompressible flow in a periodic domain, the
energy balance applied to the Navier-Stokes equations yields the
identity~\citep{pope2000turbulent}
\begin{equation}
\varepsilon_k(t) = \varepsilon_\Omega(t).
\label{eq:energy_enstrophy_balance}
\end{equation}
In an under-resolved simulation the two diverge: $\varepsilon_k$
captures the total kinetic-energy loss, including the contribution of
the unresolved scales, while $\varepsilon_\Omega$ accounts only for
the resolved viscous part. The gap $\varepsilon_k - \varepsilon_\Omega$
therefore quantifies the dissipation supplied by the subgrid model
and the numerical scheme.

\begin{figure}[!htb]
  \centering

  \begin{minipage}{0.49\textwidth}
    \centering
    \begin{tikzpicture}
      \begin{axis}[
          width=0.95\linewidth,
          xlabel={$t$},
          ylabel={Kinetic energy $E_k(t)$},
          xmin=0, xmax=20,
          ymin=0, ymax=0.14,
          ytick distance=0.02,
          xtick distance=2,
          scaled y ticks=false,
          yticklabel style={
              /pgf/number format/fixed,
              /pgf/number format/precision=3,
              /pgf/number format/fixed zerofill,
          },
          grid style={dashed, gray!30},
          tick label style={font=\footnotesize},
          label style={font=\small},
          legend to name=tgvmeshlegend,
          legend columns=-1,
          legend style={
            draw=none, fill=none,
            font=\small,
            /tikz/every even column/.append style={column sep=0.6cm},
          },
        ]
        \addplot[black, only marks, mark=o, mark size=2pt, each nth point=80]
            table[col sep=space, skip first n=1, x index=0, y index=1]
            {data/taylorgreen/wang_2013_spectral.txt};
        \addlegendentry{Wang et al., 2013 ($512^3$)}

        \addplot[blue, line width=1.5pt, no marks, each nth point=2]
            table[col sep=space, skip first n=1, x index=0, y index=1]
            {data/taylorgreen/monolithic_256X256.txt};
        \addlegendentry{$256^3$}

        \addplot[red, line width=1.5pt, dashed, no marks, each nth point=2]
            table[col sep=space, skip first n=1, x index=0, y index=1]
            {data/taylorgreen/monolithic_128X128.txt};
        \addlegendentry{$128^3$}

        \addplot[black, line width=1.5pt, densely dotted, no marks, each nth point=2]
            table[col sep=space, skip first n=1, x index=0, y index=1]
            {data/taylorgreen/monolithic_64X64.txt};
        \addlegendentry{$64^3$}
      \end{axis}
    \end{tikzpicture}
    \subcaption{Monolithic: $E_k(t)$}
    \label{fig:tgv_mono_Ek}
  \end{minipage}
  \hfill
  \begin{minipage}{0.49\textwidth}
    \centering
    \begin{tikzpicture}
      \begin{axis}[
          width=0.95\linewidth,
          xlabel={$t$},
          ylabel={Dissipation rate $\varepsilon_k(t)$},
          xmin=0, xmax=20,
          ymin=0, ymax=0.016,
          ytick distance=0.004,
          xtick distance=2,
          scaled y ticks=false,
          yticklabel style={
              /pgf/number format/fixed,
              /pgf/number format/precision=3,
              /pgf/number format/fixed zerofill,
          },
          grid style={dashed, gray!30},
          tick label style={font=\footnotesize},
          label style={font=\small},
        ]
        \addplot[black, only marks, mark=o, mark size=2pt, each nth point=80]
            table[col sep=space, skip first n=1, x index=0, y index=2]
            {data/taylorgreen/wang_2013_spectral.txt};
        \addplot[blue, line width=1.5pt, no marks, each nth point=2]
            table[col sep=space, skip first n=1, x index=0, y index=2]
            {data/taylorgreen/monolithic_256X256.txt};
        \addplot[red, line width=1.5pt, dashed, no marks, each nth point=2]
            table[col sep=space, skip first n=1, x index=0, y index=2]
            {data/taylorgreen/monolithic_128X128.txt};
        \addplot[black, line width=1.5pt, densely dotted, no marks, each nth point=2]
            table[col sep=space, skip first n=1, x index=0, y index=2]
            {data/taylorgreen/monolithic_64X64.txt};
      \end{axis}
    \end{tikzpicture}
    \subcaption{Monolithic: $\varepsilon_k(t)$}
    \label{fig:tgv_mono_eps}
  \end{minipage}

  \vspace{0.8em}

  \begin{minipage}{0.49\textwidth}
    \centering
    \begin{tikzpicture}
      \begin{axis}[
          width=0.95\linewidth,
          xlabel={$t$},
          ylabel={Kinetic energy $E_k(t)$},
          xmin=0, xmax=20,
          ymin=0, ymax=0.14,
          ytick distance=0.02,
          xtick distance=2,
          scaled y ticks=false,
          yticklabel style={
              /pgf/number format/fixed,
              /pgf/number format/precision=3,
              /pgf/number format/fixed zerofill,
          },
          grid style={dashed, gray!30},
          tick label style={font=\footnotesize},
          label style={font=\small},
        ]
        \addplot[black, only marks, mark=o, mark size=2pt, each nth point=80]
            table[col sep=space, skip first n=1, x index=0, y index=1]
            {data/taylorgreen/wang_2013_spectral.txt};
        \addplot[blue, line width=1.5pt, no marks, each nth point=2]
            table[col sep=space, skip first n=1, x index=0, y index=1]
            {data/taylorgreen/projection_256X256.txt};
        \addplot[red, line width=1.5pt, dashed, no marks, each nth point=2]
            table[col sep=space, skip first n=1, x index=0, y index=1]
            {data/taylorgreen/projection_128X128.txt};
        \addplot[black, line width=1.5pt, densely dotted, no marks, each nth point=2]
            table[col sep=space, skip first n=1, x index=0, y index=1]
            {data/taylorgreen/projection_64X64.txt};
      \end{axis}
    \end{tikzpicture}
    \subcaption{Projection: $E_k(t)$}
    \label{fig:tgv_proj_Ek}
  \end{minipage}
  \hfill
  \begin{minipage}{0.49\textwidth}
    \centering
    \begin{tikzpicture}
      \begin{axis}[
          width=0.95\linewidth,
          xlabel={$t$},
          ylabel={Dissipation rate $\varepsilon_k(t)$},
          xmin=0, xmax=20,
          ymin=0, ymax=0.016,
          ytick distance=0.004,
          xtick distance=2,
          scaled y ticks=false,
          yticklabel style={
              /pgf/number format/fixed,
              /pgf/number format/precision=3,
              /pgf/number format/fixed zerofill,
          },
          grid style={dashed, gray!30},
          tick label style={font=\footnotesize},
          label style={font=\small},
        ]
        \addplot[black, only marks, mark=o, mark size=2pt, each nth point=80]
            table[col sep=space, skip first n=1, x index=0, y index=2]
            {data/taylorgreen/wang_2013_spectral.txt};
        \addplot[blue, line width=1.5pt, no marks, each nth point=2]
            table[col sep=space, skip first n=1, x index=0, y index=2]
            {data/taylorgreen/projection_256X256.txt};
        \addplot[red, line width=1.5pt, dashed, no marks, each nth point=2]
            table[col sep=space, skip first n=1, x index=0, y index=2]
            {data/taylorgreen/projection_128X128.txt};
        \addplot[black, line width=1.5pt, densely dotted, no marks, each nth point=2]
            table[col sep=space, skip first n=1, x index=0, y index=2]
            {data/taylorgreen/projection_64X64.txt};
      \end{axis}
    \end{tikzpicture}
    \subcaption{Projection: $\varepsilon_k(t)$}
    \label{fig:tgv_proj_eps}
  \end{minipage}

  \vspace{0.4em}
  \centerline{\pgfplotslegendfromname{tgvmeshlegend}}

  \caption{Taylor--Green vortex at $\mathrm{Re}=1600$ on three mesh
  resolutions ($64^3$, $128^3$, $256^3$), compared against the
  pseudo-spectral DNS of \citet{wang2013high} on a $512^3$ grid.
  Top row: monolithic VMS scheme. Bottom row: VMS-stabilized
  projection scheme. Left column: volume-averaged kinetic energy
  $E_k(t)$. Right column: energy-based dissipation rate
  $\varepsilon_k(t)$.}
  \label{fig:tgv_results}
\end{figure}
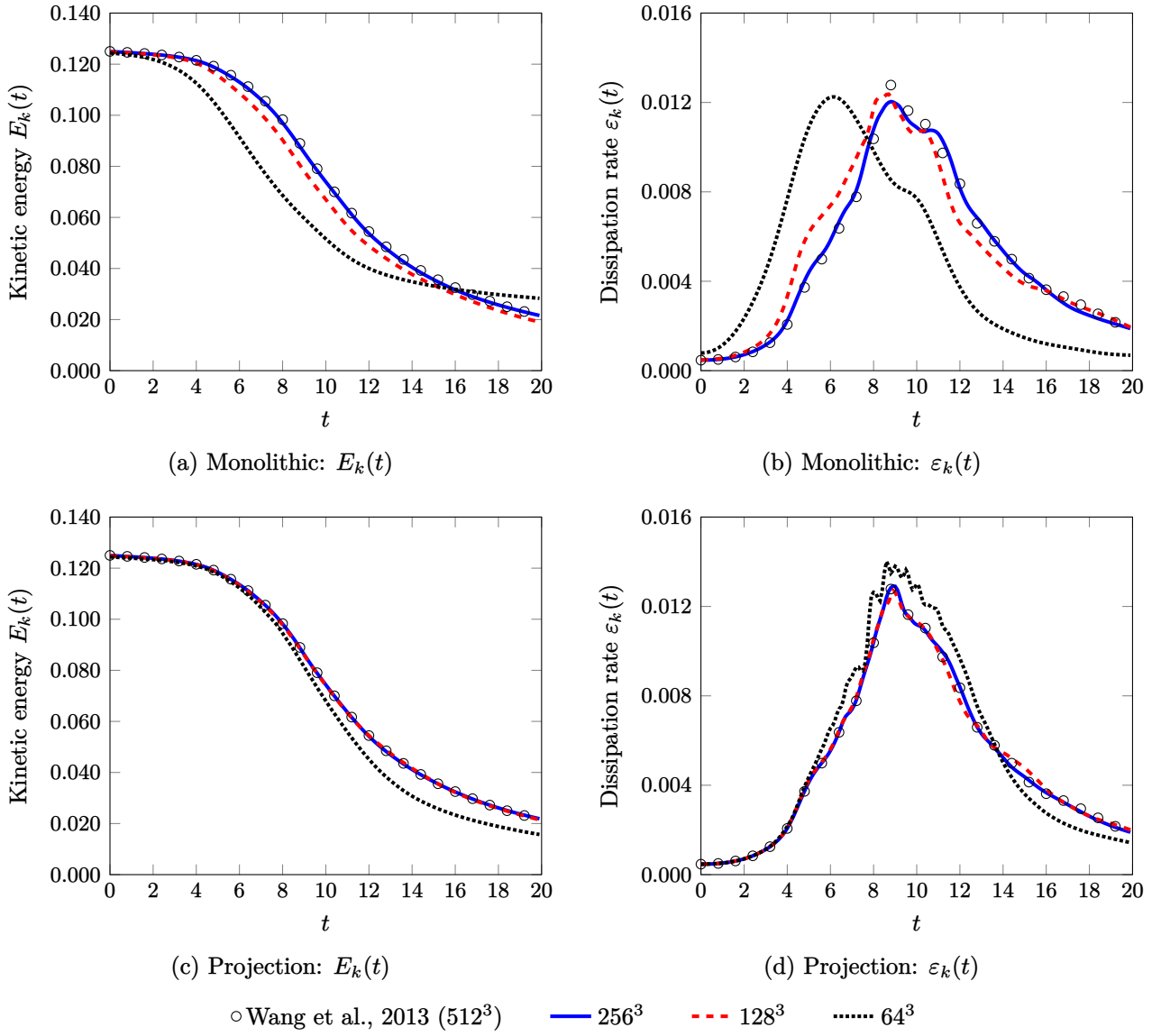

\begin{figure}[!htb]
  \centering
  \begin{tikzpicture}
    \begin{axis}[
        width=0.55\linewidth,
        xmode=log, ymode=log,
        xlabel={$k$},
        ylabel={$E(k)$},
        xmin=2, xmax=160,
        ymin=1e-7, ymax=1e-1,
        xtick={1, 2, 10, 128},
        xticklabels={$1$, $2$, $10$,  $128$},
        minor xtick={2,3,4,5,6,7,8,9,20,30,40,50,60,70,80,90,150},
        grid=none,
        tick label style={font=\footnotesize},
        label style={font=\small},
        legend style={
          at={(0.0,0.0)}, anchor=south west,
          font=\scriptsize, draw=none, fill=none,
        },
      ]
      \addplot[black, only marks, mark=o, mark size=2pt, each nth point=1]
          table[col sep=comma, x=x, y expr=\thisrow{y}/1]{data/taylorgreen/energySpectrumDNSTGV.csv};
      \addlegendentry{Wang 2013}

      \addplot[blue, line width=1.5pt, no marks]
          table[col sep=space, skip first n=1, x index=0, y index=1]
          {data/taylorgreen/spectrum_t8_256.txt};
      \addlegendentry{Projection ($256^3$)}

      \addplot[black, densely dotted, line width=1pt, domain=2:128, samples=30]
          {0.1*x^(-5/3)};
      \addlegendentry{$k^{-5/3}$ (reference)}

      \draw[black, dotted, line width=1pt]
          (axis cs:128,1e-7) -- (axis cs:128,1e-1);
      \node[black, font=\scriptsize, rotate=90, anchor=south]
          at (axis cs:128,3e-3) {Nyquist};
    \end{axis}
  \end{tikzpicture}
  \caption{Three-dimensional kinetic-energy spectrum $E(k)$ at $t=8$ on
  the $256^3$ mesh for the VMS-stabilized projection scheme. The dashed
  line indicates a $k^{-5/3}$ slope for reference. The spectrum is
  compared against the pseudo-spectral DNS of \citet{wang2013high} at $t=8$ on a $512^3$ grid.}
  \label{fig:tgv_spectrum}
\end{figure}
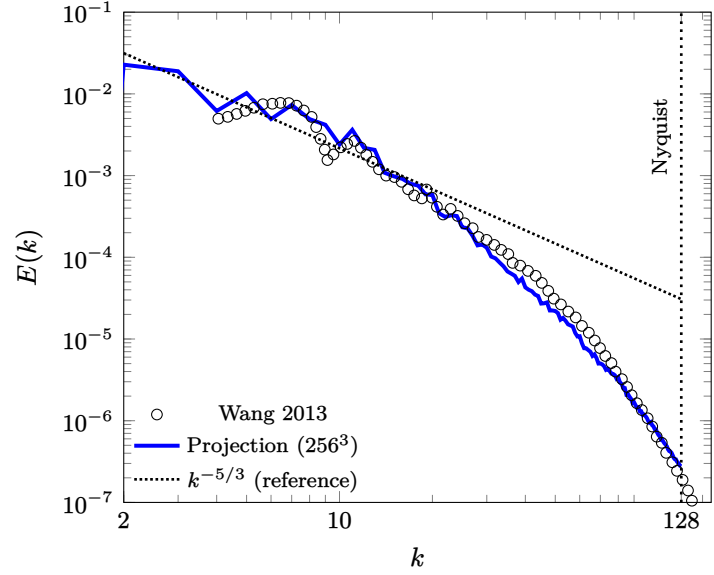



\pgfplotsset{
  tgvdiss/.style={
    width=\linewidth, height=0.9\linewidth,
    xmin=0, xmax=20, ymin=0, ymax=0.016,
    xlabel={$t$},
    xtick distance=5, ytick distance=0.004,
    scaled y ticks=false,
    yticklabel style={
      /pgf/number format/fixed,
      /pgf/number format/precision=3,
      /pgf/number format/fixed zerofill,
    },
    grid style={dashed, gray!30},
    tick label style={font=\scriptsize},
    label style={font=\footnotesize},
  }
}

\begin{figure}[!htb]
  \centering

  \begin{minipage}{0.32\textwidth}\centering
    \begin{tikzpicture}
      \begin{axis}[tgvdiss, ylabel={$\varepsilon(t)$},
            legend style={at={(0.98,0.98)}, anchor=north east,
                          font=\tiny, draw=none, fill=none, row sep=-2pt}]
        \addplot[red!70!black, line width=1pt, no marks, name path=mA]
          table[col sep=space, skip first n=1, x index=0, y index=2]
          {data/taylorgreen/monolithic_64X64.txt};
          \addlegendentry{$\varepsilon_k$}
        \addplot[red!70!black, line width=1pt, dashed, no marks, name path=mB]
          table[col sep=space, skip first n=1, x index=0, y index=3]
          {data/taylorgreen/monolithic_64X64.txt};
          \addlegendentry{$\varepsilon_\Omega$}
        \addplot[red!15, forget plot] fill between[of=mA and mB];
        \addplot[black, only marks, mark=o, mark size=1pt, each nth point=100]
          table[col sep=space, skip first n=1, x index=0, y index=2]
          {data/taylorgreen/wang_2013_spectral.txt};
          \addlegendentry{Wang 2013}
      \end{axis}
    \end{tikzpicture}
    \subcaption{Monolithic, $64^3$}\label{fig:tgv_diss_mono_64}
  \end{minipage}\hfill
  \begin{minipage}{0.32\textwidth}\centering
    \begin{tikzpicture}
      \begin{axis}[tgvdiss, yticklabels={}]
        \addplot[red!70!black, line width=1pt, no marks, name path=mC]
          table[col sep=space, skip first n=1, x index=0, y index=2]
          {data/taylorgreen/monolithic_128X128.txt};
        \addplot[red!70!black, line width=1pt, dashed, no marks, name path=mD]
          table[col sep=space, skip first n=1, x index=0, y index=3]
          {data/taylorgreen/monolithic_128X128.txt};
        \addplot[red!15, forget plot] fill between[of=mC and mD];
        \addplot[black, only marks, mark=o, mark size=1pt, each nth point=100]
          table[col sep=space, skip first n=1, x index=0, y index=2]
          {data/taylorgreen/wang_2013_spectral.txt};
      \end{axis}
    \end{tikzpicture}
    \subcaption{Monolithic, $128^3$}\label{fig:tgv_diss_mono_128}
  \end{minipage}\hfill
  \begin{minipage}{0.32\textwidth}\centering
    \begin{tikzpicture}
      \begin{axis}[tgvdiss, yticklabels={}]
        \addplot[red!70!black, line width=1pt, no marks, name path=mE]
          table[col sep=space, skip first n=1, x index=0, y index=2]
          {data/taylorgreen/monolithic_256X256.txt};
        \addplot[red!70!black, line width=1pt, dashed, no marks, name path=mF]
          table[col sep=space, skip first n=1, x index=0, y index=3]
          {data/taylorgreen/monolithic_256X256.txt};
        \addplot[red!15, forget plot] fill between[of=mE and mF];
        \addplot[black, only marks, mark=o, mark size=1pt, each nth point=100]
          table[col sep=space, skip first n=1, x index=0, y index=2]
          {data/taylorgreen/wang_2013_spectral.txt};
      \end{axis}
    \end{tikzpicture}
    \subcaption{Monolithic, $256^3$}\label{fig:tgv_diss_mono_256}
  \end{minipage}

  \vspace{0.8em}

  \begin{minipage}{0.32\textwidth}\centering
    \begin{tikzpicture}
      \begin{axis}[tgvdiss, ylabel={$\varepsilon(t)$},
            legend style={at={(0.98,0.98)}, anchor=north east,
                          font=\tiny, draw=none, fill=none, row sep=-2pt}]
        \addplot[blue!70!black, line width=1pt, no marks, name path=pA]
          table[col sep=space, skip first n=1, x index=0, y index=2]
          {data/taylorgreen/projection_64X64.txt};
          \addlegendentry{$\varepsilon_k$}
        \addplot[blue!70!black, line width=1pt, dashed, no marks, name path=pB]
          table[col sep=space, skip first n=1, x index=0, y index=3]
          {data/taylorgreen/projection_64X64.txt};
          \addlegendentry{$\varepsilon_\Omega$}
        \addplot[blue!15, forget plot] fill between[of=pA and pB];
        \addplot[black, only marks, mark=o, mark size=1pt, each nth point=100]
          table[col sep=space, skip first n=1, x index=0, y index=2]
          {data/taylorgreen/wang_2013_spectral.txt};
          \addlegendentry{Wang 2013}
      \end{axis}
    \end{tikzpicture}
    \subcaption{Projection, $64^3$}\label{fig:tgv_diss_proj_64}
  \end{minipage}\hfill
  \begin{minipage}{0.32\textwidth}\centering
    \begin{tikzpicture}
      \begin{axis}[tgvdiss, yticklabels={}]
        \addplot[blue!70!black, line width=1pt, no marks, name path=pC]
          table[col sep=space, skip first n=1, x index=0, y index=2]
          {data/taylorgreen/projection_128X128.txt};
        \addplot[blue!70!black, line width=1pt, dashed, no marks, name path=pD]
          table[col sep=space, skip first n=1, x index=0, y index=3]
          {data/taylorgreen/projection_128X128.txt};
        \addplot[blue!15, forget plot] fill between[of=pC and pD];
        \addplot[black, only marks, mark=o, mark size=1pt, each nth point=100]
          table[col sep=space, skip first n=1, x index=0, y index=2]
          {data/taylorgreen/wang_2013_spectral.txt};
      \end{axis}
    \end{tikzpicture}
    \subcaption{Projection, $128^3$}\label{fig:tgv_diss_proj_128}
  \end{minipage}\hfill
  \begin{minipage}{0.32\textwidth}\centering
    \begin{tikzpicture}
      \begin{axis}[tgvdiss, yticklabels={}]
        \addplot[blue!70!black, line width=1pt, no marks, name path=pE]
          table[col sep=space, skip first n=1, x index=0, y index=2]
          {data/taylorgreen/projection_256X256.txt};
        \addplot[blue!70!black, line width=1pt, dashed, no marks, name path=pF]
          table[col sep=space, skip first n=1, x index=0, y index=3]
          {data/taylorgreen/projection_256X256.txt};
        \addplot[blue!15, forget plot] fill between[of=pE and pF];
        \addplot[black, only marks, mark=o, mark size=1pt, each nth point=100]
          table[col sep=space, skip first n=1, x index=0, y index=2]
          {data/taylorgreen/wang_2013_spectral.txt};
      \end{axis}
    \end{tikzpicture}
    \subcaption{Projection, $256^3$}\label{fig:tgv_diss_proj_256}
  \end{minipage}

  \caption{Energy-based ($\varepsilon_k$, solid) and enstrophy-based
  ($\varepsilon_\Omega$, dashed) dissipation rates for the
  Taylor--Green vortex at $\mathrm{Re}=1600$. Top row: monolithic
  VMS solver. Bottom row: VMS/projection solver. Columns
  correspond to the $64^3$, $128^3$, and $256^3$ meshes. The shaded
  region indicates the gap $\varepsilon_k - \varepsilon_\Omega$,
  representing the contribution of unresolved scales (model and
  numerical dissipation) to the total kinetic-energy decay. Open
  circles are the pseudo-spectral DNS reference of
  \citet{wang2013high} on a $512^3$ grid.}
  \label{fig:tgv_dissipation_grid}
\end{figure}
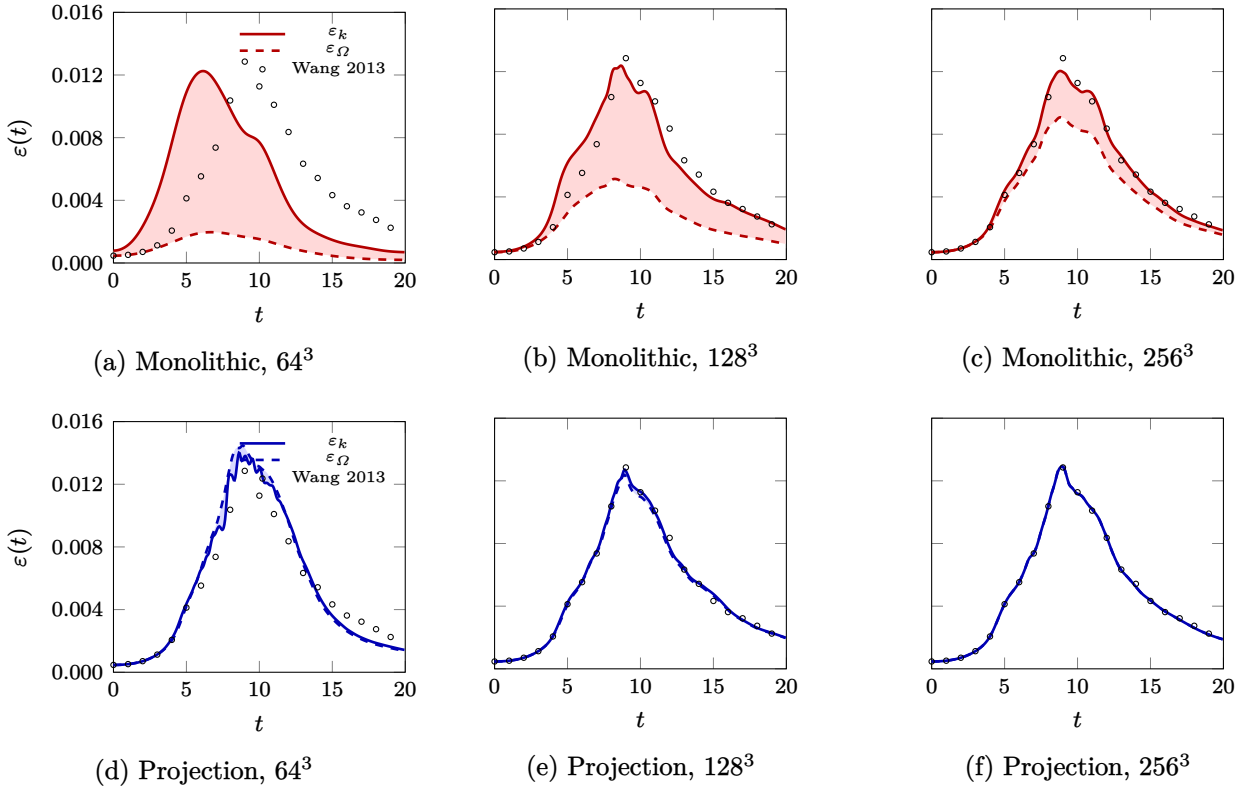

The simulation is run on uniform Cartesian meshes of $64^3$, $128^3$,
and $256^3$ elements with a constant time step of $\Delta t = 0.001$,
integrated to $t = 20$. The proposed VMS-stabilized projection scheme
and the monolithic VMS scheme are both run on the same meshes with
identical Krylov solver settings. The kinetic energy
$E_k(t)$, the energy-based dissipation rate $\varepsilon_k(t)$, and the enstrophy-based dissipation rate
$\varepsilon_\Omega(t)$ are computed at every time step and compared
with the pseudo-spectral DNS reference of \citet{wang2013high} on a
$512^3$ grid.

The flow structure is visualized in Figure~\ref{fig:taylorgreenqcriterion}
using Q-criterion iso-surfaces coloured by velocity magnitude. At $t=0$,
the flow exhibits the organized cellular vortex structure prescribed by
the initial condition. By $t=15$, the initial vortex sheets have broken
down into a dense network of small-scale vortical structures, marking
the developed turbulent regime. At $t=20$, the small-scale structures
persist but have decayed in intensity, consistent with the monotonic
energy decay observed in Figure~\ref{fig:tgv_results}(c).

\figref{fig:tgv_results} presents the kinetic-energy decay $E_k(t)$ and the energy-based dissipation rate $\varepsilon_k(t)$ for both schemes on the three meshes, alongside the spectral reference. For both formulations the curves approach the DNS reference under refinement and track it closely on the finer grids. The kinetic-energy spectrum $E(k)$ at $t=8$ on the $256^3$ mesh (\figref{fig:tgv_spectrum}) follows the $k^{-5/3}$ reference slope and decays to the Nyquist wavenumber without high-wavenumber accumulation, indicating that the VMS closure dissipates small-scale energy at the rate set by the resolved cascade and that the velocity-pressure splitting does not introduce spurious small-scale content.

\paragraph{Solution accuracy: monolithic vs.\ projection.}
The projection scheme tracks the pseudo-spectral DNS of \citet{wang2013high}
closely once the mesh reaches $128^3$, with both the kinetic energy and the
dissipation rate in close agreement with the reference
(\figref{fig:tgv_results}(c,d)). The monolithic VMS formulation, run on the
same meshes with identical time step and Krylov solver settings, also
converges toward DNS, but more slowly: at $64^3$ and $128^3$ its kinetic
energy decays faster than the reference and lies below the DNS curve through
the mid-range, and only at $256^3$ does $E_k(t)$ track DNS closely
(\figref{fig:tgv_results}(a)). It also continues to over-dissipate
relative to the reference even at $256^3$.

The mechanism is exposed by the gap between the energy-based dissipation
$\varepsilon_k(t) = -dE_k/dt$ and the enstrophy-based dissipation
$\varepsilon_\Omega(t) = (2/\mathrm{Re})\,\mathcal{E}(t)$, shown as the shaded
region in Figure~\ref{fig:tgv_dissipation_grid}. This gap measures the
dissipation supplied by the modeled subgrid stresses rather than by viscosity
acting on the resolved vorticity. For the projection scheme (bottom row) the
gap is visible at $64^3$ but closes almost entirely by $128^3$ and is
imperceptible at $256^3$, so the resolved enstrophy accounts for essentially
all of the energy decay. For the monolithic scheme (top row) the gap is wide at
$64^3$ and narrows under refinement, but remains clearly open at $256^3$, where
a substantial share of the dissipation is still supplied by the model rather
than the resolved vorticity.

Like in the previous example, we attribute this difference to the ``grad-div''-like term present in the monolithic VMS scheme (see \eqref{eq:vms-finescale-model-overview-c} and \eqref{eq:vms-coarse-residuals-overview-c}), which dissipates kinetic energy through the divergence error and damps the velocity gradients that drive vortex stretching. The projection scheme does not have this term, since the pressure is not decomposed into multiple scales. The results across the three meshes indicate that this behavior is structural rather than a resolution artifact: the monolithic dissipation gap does not close under refinement, so even at $256^3$---where $E_k(t)$ matches DNS---a large part of the energy decay is carried by the modeled term rather than the resolved enstrophy. The projection scheme, by contrast, matches DNS with an essentially closed gap already at $128^3$. The ablation in Figure~\ref{fig:tgv_tauC_effect} confirms this attribution directly. Setting $\pprime=0$ in the monolithic scheme brings a good match with the DNS reference of \citet{wang2013high} (Figure~\ref{fig:tgv_Ek_tauC}). However, the downside is that the divergence can suffer greatly. Figure~\ref{fig:tgv_divnorm_tauC} plots the $L^2$-norm of the divergence of the velocity ($\uvech$ for monolithic, $\uvecSh$ for projection) with time. The curve for projection is very close to that of monolithic with $\pprime=0$, whereas the standard monolithic solver yields the lowest divergence. The grad-div term suppresses the divergence, and is also responsible for the excess dissipation of the kinetic energy.

\begin{figure}[!htb]
  \centering
  \begin{minipage}{0.49\textwidth}
    \centering
    \begin{tikzpicture}
      \begin{axis}[
          width=0.95\linewidth, height=0.7\linewidth,
          xmin=0, xmax=20,
          ymin=0, ymax=0.14,
          xlabel={$t$},
          ylabel={Kinetic energy $E_k(t)$},
          xtick distance=5,
          ytick distance=0.02,
          scaled y ticks=false,
          yticklabel style={
              /pgf/number format/fixed,
              /pgf/number format/precision=3,
              /pgf/number format/fixed zerofill,
          },
          grid=both,
          grid style={dashed, gray!30},
          tick label style={font=\scriptsize},
          label style={font=\footnotesize},
          legend style={at={(0.0,0.0)}, anchor=south west,
                        font=\small, draw=none, fill=none},
          legend cell align={left},
        ]

        \addplot[red!70!black, line width=1pt, no marks, each nth point=2]
          table[col sep=space, skip first n=1, x index=0, y index=1]
          {data/taylorgreen/monolithic_128X128.txt};
          \addlegendentry{Monolithic}
        \addplot[blue!70!black, line width=1pt, dashed, no marks, each nth point=2]
          table[col sep=space, skip first n=1, x index=0, y index=1]
          {data/taylorgreen/monolothic_128X128_WithoutTauC.txt};
          \addlegendentry{Monolithic, $\pprime =0$}
        \addplot[black, only marks, mark=o, mark size=2pt, each nth point=80]
        table[col sep=space, skip first n=1, x index=0, y index=1]
        {data/taylorgreen/wang_2013_spectral.txt};
        \addlegendentry{Wang et al., 2013 ($512^3$)}
      \end{axis}
    \end{tikzpicture}
    \subcaption{Kinetic energy $E_k(t)$}
    \label{fig:tgv_Ek_tauC}
  \end{minipage}
  \hfill
  \begin{minipage}{0.49\textwidth}
    \centering
    \begin{tikzpicture}
      \begin{axis}[
          width=0.95\linewidth, height=0.7\linewidth,
          xmin=0, xmax=20, ymin=0,
          xlabel={$t$},
          ylabel={$\|\divergence\uvec\|_{L^2}$},
          xtick distance=5,
          grid=both,
          grid style={dashed, gray!30},
          tick label style={font=\scriptsize},
          label style={font=\footnotesize},
          legend style={at={(0.0,0.8)}, anchor=north west,
                        font=\small, draw=none, fill=none},
          legend cell align={left},
        ]

            \addplot[blue, line width=1pt, no marks, each nth point=20]
          table[col sep=comma, skip first n=1, x index=0, y index=10]
          {data/taylorgreen/Projection_integral_values.txt};
          \addlegendentry{Projection}

        \addplot[red!70!black, line width=1pt, no marks, each nth point=20]
          table[col sep=comma, skip first n=1, x index=0, y index=10]
          {data/taylorgreen/MonolothicWithTauC_integral_values.txt};
          \addlegendentry{Monolithic}
        \addplot[blue!70!black, line width=1pt, dashed, no marks, each nth point=20]
          table[col sep=comma, skip first n=1, x index=0, y index=10]
          {data/taylorgreen/MonolothicWithoutTauC_integral_values.txt};
          \addlegendentry{Monolithic, $\pprime =0$}
      \end{axis}
    \end{tikzpicture}
    \subcaption{Velocity divergence $\|\divergence\uvec\|_{L^2}$}
    \label{fig:tgv_divnorm_tauC}
  \end{minipage}

  \caption{Effect of the continuity (grad--div) stabilization on the
  Taylor--Green vortex at $\mathrm{Re}=1600$ (monolithic VMS scheme,
  $128^3$). Left: volume-averaged kinetic energy $E_k(t)$, both with and
  without the $\tau_C$ term. Right: $L^2$ norm of the divergence of
  velocity over time.}
  \label{fig:tgv_tauC_effect}
\end{figure}
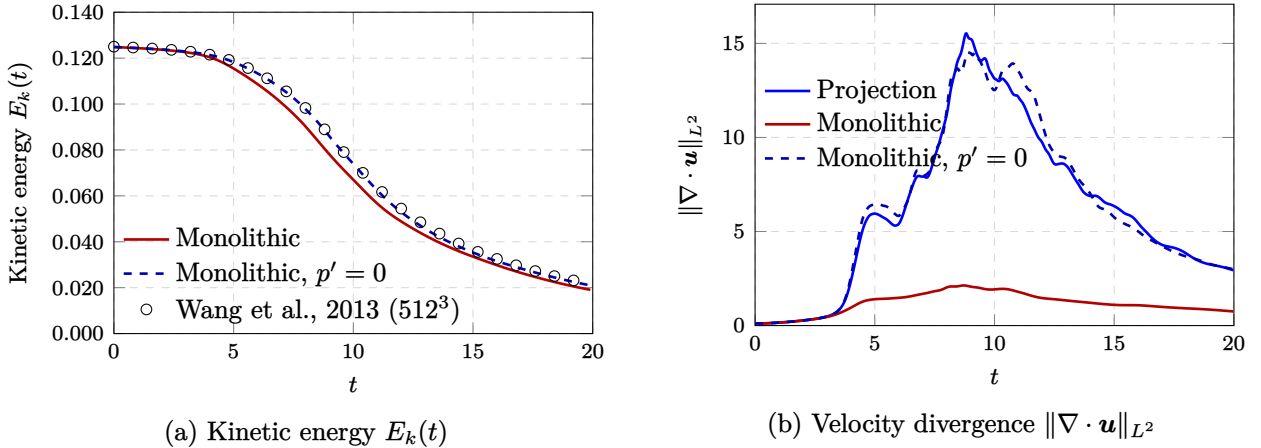

\begin{figure}[!htb]
	\centering
	\begin{tikzpicture}
		\begin{axis}[
			ybar,
			width=0.75\linewidth,
			height=0.5\linewidth,
			symbolic x coords={mesh64,mesh128,mesh256},
			xtick=data,
			xticklabels={$64^3$,$128^3$,$256^3$},
			xticklabel style={font=\small},
			ymin=0,
			ymax=5.0,
			ylabel={$T^w_{\textrm{solve}}$ (s)},
			xlabel={Mesh resolution},
			bar width=14pt,
			enlarge x limits=0.25,
			grid=both,
			grid style={dotted},
			nodes near coords,
			nodes near coords align={vertical},
			every node near coord/.append style={font=\footnotesize},
			legend style={
				at={(0.02,0.98)},
				anchor=north west,
				font=\small,
				draw=black,
				fill=white,
				cells={anchor=west},
			},
			legend cell align={left},
			]
			\addplot+[fill=black!60, draw=black,
			every node near coord/.append style={text=black}] coordinates {
				(mesh64,  1.059)
				(mesh128, 2.299)
				(mesh256, 4.060)
			};
			\addplot+[fill=black!20, draw=black,
			every node near coord/.append style={text=black}] coordinates {
				(mesh64,  0.816)
				(mesh128, 1.366)
				(mesh256, 1.523)
			};
			\legend{Monolithic, Projection}
		\end{axis}
	\end{tikzpicture}
	\caption{Average wall-clock solve time per time step for the
		Taylor-Green vortex at $Re=1600$ on three mesh resolutions,
		measured over $20{,}000$ time steps. Both solvers use BiCGStab
		with additive Schwarz preconditioning ($\textrm{rtol}=10^{-8}$).}
	\label{fig:tgv-solve-times}
\end{figure}
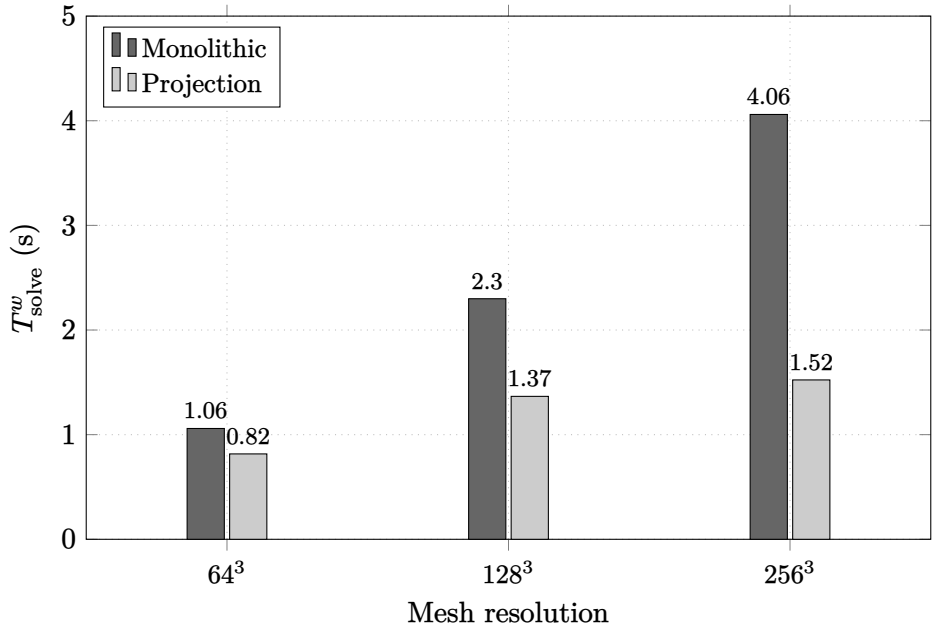

\paragraph{Computational cost.}
Figure~\ref{fig:tgv-solve-times} reports the average wall-clock solve time per
time step for both schemes on the three meshes, measured over $20{,}000$ steps
under identical BiCGStab and additive-Schwarz settings. The projection scheme is
cheaper than the monolithic scheme at every resolution tested. On the $128^3$
and $256^3$ meshes, which use the same number of elements per processor, the
projection solve takes $1.37$ and $1.52$~s per step against $2.30$ and $4.06$~s
for the monolithic scheme.

\section{Conclusion}\label{sec:conclusions}
In this work, we presented a VMS-stabilized projection method based on the Chorin-Temam family of fractional-step schemes. The method replaces the coupled velocity-pressure saddle-point system with three subproblems: (i) a stabilized nonlinear advection--diffusion--reaction problem for an intermediate velocity, (ii) a Poisson equation for the pressure, and (iii) a Helmholtz-Leray projection of the intermediate velocity onto a weakly divergence-free subspace. Because the Galerkin discretization of the momentum predictor can become unstable in advection-dominated regimes, we applied a multiscale decomposition to the predicted velocity. However, unlike in the monolithic VMS formulation, the pressure is not decomposed into coarse and fine scales. The resulting formulation introduces SUPG-like stabilization in the momentum predictor and a PSPG-like residual contribution in the pressure Poisson equation. We also provided a formal error analysis that separates the time-discretization, projection-splitting, and spatial-VMS errors, recovering second-order temporal accuracy for the velocity with BDF2 and the established first-order pressure accuracy of the non-rotational pressure-correction scheme.

We tested the method on four benchmark problems of increasing difficulty. The method of manufactured solutions confirms the predicted second-order temporal convergence of the velocity for $Re \in \{10^3,10^4,10^5,10^6\}$. For the lid-driven cavity at $Re \in \{100,1000,5000,10000\}$, the computed velocity profiles agree closely with both the monolithic VMS results and the reference data of \citet{ghia1982high}. The flow-past-a-cylinder results demonstrate the treatment of outlet boundary conditions and unstructured meshes, producing drag coefficients and Strouhal numbers in close agreement with published values \citep{kanaris2011three}. Finally, for the Taylor-Green vortex at $Re=1600$, the projection scheme approaches the pseudo-spectral DNS of \citet{wang2013high} under mesh refinement and agrees closely with it on the finer resolutions.

Comparisons with the monolithic VMS formulation reveal a consistent effect of the pressure fine scale, which is retained in the monolithic scheme but omitted from the proposed projection scheme. In the cylinder problem, the monolithic formulation overpredicts the drag coefficient by approximately $25\%$ relative to the literature values and the projection result on the same mesh; this discrepancy decreases under mesh refinement and when the pressure fine scale is suppressed. In the Taylor-Green vortex, the monolithic formulation exhibits greater modeled dissipation than the projection scheme, and this excess persists even on the finest mesh considered. Setting the pressure fine scale to zero in the monolithic formulation brings both its cylinder drag and Taylor-Green kinetic-energy predictions closer to the corresponding references, supporting the attribution of these differences to the associated grad--div-like term. This improvement comes with a tradeoff: suppressing the pressure fine scale increases the velocity-divergence error to approximately the level observed in the projection scheme. Finally, the projection formulation was also less expensive in the 3D Taylor-Green tests, reducing the average solution time per step by factors ranging from approximately $1.3$ on the $64^3$ mesh to $2.7$ on the $256^3$ mesh under identical solver and preconditioner settings.

Several directions remain open. A rigorous energy-stability analysis of the fully stabilized scheme remains to be completed; in that setting a skew-symmetric (Temam) convection form is expected to improve the stability behaviour in under-resolved regimes, at no added cost. The first-order pressure rate and the associated near-boundary accuracy loss are inherited from the non-rotational pressure-correction step and could be improved to $O(\dt^{3/2})$ with a rotational variant. The elliptic sub-solves are also naturally amenable to algebraic multigrid,
making the scheme an attractive backbone for large-scale, GPU-deployable solvers.

\section*{Acknowledgments}
This work was partly supported by the National Science Foundation under the grants NSF LEAP-HI 2053760, NSF  Advanced Simulation of Multiphase Electrochemical Systems (Award No. 2513871) and by the AI
Research Institutes program through NSF and USDA–NIFA under the AI Institute for Resilient Agriculture (Award No. 2021-67021-35329).

\bibliographystyle{plainnat}
\bibliography{references} 
\clearpage
\appendix

\section{Element metric quantities for the VMS parameters}
\label{app:vms-parameters}
Let $\boldsymbol{x}=\boldsymbol{x}(\boldsymbol{\xi})$ denote the map from the reference element, with coordinates $\boldsymbol{\xi}$, to a physical element, with coordinates $\boldsymbol{x}$. Define the inverse-Jacobian components by
\begin{align}
	K_{ij}=\frac{\partial\xi_i}{\partial x_j}.
\end{align}
The element metric tensor $\mathbf{G}$ and vector $\boldsymbol{g}$ used in \eqref{eq:vms-tau-overview} are computed as
\begin{subequations}
	\label{eq:vms-element-metrics}
	\begin{align}
		G_{ij} &= \sum_{k=1}^{n_{\mathrm{sd}}}K_{ki}K_{kj},
		\label{eq:vms-element-metric-G}\\
		g_i &= \sum_{j=1}^{n_{\mathrm{sd}}}K_{ji},
		\label{eq:vms-element-metric-g}
	\end{align}
\end{subequations}
where $n_{\mathrm{sd}}$ is the number of spatial dimensions. Consequently,
\begin{align}
	\mathbf{G}:\mathbf{G}
	&=\sum_{i,j=1}^{n_{\mathrm{sd}}}G_{ij}^2,
	&
	\boldsymbol{g}\cdot\boldsymbol{g}
	&=\sum_{i=1}^{n_{\mathrm{sd}}}g_i^2.
\end{align}
Thus, on each element, the implemented stabilization parameters are
\begin{align}
	\taum
	&=\left[
		\frac{4}{\dt^2}
		+\sum_{i,j=1}^{n_{\mathrm{sd}}}(u_h)_iG_{ij}(u_h)_j
		+\cinv\visco^2\sum_{i,j=1}^{n_{\mathrm{sd}}}G_{ij}^2
	\right]^{-1/2},
	&
	\tauc
	&=\left(\taum\sum_{i=1}^{n_{\mathrm{sd}}}g_i^2\right)^{-1}.
\end{align}
For the projection formulation, $(u_h)_i$ in the expression for $\taum$ is replaced by $(\widetilde{u}_h)_i$.

\end{document}